\documentclass[12pt, reqno]{amsart}

\usepackage{
amsmath,
amssymb,
amsthm,
mathrsfs,
amsfonts,
ytableau,
enumitem,
comment,
upgreek,
soul,
yhmath,
mathtools,
shuffle,
kotex,
array,
caption,
tabularx
}

\usepackage[euler]{textgreek}

\usepackage{tikz,tikz-cd}

\usepackage[colorinlistoftodos, textwidth = 2.3cm]{todonotes}

\ytableausetup{mathmode, boxsize=1.2em}

\newtheorem{theorem}{Theorem}[section]
\newtheorem{proposition}[theorem]{Proposition}
\newtheorem{lemma}[theorem]{Lemma}
\newtheorem{corollary}[theorem]{Corollary}

\newtheorem*{claim*}{Claim}

\theoremstyle{definition}

\newtheorem{example}[theorem]{Example}
\newtheorem{definition}[theorem]{Definition}
\newtheorem{remark}[theorem]{Remark}

\numberwithin{equation}{section} \numberwithin{figure}{section}
\numberwithin{table}{section}

\def\Z{\mathbb Z}
\def\C{\mathbb C}

\newcommand{\nc}{\newcommand}

\nc{\SG}{\mathfrak{S}}
\nc{\frakR}{\mathfrak{R}}
\nc{\frakL}{\mathfrak{L}}
\nc{\PCT}{\mathrm{PCT}}
\nc{\SPCT}{\mathrm{SPCT}}
\nc{\RT}{\mathrm{RT}}
\nc{\SRT}{\mathrm{SRT}}
\nc{\RCT}{\mathrm{RCT}}
\nc{\SRCT}{\mathrm{SRCT}}
\nc{\SYRT}{\mathrm{SYRT}}
\nc{\SYCT}{\mathrm{SYCT}}
\nc{\SPYCT}{\mathrm{SPYCT}}
\nc{\tst}{\mathtt{st}}
\nc{\Span}{\mathrm{span}}
\nc{\comp}{\mathrm{comp}}
\nc{\rmst}{\mathrm{st}}
\nc{\Des}{\mathrm{Des}}
\nc{\set}{\mathrm{set}}
\nc{\wt}{\mathrm{wt}}
\nc{\ch}{\mathrm{ch}}
\nc{\id}{\mathrm{id}}
\nc{\Sym}{\mathrm{Sym}}
\nc{\Qsym}{\mathrm{QSym}}
\nc{\Nsym}{\mathrm{NSym}}
\nc{\sh}{\mathrm{sh}}
\nc{\bfS}{\mathbf{S}}
\nc{\bfm}{\mathbf{m}}
\nc{\hbfS}{\widehat{\mathbf{S}}}
\nc{\bfF}{\mathbf{F}}
\nc{\calB}{\mathcal{B}}
\nc{\calS}{\mathcal{S}}
\nc{\calRS}{\mathcal{RS}}
\nc{\hcalS}{\widehat{\mathcal{S}}}
\nc{\alphamax}{\alpha_{\rm max}}
\nc{\brho}{\overline{\rho}}
\nc{\bphi}{\overline{\phi}}
\nc{\calV}{\mathcal{V}}
\nc{\calR}{\mathcal{R}}
\nc{\sfR}{\mathsf{R}}
\nc{\calG}{\mathcal{G}}
\nc{\tal}{\lambda(\alpha)}
\nc{\tbe}{\widetilde{\beta}}
\nc{\opi}{\overline{\pi}}
\nc{\calP}{\mathcal{P}}
\nc{\rmtop}{\mathrm{top}}
\nc{\rad}{\mathrm{rad}}
\nc{\bfP}{\mathbf{P}}
\nc{\SET}{\mathrm{SET}}
\nc{\SIT}{\mathrm{SIT}}
\nc{\rev}{\mathrm{r}}
\nc{\Th}{\theta}
\nc{\mPhi}{\Phi}
\nc{\mphi}{\phi}
\nc{\mPsi}{\Psi}
\nc{\hmPsi}{\widehat{\Psi}}
\nc{\mpsi}{\psi}
\nc{\mGam}{\Gamma}
\nc{\tcd}{\mathtt{cd}}
\nc{\trd}{\mathtt{rd}}
\nc{\trcd}{\mathtt{rcd}}
\nc{\rmr}{\mathrm{r}}
\nc{\rmc}{\mathrm{c}}
\nc{\rmt}{\mathrm{t}}

\nc{\col}{\mathrm{col}}
\nc{\row}{\mathrm{row}}
\nc{\calE}{\mathcal{E}}
\nc{\calT}{\mathscr{T}}
\nc{\sfT}{\mathsf{T}}
\nc{\calEsa}{\mathcal{E}^\upsig(\alpha)}
\nc{\tauC}{\tau_{\scalebox{0.5}{$C$}}}
\nc{\sytabC}{\sytab_{\scalebox{0.5}{$C$}}}
\nc{\bbfP}{\overline{\bfP}}
\nc{\pr}{\mathsf{pr}}
\nc{\Ups}{\Upsilon}
\nc{\pact}{\diamond}
\nc{\tauE}{\tau_{E}^{~}}
\nc{\tauF}{\tau_{\scalebox{0.5}{$F$}}}
\nc{\tauG}{\tau_{\scalebox{0.5}{$G$}}}
\nc{\rtE}{T_{\scalebox{0.5}{$E$}}}
\nc{\rtF}{T_{\scalebox{0.5}{$F$}}}
\nc{\rtG}{T_{\scalebox{0.5}{$G$}}}
\nc{\oPaE}{\overline{\Phi}_{\alpha_E}}
\nc{\oPaF}{\overline{\Phi}_{\alpha_F}}
\nc{\oPaG}{\overline{\Phi}_{\alpha_G}}
\nc{\tab}{\tau}
\nc{\sytab}{\widehat{\tau}}
\nc{\hatE}{\widehat{E}}
\nc{\hati}{\hat{i}}
\nc{\hcalE}{\widehat{\calE}}
\nc{\hatC}{\widehat{C}}
\nc{\bal}{{\boldsymbol{\upalpha}}}
\nc{\bbe}{{\boldsymbol{\upbeta}}}
\nc{\SPYRT}{\mathrm{SPYRT}}

\nc{\bgam}{{\boldsymbol{\upgamma}}}
\nc{\bdel}{{\boldsymbol{\updelta}}}
\nc{\weakcon}{\odot}
\nc{\basisI}{I}

\nc{\ldalpha}{\lambda(\alpha)}
\nc{\SRIT}{\mathrm{SRIT}}
\nc{\re}{\mathrm{rev}}
\nc{\otau}{\overline{\tau}}
\nc{\rtop}{{\rm top}}
\nc{\sfc}{\mathsf{c}}
\nc{\sfr}{\mathsf{r}}
\nc{\tH}{\mathtt{H}}
\nc{\tV}{\mathtt{V}}
\nc{\rpi}{\mathring{\pi}}
\nc{\cpi}{\check{\pi}}
\nc{\frakm}{\mathfrak{m}}
\nc{\sfem}{\mathsf{em}}
\nc{\Hom}{\mathrm{Hom}}
\nc{\module}{\mathrm{mod} \, }
\nc{\rmread}{\mathsf{read}}
\nc{\tread}{\underline{\mathsf{read}}}
\nc{\ocalE}{\overline{\calE}}
\nc{\oE}{\overline{E}}

\nc{\SPCTsa}{\SPCT^\upsig(\alpha)}
\nc{\bfSsa}{\bfS_\alpha^\upsig}
\nc{\bfSsaC}{{\bfS}^\upsig_{\alpha,C}}
\nc{\hbfSsa}{\widehat{\bfS}_\alpha^\upsig}
\nc{\upineq}{\rotatebox{90}{$<$}}
\nc{\downineq}{\rotatebox{270}{$<$}}
\nc{\diagineq}{\rotatebox{135}{$<$}}
\nc{\sfB}{\mathsf{B}}
\nc{\hxi}{\widehat{\xi}}
\nc{\hxidwJ}{\hxi_{\scalebox{0.55}{$J$}}}
\nc{\hxiupJ}{\hxi^{\scalebox{0.55}{$J$}}}
\nc{\scrS}{\mathscr{S}}
\nc{\bfT}{\mathbf{T}}
\nc{\tshuffle}{\,\widetilde{\shuffle}\,}
\nc{\sJ}{\scalebox{0.55}{$J$}}
\nc{\sJo}{\scalebox{0.55}{$J_1$}}
\nc{\sJt}{\scalebox{0.55}{$J_2$}}

\nc{\ra}{\rightarrow}

\nc{\matr}[2]{\left( \hspace{-1ex} \begin{array}{c} #1 \\ #2 \end{array} \hspace{-1ex} \right)}

\newcommand{\tre}{\textcolor{red}}
\definecolor{wsgreen}{rgb}{0,0.5,0}

\nc{\DIRT}{\mathrm{DIRT}}
\nc{\hpi}{\pi}
\nc{\frakI}{\mathfrak{I}}
\nc{\hfrakI}{\widehat{\mathfrak{I}}}
\nc{\orho}{\overline{\rho}}
\nc{\autotheta}{\uptheta}
\nc{\calW}{\mathcal{W}}

\nc{\autophi}{\upphi}
\nc{\autochi}{\upchi}
\nc{\autoomega}{\upomega}
\nc{\hIM}{\widehat{\sfB}}
\nc{\bfpi}{\boldsymbol{\uppi}}
\nc{\bfopi}{\overline{\boldsymbol{\uppi}}}
\nc{\osfB}{\overline{\sfB}}
\nc{\rmw}{\mathrm{w}}
\nc{\conc}{\; {\bullet} \;}
\nc{\ostar}{\; \overline{\bullet} \;}
\nc{\rank}{\mathrm{rank}}
\nc{\fkp}{\mathfrak{p}}
\nc{\bfR}{\mathbf{R}}
\nc{\calD}{\overline{\mathrm{Des}}}
\nc{\tPhi}{\widetilde{\mPhi}}

\nc{\upsig}{{\boldsymbol{\upsigma}}}
\nc{\bfSsaE}{{\bfS}^\upsig_{\alpha,E}}

\nc{\hfkp}{\widehat{\mathfrak{p}}}

\nc{\hautophi}{{\widehat{\autophi}}}
\nc{\hautotheta}{{\widehat{\autotheta}}}
\nc{\hautoomega}{{\widehat{\autoomega}}}
\nc{\rmperm}{\mathrm{perm}}

\nc{\bfsigJ}{\sigma_{\scalebox{0.55}{$J$}}}
\nc{\bfrhoJ}{\rho^{\scalebox{0.55}{$J$}}}

\nc{\teta}{\widetilde{\eta}}
\nc{\urmw}{\underline{\mathrm{w}}}

\newcommand*\sq{\mathbin{\vcenter{\hbox{\rule{.5ex}{.5ex}}}}}

\nc{\pistar}[1]{\pi_{#1}^*}
\nc{\wfkp}{\widetilde{\mathfrak{p}}}
\nc{\bfpsi}{\boldsymbol{\uppsi}}

\nc{\yt}[1]{\todo[size=\tiny,color=blue!10]{#1 \\ \hfill --- Young-Tak}}
\nc{\YT}[1]{\todo[size=\tiny,inline,color=blue!10]{#1
		\\ \hfill --- Young-Tak}}
		
\nc{\yh}[1]{\todo[size=\tiny,color=cyan!10]{#1 \\ \hfill --- Young-Hun}}
\nc{\YH}[1]{\todo[size=\tiny,inline,color=cyan!10]{#1
		\\ \hfill --- Young-Hun}}
		
\nc{\sy}[1]{\todo[size=\tiny,color=magenta!10]{#1 \\ \hfill --- So-Yeon}}
\nc{\SY}[1]{\todo[size=\tiny,inline,color=magenta!10]{#1
		\\ \hfill --- So-Yeon}}
		
\nc{\ws}[1]{\todo[size=\tiny,color=green!10]{#1 \\ \hfill ---  Woo-Seok}}
\nc{\WS}[1]{\todo[size=\tiny,inline,color=green!10]{#1
		\\ \hfill --- Woo-Seok}}

\definecolor{purple}{rgb}{0.44, 0.0, 1.0}

\definecolor{yhblue}{rgb}{0,0,0.6}

\newenvironment{red}{\relax\color{red}}{\hspace*{.5ex}\relax}
\newenvironment{blue}{\relax\color{yhblue}}{\hspace*{.5ex}\relax}
\newenvironment{green}{\relax\color{wsgreen}}{\hspace*{.5ex}\relax}
\newenvironment{magenta}{\relax\color{magenta}}{\hspace*{.5ex}\relax}
\newenvironment{purple}{\relax\color{purple}}{\hspace*{.5ex}\relax}

\nc{\ber}{\begin{red}}
\nc{\er}{\end{red}}
\nc{\beb}{\begin{blue}}
\nc{\eb}{\end{blue}}
\nc{\bema}{\begin{magenta}}
\nc{\ema}{\end{magenta}}
\nc{\begr}{\begin{green}}
\nc{\egr}{\end{green}}

\nc{\bepu}{\begin{purple}}
\nc{\epu}{\end{purple}}

\title[Weak Bruhat interval modules of the 0-Hecke algebra]{Weak Bruhat interval modules of the 0-Hecke algebra}

\author[W.-S. Jung]{Woo-Seok Jung}
\address{Department of Mathematics, Sogang University, Seoul 04107, Republic of Korea}
\email{jungws@sogang.ac.kr}

\author[Y.-H. Kim]{Young-Hun Kim}
\address{Department of Mathematics, Sogang University, Seoul 04107, Republic of Korea \& 
Research Institute for Basic Science, Sogang University, Seoul 04107, Republic of Korea \&
Department of Mathematics, Ewha Womans University, Seoul 03760, Republic of Korea}
\email{ykim.math@gmail.com}

\author[S.-Y. Lee]{So-Yeon Lee}
\address{Department of Mathematics, Sogang University, Seoul 04107, Republic of Korea}
\email{sylee0814@sogang.ac.kr}

\author[Y.-T. Oh]{Young-Tak Oh}
\address{Department of Mathematics, Sogang University, Seoul 04107, Republic of Korea}
\email{ytoh@sogang.ac.kr}

\thanks{All authors were supported by the National Research Foundation of Korea (NRF) grant funded by the Korean Government (NRF-2020R1F1A1A01071055).
The second author was also supported by NRF grant funded by the Korean Government (NRF-2019R1A2C4069647).}


\date{\today}
\begin{document}
\maketitle

\begin{abstract}
The purpose of this paper is to provide a unified method for dealing with various 0-Hecke modules constructed using tableaux so far.
To do this, we assign a $0$-Hecke module to each left weak Bruhat interval, called a weak Bruhat interval module.
We prove that every indecomposable summand of the $0$-Hecke modules categorifying dual immaculate quasisymmetric functions, extended Schur functions, quasisymmetric Schur functions, and Young row-strict quasisymmetric Schur functions is a weak Bruhat interval module. 
We further study embedding into the regular representation, induction product, restriction, and (anti-)involution twists of 
weak Bruhat interval modules. 
\end{abstract}

\section{Introduction}

The $0$-Hecke algebra $H_n(0)$ is a degenerate Hecke algebra obtained from the generic Hecke algebra $H_n(q)$ by specializing $q$ to $0$. 
The representation theory of $H_n(0)$ is very complicated, as can be inferred from the fact that it is not representation-finite for $n > 3$ (see~\cite{11DY, 02DHT}).
Nevertheless, it has attracted the attention of many mathematicians because of its close connection with quasi-symmetric functions.
This link was discovered by Duchamp, Krob, Leclerc, and Thibon~\cite{96DKLT}, who constructed an isomorphism called the \emph{quasisymmetric characteristic} between the Grothendieck ring associated to $0$-Hecke algebras and the ring $\Qsym$ of quasisymmetric functions.
In particular, since the mid-2010s, there have been many attempts to construct $H_n(0)$-modules categorifying important quasisymmetric functions using tableau models, rather than simply adding irreducible modules (for instance, see~\cite{20BS, 15BBSSZ, 19Searles, 15TW,19TW}).

The purpose of the present paper is to provide a method to treat these modules in a uniform manner. 
We start with the observation that every indecomposable direct summand of these modules has a basis isomorphic to a left weak Bruhat interval of $\SG_n$ when it is equipped with the partial order $\preceq$ defined by 
\begin{align*}
T \preceq T' \quad \text{if $\pi_\sigma \cdot T = T'$ for some $\sigma \in \SG_n$}.
\end{align*}

This leads us to consider the $H_n(0)$-module $\sfB(\sigma,\rho)$ for each weak Bruhat interval $[\sigma,\rho]_L$, called the \emph{weak Bruhat interval module associated to  $[\sigma,\rho]_L$}, whose underlying space is the $\C$-span of $[\sigma,\rho]_L$ and whose action is given by
\begin{equation*}
\pi_i \cdot \gamma  := \begin{cases}
\gamma & \text{if $i \in \Des_L(\gamma)$}, \\
0 & \text{if $i \notin \Des_L(\gamma)$ and $s_i\gamma \notin [\sigma,\rho]_L$,} \\
s_i \gamma & \text{if $i \notin \Des_L(\gamma)$ and $s_i\gamma \in [\sigma,\rho]_L$.}
\end{cases} 
\end{equation*}
In a similar point of view, we also consider the $H_n(0)$-module $\osfB(\sigma,\rho)$ for each weak Bruhat interval $[\sigma,\rho]_L$, called the \emph{negative weak Bruhat interval module associated to  $[\sigma,\rho]_L$}, whose underlying space is the $\C$-span of $[\sigma,\rho]_L$ and whose action is given by
\begin{equation*}
\opi_i \star \gamma  := \begin{cases}
- \gamma & \text{if $i \in \Des_L(\gamma)$}, \\
0 & \text{if $i \notin \Des_L(\gamma)$ and $s_i\gamma \notin [\sigma,\rho]_L$,} \\
s_i \gamma & \text{if $i \notin \Des_L(\gamma)$ and $s_i\gamma \in [\sigma,\rho]_L$.}
\end{cases} 
\end{equation*}
Here $\opi_i = \pi_i - 1$.
It should be pointed out that Hivert, Novelli, and Thibon~\cite{06HNT} introduced semi-combinatorial $H_n(0)$-modules associated to Yang-Baxter intervals $[Y_{\sigma}(\tau), Y_{\rho}(\tau)]$ to study the representation theory of $0$-Ariki-Koike-Shoji algebras, and our $\sfB(\sigma,\rho)$ and $\osfB(\sigma,\rho)$ can also be recovered by the $\tau = \id$ and $\tau = w_0$ specialization of these modules, respectively.
Here, $w_0$ is the longest element of $\SG_n$.

The family of weak and negative weak Bruhat interval modules is very adequate to our purpose in that it contains many $H_n(0)$-modules of our interest 
such as projective indecomposable modules, irreducible modules, the specializations $q = 0$ of the Specht modules of $H_n(q)$ in~\cite{02DHT}, and all indecomposable direct summands of the $H_n(0)$-modules in~\cite{20BS, 15BBSSZ, 19Searles, 15TW,19TW}. 
What is more appealing is that weak and negative weak Bruhat interval modules 
can be embedded into the regular representation of $H_n(0)$ and
they behave very nicely with respect to induction product, restriction, and (anti-)involution twists. 
Let $\mathscr{B}_n$ be the full subcategory of the category $\module H_n(0)$ of finite dimensional $H_n(0)$-modules whose objects are direct sums of weak and negative weak Bruhat interval modules up to isomorphism.
From a categorical point of view, $\mathscr{B}_n$ is a good subcategory in the sense that 
\begin{itemize}
\item the Grothendieck group of $\mathscr{B}_n$ is isomorphic to the Grothendieck group of $\module H_n(0)$, and
\item the subcategory $\bigoplus_{n \ge 0} \mathscr{B}_n$ of $\bigoplus_{n \ge 0} \module H_n(0)$ is closed under induction product, restriction, and (anti-)involution twists.
\end{itemize}

In Section \ref{sec: WBIM}, we study structural properties of weak Bruhat interval modules.
In the first two subsections,
we present background material for weak and negative weak Bruhat interval modules and then construct  
an $H_n(0)$-module isomorphism 
\[
\sfem:\sfB(\sigma,\rho) \ra H_n(0) \pi_\sigma \opi_{\rho^{-1}w_0}, \quad 
\gamma \mapsto \pi_{\gamma}\opi_{\rho^{-1}w_0} \quad \text{for $\gamma \in [\sigma, \rho]_L$}
\]
(Theorem~\ref{thm: embedding}).
As an immediate consequence of this isomorphism, we see that projective indecomposable modules and irreducible modules
appear as weak Bruhat interval modules up to isomorphism.
By a slight modification of $\sfem$, we also see that the specializations $q = 0$ of the Specht modules of $H_n(q)$ are certain involution twists of weak Bruhat interval modules (Remark~\ref{rem: Specht}).
Although not covered here, the isomorphism $\sfem$ and its modification reveal interesting connections between certain $\mathcal{U}_0(gl_N)$-modules and weak Bruhat interval modules (Section~\ref{sec: Further remark}~(3)).

In the third subsection, we study restriction of weak Bruhat interval modules. 
As for induction product $\boxtimes$ of weak Bruhat interval modules, the following formula can be derived from~\cite{06HNT}: 
\begin{align}\label{eq: induction product}
\sfB(\sigma,\rho) \boxtimes \sfB(\sigma',\rho') 
\cong \sfB(\sigma \conc \sigma', ~ \rho \ostar \rho')
\end{align}
(see Lemma~\ref{lem: tensor product}).
We provide an explicit formula concerning restriction of weak Bruhat interval modules.
Let $\matr{[m+n]}{m}$ be the set of $m$-element subsets of $[m+n]$.
Given $\sigma, \rho \in \SG_{m+n}$, set 
\begin{align*}
\scrS_{\sigma, \rho}^{(m)} := \left\{J \in \matr{[m+n]}{m} \; \middle| \; 
J = \gamma^{-1}([1,m])~\text{for some $\gamma \in [\sigma,\rho]_L$}
\right\}.
\end{align*}
With this notation, our restriction rule appears in the following form: 
\begin{align}\label{eq: restriction rule}
\sfB(\sigma,\rho)\downarrow_{H_m(0) \otimes H_n(0)}^{H_{m+n}(0)} 
\hspace{1ex} \cong \hspace{-0.5ex} \bigoplus_{J\in \scrS_{\sigma, \rho}^{(m)} }
\sfB((\bfsigJ)_{\le m},  (\bfrhoJ)_{\le m}) \otimes \sfB((\bfsigJ)_{> m}, (\bfrhoJ)_{> m})
\end{align}
(Theorem~\ref{thm: restriction}).
For undefined notations $(\bfsigJ)_{\le m}$,  $(\bfrhoJ)_{\le m}$, $(\bfsigJ)_{> m}$, and $(\bfrhoJ)_{> m}$, see~\eqref{eq: sigJ rhoJ} and \eqref{eq: gamma >, <}.
The elements in $\scrS_{\sigma, \rho}^{(m)}$ parametrize the direct summands appearing in the right hand side of~\eqref{eq: restriction rule}.
It would be nice to find an easy description of this set, but it is not available at the current stage. 
Combining~\eqref{eq: induction product} with~\eqref{eq: restriction rule} yields a Mackey formula for weak Bruhat interval modules.
We prove that this is a natural lift of the Mackey formula due to Bergeron and Li~\cite{09BL}, which works for elements of the Grothendieck ring of $0$-Hecke algebras, to weak Bruhat interval modules (Theorem~\ref{thm: Mackey formula}).

In the final subsection, we describe the (anti-)involution twists of weak Bruhat interval modules for the involutions $\autophi, \autotheta$ and the anti-involution $\autochi$ due to Fayers~\cite{05Fayers}
and then demonstrate the patterns how these (anti-)involution twists act with respect to induction product and restriction.
Here,  
$\autophi(\pi_i)=\pi_{n-i}$, 
$\autotheta(\pi_i)= - \opi_i$, 
and $\autochi(\pi_i)=\pi_i $ for $1 \le i \le n-1$.
Given an $H_n(0)$-module $M$ and an (anti-)automorphism $\mu$, we denote by $\mu[M]$ the $\mu$-twist of $M$.
The precise definition can be found in~\eqref{eq: isomorphism twist} and~\eqref{eq: anti-automorphism twist}.
We prove that 
\[
\autophi[\sfB(\sigma, \rho)] \cong \sfB(\sigma^{w_0},\rho^{w_0}), \quad 
\autotheta[\sfB(\sigma,\rho)] \cong \osfB(\sigma,\rho), \quad \text{and} \quad
\autochi[\sfB(\sigma,\rho)] \cong \osfB(\rho w_0, \sigma w_0).
\]
Here, $\sigma^{w_0}$ is the conjugation of $\sigma$ by $w_0$. 
Using these isomorphisms, we can also describe the twists for the compositions of $\autophi, \autotheta$, and $\autochi$.
For the full list of (anti-)involution twists, see Table~\ref{tab: auto twist of BIM} and Table~\ref{tab: auto twist of F and P}.
Next, we explain the patterns how the (anti-)involution twists act with respect to induction product and restriction. 
Let $M$, $N$, and $L$ be weak Bruhat interval modules of $H_m(0)$, $H_n(0)$, and $H_{m+n}(0)$ respectively.
We see that 
\begin{gather*}
\autotheta[M \boxtimes N] \cong \autotheta[M] \boxtimes \autotheta[N], 
\\
\text{whereas} 
\quad \autophi[M \boxtimes N] \cong \autophi[N] \boxtimes \autophi[M], 
\quad \autochi [M \boxtimes N] \cong \autochi [N] \boxtimes \autochi [M]
\end{gather*}
and
\begin{gather*}
\autotheta[L \downarrow^{H_{m+n}(0)}_{H_m(0) \otimes H_n(0)}] \cong \autotheta[L] \downarrow^{H_{m+n}(0)}_{H_m(0) \otimes H_n(0)},  \quad
\autochi[L \downarrow^{H_{m+n}(0)}_{H_m(0) \otimes H_n(0)}] \cong \autochi[L] \downarrow^{H_{m+n}(0)}_{H_m(0) \otimes H_n(0)}, \\
\text{whereas} \quad
\autophi[L \downarrow^{H_{m+n}(0)}_{H_n(0) \otimes H_m(0)}] \cong \autophi[L] \downarrow^{H_{m+n}(0)}_{H_m(0) \otimes H_n(0)}
\end{gather*}
(Corollary~\ref{cor: automorhpism, induction, restriction}).

Section~\ref{sec: Examples of WBI modules} is devoted to show that every indecomposable direct summand arising from the $H_n(0)$-modules in~\cite{20BS, 15BBSSZ, 19Searles, 15TW,19TW} are weak Bruhat interval modules up to isomorphism.
We begin with reviewing the results in these papers.
Let $\alpha$ be a composition of $n$.
\begin{itemize}[leftmargin=6mm, itemsep = 0.5em]
\item
Tewari and van Willigenburg~\cite{15TW} construct an $H_n(0)$-module $\bfS_\alpha$ by defining a $0$-Hecke action on the set of \emph{standard reverse composition tableaux of shape $\alpha$}.
Its image under the quasisymmetric characteristic is the quasisymmetric Schur function attached to $\alpha$.

\item
Let $\upsig$ be a permutation in $\SG_{\ell(\alpha)}$, where $\ell(\alpha)$ is the length of $\alpha$. 
Tewari and van Willigenburg~\cite{19TW} construct an $H_n(0)$-module $\bfSsa$ by defining a $0$-Hecke action on the set $\SPCTsa$ of \emph{standard permuted composition tableaux of shape $\alpha$ and type $\upsig$}.
When $\upsig={\rm id}$, this module is equal to $\bfS_\alpha$ in~\cite{15TW}.

\item
Berg et al.~\cite{15BBSSZ} construct an indecomposable $H_n(0)$-module $\calV_\alpha$ by defining a $0$-Hecke action on the set $\SIT(\alpha)$ of \emph{standard immaculate tableaux of shape $\alpha$}.
Its image under the quasisymmetric characteristic is the dual immaculate quasisymmetric function attached to $\alpha$. 

\item
Searles~\cite{19Searles} constructs an indecomposable $H_n(0)$-module $X_\alpha$ by defining a $0$-Hecke action on the set $\SET(\alpha)$ of \emph{standard extended tableaux of shape $\alpha$}.
Its image under the quasisymmetric characteristic is the extended Schur function attached to $\alpha$.

\item
Bardwell and Searles~\cite{20BS} construct an $H_n(0)$-module $\bfR_\alpha$ by defining a $0$-Hecke action on the set $\SYRT(\alpha)$ of \emph{standard Young row-strict tableaux of shape $\alpha$}.
Its image under the quasisymmetric characteristic is the Young row-strict quasisymmetric Schur function attached to $\alpha$.
\end{itemize}
The modules $\calV_\alpha$ and $X_\alpha$ are indecomposable, 
whereas $\bfSsa$, and $\bfR_\alpha$ are not indecomposable in general.
The problem of decomposing $\bfSsa$ into indecomposables have been completely settled out by virtue of the papers~\cite{20CKNO, 19Konig, 15TW,19TW}. 
Indeed, $\bfSsa$ has the decomposition of the form
\begin{align*}
\bfSsa = \bigoplus_{E \in \calEsa} \bfSsaE,
\end{align*}
where $\calEsa$ represents the set of equivalence classes of $\SPCTsa$ under a certain distinguished equivalence relation and $\bfSsaE$ is the indecomposable submodule of $\bfSsa$ spanned by $E$.

To achieve our purpose, we first show that all of $\SIT(\alpha)$, $\SET(\alpha)$, and $E (\in \calEsa)$ have the source and sink, which are written as follows: 
\[
\begin{tabularx}{0.5\textwidth}{
>{\centering\arraybackslash}X |
>{\centering\arraybackslash}X |
>{\centering\arraybackslash}X |
>{\centering\arraybackslash}X }
 & $\SIT(\alpha)$ & $\SET(\alpha)$ &  $E$ \\ \hline
source & $\calT_\alpha$ & $\sfT_\alpha$ & $\tau_E$ \\ \hline
sink & $\calT'_\alpha$ & $\sfT'_\alpha$ & $\tau'_E$
\end{tabularx}
\]
For the definitions of source and sink, see Definition~\ref{def: of source and sink}.
Then, using the essential epimorphisms constructed in~\cite{20CKNO2}, 
we introduce two readings $\rmread$ and $\tread$, 
where the former is defined on $\SIT(\alpha)$ and $\SET(\alpha)$ and the latter is defined on each class $E$.
With this preparation, we prove that 
\begin{gather*}
\calV_\alpha \cong \sfB(\rmread(\calT_\alpha), \rmread(\calT'_\alpha)), \
X_\alpha \cong \sfB(\rmread(\sfT_\alpha),\rmread(\sfT'_\alpha)),\\ \text{and} \quad  
\bfSsaE \cong \sfB(\tread(\tauE), \tread(\tau'_E))
\end{gather*}
(Theorem~\ref{thm: Va Xa pi opi form} and Theorem~\ref{thm: bfSsaE pi opi form}).
It should be remarked that a reading on $E$ different from ours has already been introduced by Tewari and van Willigenburg~\cite{15TW}.
They assign a reading word $\col_{\tau}$ to each standard reverse composition tableau $\tau$ and
show that $(E, \preceq)$ is isomorphic to $([\col_{\tauE},\col_{\tau_E'}]_L, \preceq_L)$ as graded posets.
The weak Bruhat interval module $\sfB(\col_{\tau_E}, \col_{\tau_E'})$, however, is not isomorphic to $\bfS_\alpha$ in general (see Remark~\ref{rem: SPCT conj}).
Concerned with $\bfR_\alpha$, we consider its permuted version instead of itself.
We introduce new combinatorial objects, called \emph{permuted standard Young row-strict tableaux of shape $\alpha$ and type $\upsig$}, and define a $0$-Hecke action on them.
The resulting module $\bfR^\upsig_\alpha$ turns out to be isomorphic to the $\hautoomega$-twist of 
$\bfS^{\upsig^{w_0}}_{\alpha^{\rmr}}$.
This enables us to transport various properties of $\bfS^{\upsig^{w_0}}_{\alpha^{\rmr}}$ to $\bfR^\upsig_\alpha$ 
in a functorial way. 

In the final section, we provide some future directions to pursue.

\section{Preliminaries}\label{sec: Preliminaries}

Given any integers $m$ and $n$, define $[m,n]$ to be the interval $\{t\in \mathbb Z: m\le t \le n\}$ whenever $m \le n$ and the empty set $\emptyset$ else.
For simplicity, we set $[n]:=[1,n]$.
Unless otherwise stated, $n$ will denote a nonnegative integer throughout this paper.

\subsection{Compositions and their diagrams}\label{subsec: comp and diag}

A \emph{composition} $\alpha$ of $n$, denoted by $\alpha \models n$, is a finite ordered list of positive integers $(\alpha_1,\alpha_2,\ldots, \alpha_k)$ satisfying $\sum_{i=1}^k \alpha_i = n$.
We call $\alpha_i$ ($1 \le i \le k$) a \emph{part} of $\alpha$,  
$k =: \ell(\alpha)$ the \emph{length} of $\alpha$, and $n =:|\alpha|$ the \emph{size} of $\alpha$. 
And, we define the empty composition $\varnothing$ to be the unique composition of size and length $0$.

Given $\alpha = (\alpha_1,\alpha_2,\ldots,\alpha_k) \models n$ and $I = \{i_1 < i_2 < \cdots < i_k\} \subset [n-1]$, 
let $\set(\alpha) := \{\alpha_1,\alpha_1+\alpha_2,\ldots, \alpha_1 + \alpha_2 + \cdots + \alpha_{k-1}\}$ and $\comp(I) := (i_1,i_2 - i_1, i_3 - i_2, \ldots,n-i_k)$.
The set of compositions of $n$ is in bijection with the set of subsets of $[n-1]$ under the correspondence $\alpha \mapsto \set(\alpha)$ (or $I \mapsto \comp(I)$).
The \emph{reverse composition $\alpha^\rmr$ of $\alpha$} is defined to be $(\alpha_k, \alpha_{k-1}, \ldots, \alpha_1)$, 
the \emph{complement $\alpha^\rmc$ of $\alpha$} be the unique composition satisfying $\set(\alpha^c) = [n-1] \setminus \set(\alpha)$, and the \emph{conjugate $\alpha^\rmt$ of $\alpha$} be the composition $(\alpha^\rmr)^\rmc=(\alpha^\rmc)^\rmr$.

\subsection{Weak Bruhat orders on the symmetric group}\label{subsec: Symmetric group}

Every element $\sigma$ of the symmetric group $\SG_n$ may be written as a word in $s_i := (i,i+1)$ with $1 \le i \le n-1$. 
A \emph{reduced expression for $\sigma$} is one of minimal length.
The number of simple transpositions in any reduced expression for $\sigma$, denoted by $\ell(\sigma)$, is called the \emph{length} of $\sigma$.
Let
\[
\Des_L(\sigma):= \{i \in [n-1] \mid \ell(s_i \sigma) < \ell(\sigma)\}
\ \  \text{and} \ \  
\Des_R(\sigma):= \{i \in [n-1] \mid \ell(\sigma s_i) < \ell(\sigma)\}.
\]
It is well known that if $\sigma = w_1 w_2 \cdots w_n$ in one-line notation, then
\begin{align}\label{eq: Des alternating def}
\begin{aligned}
\Des_L(\sigma) & = \{ i \in [n-1] \mid \text{$i$ is right of $i+1$ in $w_1 w_2 \cdots w_n$} \} \quad \text{and} \\
\Des_R(\sigma) & = \{ i \in [n-1] \mid w_i > w_{i+1}
\}.
\end{aligned} 
\end{align}

The \emph{left weak Bruhat order} $\preceq_L$ 
(resp. \emph{right weak Bruhat order} $\preceq_R$) on $\SG_n$
is the partial order on $\SG_n$ whose covering relation $\preceq_L^c$ (resp. $\preceq_R^c$) is defined as follows: 
$\sigma \preceq_L^c s_i \sigma$ if and only if $i \notin \Des_L(\sigma)$
(resp. $\sigma \preceq_R^c \sigma s_i$ if and only if $i \notin \Des_R(\sigma)$).
Equivalently, for any $\sigma, \rho \in \SG_n$,
\begin{align*}
& \text{$\sigma \preceq_L \rho$ if and only if 
$\rho = \gamma \sigma$ and $\ell(\rho) = \ell(\sigma) + \ell(\gamma)$ for some $\gamma \in \SG_n$, and} \\
& \text{$\sigma \preceq_R \rho$ if and only if
$\rho = \sigma \gamma$ and $\ell(\rho) = \ell(\sigma) + \ell(\gamma)$ for some $\gamma \in \SG_n$}.
\end{align*}
Although these two weak Bruhat orders are not identical,  
there exists a poset isomorphism $$(\SG_n, \preceq_L)\to (\SG_n, \preceq_R),\quad \sigma \mapsto \sigma^{-1}.$$
To avoid redundant overlap, our statements will be restricted to the left weak Bruhat order.

Given $\sigma$ and $\rho \in \SG_n$, the closed interval $\{\gamma \in \SG_n \mid \sigma \preceq_L \gamma \preceq_L \rho \}
$ is called the \emph{left weak Bruhat interval from $\sigma$ to $\rho$} and denoted by $[\sigma,\rho]_L$. 
It can be represented by the colored digraph whose vertices are given by $[\sigma,\rho]_L$ and $\{1,2,\ldots, n-1\}$-colored arrows given by
\begin{align*}
\gamma \overset{i}{\ra} \gamma' \quad \text{if and only if} \quad 
\text{$\gamma \preceq_L \gamma'$ and $s_i \gamma = \gamma'$}. 
\end{align*}

Let us collect notations which will be used later.
We use $w_0$ to denote the longest element in $\SG_n$.
For $I \subseteq [n-1]$, let $\SG_{I}$ be the parabolic subgroup of $\SG_n$ generated by $\{s_i \mid i\in I\}$ and $w_0(I)$ the longest element in $\SG_{I}$.
For $\alpha \models n$, let $w_0(\alpha) := w_0(\set(\alpha))$.
Finally, for $\sigma \in \SG_n$, we let $\sigma^{w_0} := w_0 \sigma w_0$.

\subsection{The $0$-Hecke algebra and the quasisymmetric characteristic}\label{subsec: 0-Hecke alg and QSym}
The $0$-Hecke algebra $H_n(0)$ is the associative $\C$-algebra with $1$ generated by the elements $\pi_1,\pi_2,\ldots,\pi_{n-1}$ subject to the following relations:
\begin{align*}
\pi_i^2 &= \pi_i \quad \text{for $1\le i \le n-1$},\\
\pi_i \pi_{i+1} \pi_i &= \pi_{i+1} \pi_i \pi_{i+1}  \quad \text{for $1\le i \le n-2$},\\
\pi_i \pi_j &=\pi_j \pi_i \quad \text{if $|i-j| \ge 2$}.
\end{align*}
Frequently we use another set of generators  $\{\opi_i := \pi_i -1 \mid 1 \le i \le n-1\}$.

For any reduced expression $s_{i_1} s_{i_2} \cdots s_{i_p}$ for $\sigma \in \SG_n$, let 
\[
\pi_{\sigma} := \pi_{i_1} \pi_{i_2 } \cdots \pi_{i_p} \quad \text{and} \quad \opi_{\sigma} := \opi_{i_1} \opi_{i_2} \cdots \opi_{i_p}.
\]
It is well known that these elements are independent of the choices of reduced expressions, and both $\{\pi_\sigma \mid \sigma \in \SG_n\}$ and $\{\opi_\sigma \mid \sigma \in \SG_n\}$ are $\mathbb C$-bases for $H_n(0)$.

According to \cite{79Norton}, there are $2^{n-1}$ pairwise inequivalent irreducible $H_n(0)$-modules and $2^{n-1}$ pairwise inequivalent projective indecomposable $H_n(0)$-modules, which are naturally indexed by compositions of $n$.
For $\alpha \models n$, let $\bfF_{\alpha}=\C v_{\alpha}$ and endow it with the $H_n(0)$-action as follows: for each $1 \le i \le n-1$,
\[
\pi_i \cdot v_\alpha = \begin{cases}
0 & i \in \set(\alpha),\\
v_\alpha & i \notin \set(\alpha).
\end{cases}
\]
This module is the irreducible $1$-dimensional $H_n(0)$-module corresponding to $\alpha$.
And, the projective indecomposable $H_n(0)$-module corresponding to $\alpha$
is given by the submodule 
\[
\calP_\alpha := H_n(0) \pi_{w_0(\alpha^\rmc)} \opi_{w_0(\alpha)}
\]
of the regular representation of $H_n(0)$.
It is known that $\bfF_\alpha$ is isomorphic to $\calP_\alpha / \rad \; \calP_\alpha$, where $\rad \; \calP_\alpha$ is the \emph{radical} of $\calP_\alpha$, the intersection of maximal submodules of $\calP_\alpha$.

Let $\calR(H_n(0))$ denote the $\Z$-span of the isomorphism classes of finite dimensional $H_n(0)$-modules. 
We denote by $[M]$ the isomorphism class corresponding to an $H_n(0)$-module $M$. 
The \emph{Grothendieck group} $\calG_0(H_n(0))$ is the quotient of $\calR(H_n(0))$ modulo the relations $[M] = [M'] + [M'']$ whenever there exists a short exact sequence $0 \ra M' \ra M \ra M'' \ra 0$. 
The irreducible $H_n(0)$-modules form a free $\Z$-basis for $\calG_0(H_n(0))$. Let
\[
\calG := \bigoplus_{n \ge 0} \calG_0(H_n(0)).
\]

Let us review the beautiful connection between $\calG$ and the ring $\Qsym$ of quasisymmetric functions.
For the definition of quasisymmetric functions, see~\cite[Section 7.19]{99Stanley}.

For a composition $\alpha$, the \emph{fundamental quasisymmetric function} $F_\alpha$, introduced in~\cite{84Gessel}, is defined by
\[
F_\varnothing = 1 
\quad \text{and} \quad
F_\alpha = \sum_{\substack{1 \le i_1 \le i_2 \le \cdots \le i_k \\ i_j < i_{j+1} \text{ if } j \in \set(\alpha)}} x_{i_1} \cdots x_{i_k} \quad \text{if $\alpha \neq \varnothing$}.
\]
It is known that $\{F_\alpha \mid \text{$\alpha$ is a composition}\}$ forms a $\mathbb Z$-basis for $\Qsym$.
When $M$ is an $H_m(0)$-module and $N$ is
an $H_n(0)$-module, we write $M \boxtimes N$ for the induction product of $M$ and $N$, that is,
\begin{align*}
M \boxtimes N= M \otimes N \uparrow_{H_m(0) \otimes H_n(0)}^{H_{m+n}(0)}.
\end{align*}
Here, $H_m(0) \otimes H_n(0)$ is viewed as the subalgebra of $H_{m+n}(0)$ generated by $\{\pi_i \mid i \in [m+n-1] \setminus \{m\} \}$.

It was shown in \cite{96DKLT} that, when $\calG$ is equipped with this product as multiplication, the linear map
\begin{equation}\label{quasi characteristic}
\ch : \calG \ra \Qsym, \quad [\bfF_{\alpha}] \mapsto F_{\alpha},
\end{equation}
called \emph{quasisymmetric characteristic}, is a ring isomorphism.
Indeed it is not only a ring isomorphism, but also a Hopf algebra isomorphism.

\section{Weak Bruhat interval modules}\label{sec: WBIM}

In this section, we present background material for weak Bruhat interval modules and then investigate their structural properties extensively. 

\subsection{Definition and basic properties}\label{subsec: def and basic properties}

Let us start with the definitions of weak and negative weak Bruhat interval modules.

\begin{definition}
Let $\sigma, \rho \in \SG_n$.
\begin{enumerate}
\item 
The \emph{weak Bruhat interval module associated to $[\sigma,\rho]_L$}, denoted by $\sfB(\sigma,\rho)$, is the $H_n(0)$-module with  $\C[\sigma,\rho]_L$ as the underlying space and with the $H_n(0)$-action defined by
\begin{align}\label{Hecke algebra action}
\pi_i \cdot \gamma := \begin{cases}
\gamma & \text{if $i \in \Des_L(\gamma)$}, \\
0 & \text{if $i \notin \Des_L(\gamma)$ and $s_i\gamma \notin [\sigma,\rho]_L$,} \\
s_i \gamma & \text{if $i \notin \Des_L(\gamma)$ and $s_i\gamma \in [\sigma,\rho]_L$}.
\end{cases} 
\end{align}
\item 
The \emph{negative weak Bruhat interval module associated to $[\sigma,\rho]_L$}, denoted by $\osfB(\sigma,\rho)$, is the $H_n(0)$-module with  $\C[\sigma,\rho]_L$ as the underlying space and with the $H_n(0)$-action defined by
\begin{align*}
\opi_i \star \gamma  := \begin{cases}
- \gamma & \text{if $i \in \Des_L(\gamma)$}, \\
0 & \text{if $i \notin \Des_L(\gamma)$ and $s_i\gamma \notin [\sigma,\rho]_L$,} \\
s_i \gamma & \text{if $i \notin \Des_L(\gamma)$ and $s_i\gamma \in [\sigma,\rho]_L$}.
\end{cases} 
\end{align*}
\end{enumerate}
\end{definition}
Hivert, Novelli, and Thibon~\cite{06HNT} introduced a semi-combinatorial $H_n(0)$-module associated to the interval $[Y_\sigma(\tau), Y_\rho(\tau)]$ of a Yang-Baxter basis $\{Y_\gamma(\tau)\}_{\gamma \in \SG_n}$ for each $\tau \in \SG_n$.
We omit the proof of the well-definedness of $\sfB(\sigma,\rho)$ and $\osfB(\sigma,\rho)$ since they can be recovered as the semi-combinatorial modules associated to $[Y_\sigma(\id), Y_\rho(\id)]$ and $[Y_\sigma(w_0), Y_\rho(w_0)]$, respectively. 

The following properties are almost straightforward. 
In particular, the last one can be obtained by mimicking \cite[Section 5]{15TW}.
\begin{enumerate}[label = $\bullet$]
\item Given $\sigma' \in [\sigma, \rho]_L$, the linear map $\iota: \sfB(\sigma',\rho) \ra \sfB(\sigma, \rho)$, sending $\gamma \mapsto \gamma$, is an injective $H_n(0)$-module homomorphism.
\item Given $\rho' \in [\sigma, \rho]_L$, the linear map $\pr: \sfB(\sigma,\rho) \ra \sfB(\sigma, \rho')$, sending $\gamma \mapsto \gamma$ if $\gamma \in [\sigma,\rho']_L$ and $\gamma \mapsto 0$ else, is a surjective $H_n(0)$-module homomorphism.
\item 
$\ch([\sfB(\sigma, \rho)]) = \sum_{\gamma \in [\sigma,\rho]_L} F_{\comp(\gamma)^\rmc}$,
where $\comp(\gamma) := \comp(\Des_L(\gamma))$.
\end{enumerate}

Finally, we briefly remark the classification of weak Bruhat interval modules up to isomorphism.
As pointed out in Subsection~\ref{subsec: Symmetric group}, every weak Bruhat interval can be viewed as a colored digraph.
One sees from \eqref{Hecke algebra action} that  
$\sfB(\sigma,\rho) \cong \sfB(\sigma',\rho')$ if 
there is a descent-preserving (colored digraph) isomorphism between $[\sigma,\rho]_L$ and $[\sigma',\rho']_L$.
For instance, 
$$
f:[14325,24315]_L \to [41352,42351]_L, 
\quad \pi_\gamma \cdot 14325 \mapsto \pi_\gamma \cdot 41352 
\quad (\gamma \in \SG_5)
$$ 
is such an isomorphism, thus $\sfB(14325,24315)\cong \sfB(41352,42351)$.  
On the other hand, it is not difficult to show that there is no 
descent-preserving isomorphism between $[14325,24315]_L$ and $[45312,45321]_L$ although they are isomorphic as colored digraphs.
Indeed, $\sfB(14325,24315)$ is indecomposable (Remark~\ref{rem: SPCT conj}), 
whereas $\sfB(45312,45321)$ can be decomposed into $\C 45321 \oplus \C (45312-45321)$ as seen in  Figure~\ref{fig:equi int not guarantee module iso}.
It would be very nice to characterize when a descent-preserving isomorphism between two intervals exists and ultimately to classify all weak Bruhat interval modules up to isomorphism.
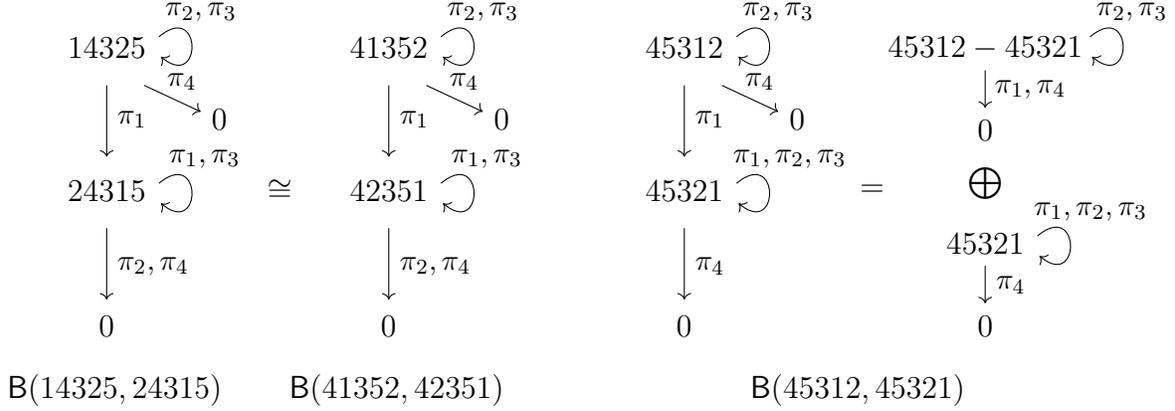
\begin{figure}[t]
\def \hp{0.15}
\def \vp{0.3}
\def \hhp{3.75}
\def \hhhp{4}
\begin{tikzpicture}[baseline = 0mm]
\node[below] at (0.1,-3.6) {$\sfB(14325,24315)$};

\node at (0,0+2*\vp) {
$14325$
};
\node at (0.7-\hp,0.25+1.25*\vp) {} edge [out=40,in=320, loop] ();
\node at (1.4-\hp,0.7+1.25*\vp) {\small $\pi_2,\pi_3$};

\node at (1.5,-1.25+3*\vp) {$0$};
\draw [->] (0.5,-0.8+3*\vp) -- (1.25,-1.15+3*\vp);
\node at (1,-0.75+3*\vp) {\small $\pi_4$};

\draw [->] (0,-0.8+3*\vp) -- (0,-1.75+3*\vp);
\node[right] at (0.,-1.25+3*\vp) {\small $\pi_1$};

\node at (0,-2.5+4*\vp) {
$24315$
};
\node at (0.7-\hp,-2.75+4.75*\vp) {} edge [out=40,in=320, loop] ();
\node at (1.3-0*\hp,-2.3+4.75*\vp) {\small $\pi_1,\pi_3$};

\draw [->] (0,-3.0 +4*\vp) -- (0,-3.95+4*\vp);
\node[right] at (0,-3.5+4*\vp) {\small $\pi_2, \pi_4$};
\node at (0,-4.6+5*\vp) {$0$};

\node at (2.3,-2.5+4*\vp) {$\cong$};

\node[below] at (0.1 + \hhp,-3.6) {$\sfB(41352,42351)$};

\node at (0 + \hhp,0+2*\vp) {
$41352$
};
\node at (0.7-\hp + \hhp, 0.25+1.25*\vp) {} edge [out=40,in=320, loop] ();
\node at (1.4-\hp + \hhp, 0.7+1.25*\vp) {\small $\pi_2,\pi_3$};

\node at (1.5 + \hhp, -1.25+3*\vp) {$0$};
\draw [->] (0.5 + \hhp, -0.8+3*\vp) -- (1.25 + \hhp,-1.15+3*\vp);
\node at (1 + \hhp, -0.75+3*\vp) {\small $\pi_4$};

\draw [->] (0 + \hhp, -0.8+3*\vp) -- (0 + \hhp, -1.75+3*\vp);
\node[right] at (0 + \hhp, -1.25+3*\vp) {\small $\pi_1$};

\node at (0 + \hhp, -2.5+4*\vp) {
$42351$
};
\node at (0.7-\hp + \hhp, -2.75+4.75*\vp) {} edge [out=40,in=320, loop] ();
\node at (1.3-0*\hp + \hhp, -2.3+4.75*\vp) {\small $\pi_1,\pi_3$};

\draw [->] (0 + \hhp, -3.0 +4*\vp) -- (0 + \hhp,-3.95+4*\vp);
\node[right] at (0 + \hhp, -3.5+4*\vp) {\small $\pi_2, \pi_4$};
\node at (0 + \hhp,-4.6+5*\vp) {$0$};
\end{tikzpicture}
\hspace{6ex}
\begin{tikzpicture}[baseline = 0mm]
\node[below] at (2.3,-3.6) {$\sfB(45312,45321)$};

\node at (0,0+2*\vp) {
$45312$
};
\node at (0.7-\hp,0.25+1.25*\vp) {} edge [out=40,in=320, loop] ();
\node at (1.4-\hp,0.7+1.25*\vp) {\small $\pi_2,\pi_3$};

\node at (1.5,-1.25+3*\vp) {$0$};
\draw [->] (0.5,-0.8+3*\vp) -- (1.25,-1.15+3*\vp);
\node at (1,-0.75+3*\vp) {\small $\pi_4$};

\draw [->] (0,-0.8+3*\vp) -- (0,-1.75+3*\vp);
\node[right] at (0.,-1.25+3*\vp) {\small $\pi_1$};

\node at (0,-2.5+4*\vp) {
$45321$
};
\node at (0.7-\hp,-2.75+4.75*\vp) {} edge [out=40,in=320, loop] ();
\node at (1.4-0*\hp,-2.3+4.75*\vp) {\small $\pi_1,\pi_2,\pi_3$};

\draw [->] (0,-3.0 +4*\vp) -- (0,-3.95+4*\vp);
\node[right] at (0,-3.5+4*\vp) {\small $\pi_4$};
\node at (0,-4.6+5*\vp) {$0$};

\node at (2.5,-2.5+4*\vp) {$=$};

\node at (0+\hhhp, 0+2*\vp) {
$45312-45321$
};
\node at (0.75-\hp+\hhhp + 0.65, 0.25+1.25*\vp) {} edge [out=40,in=320, loop] ();
\node at (1.45-\hp+\hhhp+ 0.65, 0.7+1.25*\vp) {\small $\pi_2,\pi_3$};

\draw [->] (0+\hhhp, -0.6+3*\vp) -- (0+\hhhp, -1.1+3*\vp);
\node at (0.6+\hhhp, -0.85+3*\vp) {\small $\pi_1, \pi_4$};
\node at (0+\hhhp, -1.4+3*\vp) {$0$};

\node at (0+\hhhp, -2.4+4*\vp) {$\bigoplus$};

\node at (0+\hhhp, -3.2+4*\vp) {
$45321$
};
\node at (0.7-\hp+\hhhp, -3.45+4.75*\vp) {} edge [out=40,in=320, loop] ();
\node at (1.4-0*\hp+\hhhp, -3+4.75*\vp) {\small $\pi_1,\pi_2,\pi_3$};

\draw [->] (0+\hhhp, -3.8 + 5*\vp) -- (0+\hhhp, -4.3+5*\vp);
\node[right] at (0+\hhhp, -4.05 + 5*\vp) {\small $\pi_4$};
\node at (0+\hhhp, -4.6+5*\vp) {$0$};
\end{tikzpicture}
\caption{$\sfB(14325,24315)$, $\sfB(41352,42351)$, and $\sfB(45312,45321)$}
\label{fig:equi int not guarantee module iso}
\end{figure}

\subsection{Embedding weak Bruhat interval modules into the regular representation}\label{subsec: embedding}

The purpose of this subsection is to see that every weak Bruhat interval modules can be embedded into the regular representation.
More precisely, we prove that $\sfB(\sigma,\rho)$ is isomorphic to $H_n(0)\pi_\sigma \opi_{\rho^{-1} w_0}$.

For ease of notation, we use $\rpi_i$ to denote an arbitrary element in $\{\pi_i, \opi_i\}$ for each $1\le i \le n-1$.
The following relations among $\pi_i$'s and $\opi_i$'s, which will be used significantly, can be easily verified:    
\begin{equation}\label{eq: relations on H_n(0)}
\begin{aligned}
&\pi_i \pi_{i+1} \opi_i = \opi_{i+1} \pi_i \pi_{i+1}, \quad \pi_{i+1} \pi_{i} \opi_{i+1} = \opi_{i} \pi_{i+1} \pi_{i}  &(1\le i \le n-2),\\
& \pi_i \opi_{i+1} \opi_i = \opi_{i+1} \opi_i \pi_{i+1}, \quad \pi_{i+1} \opi_{i} \opi_{i+1} = \opi_{i} \opi_{i+1} \pi_{i} &  (1\le i \le n-2), \\
& \pi_i \opi_j = \opi_j \pi_i  &  (|i-j| \ge 2).
\end{aligned}
\end{equation}

One sees that $\pi_r \opi_{r+1} \pi_r, \pi_{r+1} \opi_{r} \pi_{r+1}, \opi_r \pi_{r+1} \opi_r$, and $\opi_{r+1} \pi_{r} \opi_{r+1}$ are missing 
in the list of ~\eqref{eq: relations on H_n(0)}.
The following lemma tells us the reason for this. 

\begin{lemma}\label{lem: constraint of braid}
For $\sigma, \rho \in \SG_n$ with $\ell(\sigma \rho) = \ell(\sigma) + \ell(\rho)$,
let $s_{u_1}s_{u_2}\cdots s_{u_l}$ and $s_{v_1}s_{v_2}\cdots s_{v_m}$  be arbitrary reduced expressions of $\sigma$ and $\rho$, respectively.
Let $\urmw$ be the word
\[
\pi_{u_1}\pi_{u_2}\cdots \pi_{u_l} \opi_{v_1}\opi_{v_2}\cdots \opi_{v_m}
\]
from the alphabet $\{\pi_i, \opi_i \mid 1\le i \le n-1\}$, where $\pi_i, \opi_i$ are viewed as letters not as elements of $H_n(0)$.  
Then every word $\rmw = w_1w_2 \cdots w_{l+m}$ obtained from $\urmw$ by applying only the following six braid relations and three commutation-relations  
\begin{align}\label{eq: braid relation}
&\begin{aligned}
\pi_i \pi_{i+1} \pi_i &= \pi_{i+1}\pi_{i} \pi_{i+1} &  (1\le i \le n-2),\\ 
\opi_i \opi_{i+1} \opi_i &= \opi_{i+1} \opi_{i} \opi_{i+1} &  (1\le i \le n-2),\\ 
\pi_i \pi_{i+1} \opi_i &= \opi_{i+1} \pi_i \pi_{i+1} &  (1\le i \le n-2),\\
\pi_{i+1} \pi_{i} \opi_{i+1} &= \opi_{i} \pi_{i+1} \pi_{i} &  (1\le i \le n-2) ,\\
\pi_i \opi_{i+1} \opi_i &= \opi_{i+1} \opi_i \pi_{i+1} &  (1\le i \le n-2),\\
\pi_{i+1} \opi_{i} \opi_{i+1} &= \opi_{i} \opi_{i+1} \pi_{i} &  (1\le i \le n-2),
\end{aligned}
\\[1ex]
\label{eq: commutation relation}
&\hspace{6ex}
\begin{aligned}
\pi_i \pi_j &= \pi_j \pi_i  &  (|i-j| \ge 2),\\
\opi_i \opi_j &= \opi_j \opi_i  &  (|i-j| \ge 2),\\
\pi_i \opi_j &= \opi_j \pi_i  &  (|i-j| \ge 2),
\end{aligned}
\end{align}
contains no subwords $w_rw_{r+1}w_{r+2}$ $(1\le r \le l+m-2)$ of the form 
\begin{align}\label{eq: prohibited subword}
\pi_i \opi_{i+1} \pi_i, \quad \pi_{i+1} \opi_{i} \pi_{i+1},
\quad \opi_i \pi_{i+1} \opi_i,   
\quad \opi_{i+1} \pi_{i} \opi_{i+1}. 
\end{align}
\end{lemma}

\begin{proof}
Let $k_{\rmw}$ denote the number of occurrences of 
the relations in~\eqref{eq: braid relation} in the middle of obtaining $\rmw$ from $\urmw$
by applying the relations in both ~\eqref{eq: braid relation} and~\eqref{eq: commutation relation}.
We will prove our assertion using the mathematical induction 
on $k_{\rmw}$. 
When $k_{\rmw}=0$, our assertion is obvious since $\rmw$ is obtained from $\urmw$ by applying the relations in~\eqref{eq: commutation relation} only.

Given a positive integer $\hat{k}$, assume that our assertion holds for all $\rmw$ whenever $k_{\rmw} < \hat{k}$.
Now, let $k_{\rmw}= \hat{k}$, which means that there is a  word $\rmw' = w'_1 w'_2 \cdots w'_{l+m}$ such that 
$k_{\rmw'}={\hat{k}} -1$ and $\rmw$ is obtained from $\rmw'$ by applying 
the relations in~\eqref{eq: commutation relation} and one of the relations in~\eqref{eq: braid relation} only once.
By the first paragraph, we have only to consider the case where $\rmw$ is obtained from $\rmw'$ by applying a relation in~\eqref{eq: braid relation} to $w'_{t_0 - 1} w'_{t_0} w'_{t_0 + 1}$ for some $1 < t_0 < l+m$.
Suppose that $\rmw$ contains a subword $w_{j_0-1} w_{j_0} w_{j_0 + 1}$ of the form~\eqref{eq: prohibited subword}.
We observe that no words in~\eqref{eq: prohibited subword} appear in the six relations in ~\eqref{eq: braid relation}, 
so $j_0\ne t_0$.
In case where $|j_0 - t_0| > 2$, $w'_{j_0-1} w'_{j_0} w'_{j_0+1}$ is of the form~\eqref{eq: prohibited subword}.
But, this is absurd since $k_{\rmw'} < \hat{k}$.
In case where $0 < |j_0 - t_0| \le 2$, $\rmw$ has a consecutive subword of the form $\rpi_{r} \rpi_{r+1} \rpi_{r} \rpi_{r+1}$ or $\rpi_{r+1} \rpi_{r} \rpi_{r+1} \rpi_{r}$.
For $1\le t \le l+m$, let $i_t$ be the subindex given by $w_t=\rpi_{i_t}$.
Using the notation $\rpi_{i}$ instead of $\pi_{i}$ and $\opi_{i}$, 
one can see that the relations in both~\eqref{eq: braid relation} and~\eqref{eq: commutation relation} 
coincide with braid and commutation relations of $\SG_n$. 
It says that $s_{i_1} s_{i_2} \cdots s_{i_{l+m}}$ is obtained from $s_{u_1} s_{u_2} \cdots s_{u_l} s_{v_1} s_{v_2} \cdots s_{v_m}$ by applying braid and commutation relations of $\SG_n$, thus is a reduced expression for $\sigma \rho$. 
This, however, cannot occur since 
$s_{i_1} s_{i_2} \cdots s_{i_{l+m}}$ contains a consecutive subword of the form 
$s_{r} s_{r+1} s_{r} s_{r+1}$ or $s_{r+1} s_{r} s_{r+1} s_{r}$.
Consequently, $\rmw$ contains no subwords of the form~\eqref{eq: prohibited subword}, so we are done.
\end{proof}

Given a word $\rmw = w_1 w_2 \cdots w_p$ from the alphabet $\{\pi_i, \opi_i \mid 1\le i \le n-1\}$, 
we can naturally see it as an element of $H_n(0)$. 
For clarity, we write it as $\iota(\rmw)$.
If a word $\rmw'$ is obtained from $\rmw$ by applying the relations in~\eqref{eq: braid relation} and~\eqref{eq: commutation relation}, then $\iota(\rmw) = \iota(\rmw')$ by~\eqref{eq: relations on H_n(0)}.
Combining ~\eqref{eq: relations on H_n(0)} with Lemma~\ref{lem: constraint of braid} yields the following lemma.

\begin{lemma}\label{lem: Des and last pi_i} 
For $\sigma, \rho \in \SG_n$ with $\ell(\sigma \rho) = \ell(\sigma) + \ell(\rho)$, if $j_0 \in \Des_R(\sigma \rho) \setminus \Des_R(\rho)$, then 
there exists a word $\rmw = w_1 w_2 \cdots w_{\ell(\sigma \rho)}$ from the alphabet $\{\pi_i, \opi_i \mid 1\le i \le n-1\}$ such that
$\iota(\rmw) = \pi_\sigma \opi_\rho$ and $w_{\ell(\sigma \rho)} = \pi_{j_0}$.
\end{lemma}

\begin{proof}
Let $s_{u_1}s_{u_2} \cdots s_{u_l}$ be a reduced expression of $\sigma$ and $s_{v_1}s_{v_2} \cdots s_{v_m}$ a reduced expression of $\rho$.
Since $j_0 \in \Des_R(\sigma \rho) \setminus \Des_R(\rho)$, by the exchange property of Coxeter groups, there exists $1 \le r \le l$ such that $\sigma \rho = s_{u_1} s_{u_2} \cdots \widehat{s}_{u_r} \cdots s_{u_l} s_{v_1} s_{v_2} \cdots s_{v_m} s_{j_0}$, where $s_{u_1} s_{u_2} \cdots \widehat{s}_{u_r} \cdots s_{u_l}$ denotes the permutation obtained from $s_{u_1} s_{u_2} \cdots s_{u_l}$ by removing $s_{u_r}$.
Viewing $\pi_\sigma \opi_\rho$ just as a word from the alphabet $\{\rpi_i \mid 1\le i \le n-1\}$, 
due to ~\eqref{eq: relations on H_n(0)} and Lemma~\ref{lem: constraint of braid}, 
the relations in~\eqref{eq: braid relation} and~\eqref{eq: commutation relation} 
play the same role as the braid and commutation relations of $\SG_n$.
This implies that 
$\pi_\sigma \opi_\rho = \rpi_{u_1} \rpi_{u_2} \cdots \widehat{\rpi}_{u_r} \cdots \rpi_{u_l} \rpi_{v_1} \rpi_{v_2} \cdots \rpi_{v_m} \rpi_{j_0}.$
Let $\frakm(\rpi_i)=0$ if $\rpi_i=\pi_i$ and $1$ if $\rpi_i=\opi_i$.
Note that, while applying the relations in~\eqref{eq: braid relation} and~\eqref{eq: commutation relation} to $\pi_\sigma \opi_\rho$,
the values of $\rpi_j$'s under $\frakm$ are staying unchanged as seen in the following figure:
\begin{displaymath}
\begin{array}{ccc}
\text{the relations in \eqref{eq: braid relation}:}
&
\begin{array}{c}
\begin{tikzpicture}[xscale=0.8,yscale=0.8]
\def \vv{1.5cm}
\def \hh{1cm}
\draw[<->] (0.1, 1.2) -- (\hh - 0.1, 0.3);
\draw[<->] (0.1 + \hh, 1.2) -- (2*\hh - 1mm, 0.3);
\draw[<->] (2*\hh - 0.1, 1.2) -- (0.1, 0.3);

\node[] at (0,\vv) {$\pi_i$};
\node[] at (\hh,\vv) {$\pi_{i+1}$};
\node[] at (2*\hh,\vv) {$\pi_i$};

\node[] at (0,0) {$\pi_{i+1}$};
\node[] at (\hh,0) {$\pi_i$};
\node[] at (2*\hh,0) {$\pi_{i+1}$};
\end{tikzpicture}
 \quad
\begin{tikzpicture}[xscale=0.9,yscale=0.9]
\def \vv{1.5cm}
\def \hh{1cm}
\draw[<->] (0.1, 1.2) -- (\hh - 0.1, 0.3);
\draw[<->] (0.1 + \hh, 1.2) -- (2*\hh - 1mm, 0.3);
\draw[<->] (2*\hh - 0.1, 1.2) -- (0.1, 0.3);

\node[] at (0,\vv) {$\pi_i$};
\node[] at (\hh,\vv) {$\pi_{i+1}$};
\node[] at (2*\hh,\vv) {$\opi_i$};

\node[] at (0,0) {$\opi_{i+1}$};
\node[] at (\hh,0) {$\pi_i$};
\node[] at (2*\hh,0) {$\pi_{i+1}$};
\end{tikzpicture}
 \quad
\begin{tikzpicture}[xscale=0.9,yscale=0.9]
\def \vv{1.5cm}
\def \hh{1cm}
\draw[<->] (0.1, 1.2) -- (\hh - 0.1, 0.3);
\draw[<->] (0.1 + \hh, 1.2) -- (2*\hh - 1mm, 0.3);
\draw[<->] (2*\hh - 0.1, 1.2) -- (0.1, 0.3);

\node[] at (0,\vv) {$\pi_{i+1}$};
\node[] at (\hh,\vv) {$\pi_{i}$};
\node[] at (2*\hh,\vv) {$\opi_{i+1}$};

\node[] at (0,0) {$\opi_{i}$};
\node[] at (\hh,0) {$\pi_{i+1}$};
\node[] at (2*\hh,0) {$\pi_{i}$};
\end{tikzpicture}
\\
\begin{tikzpicture}[xscale=0.9,yscale=0.9]
\def \vv{1.5cm}
\def \hh{1cm}
\draw[<->] (0.1, 0.3) -- (\hh - 0.1, 1.2);
\draw[<->] (0.1 + \hh, 0.3) -- (2*\hh - 1mm, 1.2);
\draw[<->] (2*\hh - 0.1, 0.3) -- (0.1, 1.2);

\node[] at (0,\vv) {$\opi_i$};
\node[] at (\hh,\vv) {$\opi_{i+1}$};
\node[] at (2*\hh,\vv) {$\opi_i$};

\node[] at (0,0) {$\opi_{i+1}$};
\node[] at (\hh,0) {$\opi_i$};
\node[] at (2*\hh,0) {$\opi_{i+1}$};
\end{tikzpicture}
 \quad
\begin{tikzpicture}[xscale=0.9,yscale=0.9]
\def \vv{1.5cm}
\def \hh{1cm}
\draw[<->] (0.1, 0.3) -- (\hh - 0.1, 1.2);
\draw[<->] (0.1 + \hh, 0.3) -- (2*\hh - 1mm, 1.2);
\draw[<->] (2*\hh - 0.1, 0.3) -- (0.1, 1.2);

\node[] at (0,\vv) {$\pi_i$};
\node[] at (\hh,\vv) {$\opi_{i+1}$};
\node[] at (2*\hh,\vv) {$\opi_i$};

\node[] at (0,0) {$\opi_{i+1}$};
\node[] at (\hh,0) {$\opi_i$};
\node[] at (2*\hh,0) {$\pi_{i+1}$};
\end{tikzpicture}
 \quad
\begin{tikzpicture}[xscale=0.9,yscale=0.9]
\def \vv{1.5cm}
\def \hh{1cm}
\draw[<->] (0.1, 0.3) -- (\hh - 0.1, 1.2);
\draw[<->] (0.1 + \hh, 0.3) -- (2*\hh - 1mm, 1.2);
\draw[<->] (2*\hh - 0.1, 0.3) -- (0.1, 1.2);

\node[] at (0,\vv) {$\pi_{i+1}$};
\node[] at (\hh,\vv) {$\opi_{i}$};
\node[] at (2*\hh,\vv) {$\opi_{i+1}$};

\node[] at (0,0) {$\opi_{i}$};
\node[] at (\hh,0) {$\opi_{i+1}$};
\node[] at (2*\hh,0) {$\pi_{i}$};
\end{tikzpicture}
\end{array}
\\[13ex]
\text{the relations in \eqref{eq: commutation relation}:}
&
\begin{array}{c}
\begin{tikzpicture}[xscale=0.8,yscale=0.8]
\def \vv{1.5cm}
\def \hh{1cm}
\draw[<->] (0.1, 0.3) -- (\hh - 0.1, 1.2);
\draw[<->] (\hh - 0.1, 0.3) -- (0.1, 1.2);

\node[] at (0,\vv) {$\pi_i$};
\node[] at (\hh,\vv) {$\pi_{j}$};

\node[] at (0,0) {$\pi_{j}$};
\node[] at (\hh,0) {$\pi_i$};
\end{tikzpicture}
 \quad
\begin{tikzpicture}[xscale=0.9,yscale=0.9]
\def \vv{1.5cm}
\def \hh{1cm}
\draw[<->] (0.1, 0.3) -- (\hh - 0.1, 1.2);
\draw[<->] (\hh - 0.1, 0.3) -- (0.1, 1.2);

\node[] at (0,\vv) {$\pi_i$};
\node[] at (\hh,\vv) {$\opi_{j}$};

\node[] at (0,0) {$\opi_{j}$};
\node[] at (\hh,0) {$\pi_i$};
\end{tikzpicture}
 \quad
\begin{tikzpicture}[xscale=0.9,yscale=0.9]
\def \vv{1.5cm}
\def \hh{1cm}
\draw[<->] (0.1, 0.3) -- (\hh - 0.1, 1.2);
\draw[<->] (\hh - 0.1, 0.3) -- (0.1, 1.2);

\node[] at (0,\vv) {$\opi_i$};
\node[] at (\hh,\vv) {$\pi_{j}$};

\node[] at (0,0) {$\pi_{j}$};
\node[] at (\hh,0) {$\opi_i$};
\end{tikzpicture}
 \quad
\begin{tikzpicture}[xscale=0.9,yscale=0.9]
\def \vv{1.5cm}
\def \hh{1cm}
\draw[<->] (0.1, 0.3) -- (\hh - 0.1, 1.2);
\draw[<->] (\hh - 0.1, 0.3) -- (0.1, 1.2);

\node[] at (0,\vv) {$\opi_i$};
\node[] at (\hh,\vv) {$\opi_{j}$};

\node[] at (0,0) {$\opi_{j}$};
\node[] at (\hh,0) {$\opi_i$};
\end{tikzpicture}
\end{array}
\end{array}
\end{displaymath}
Thus, we conclude that $\pi_{\sigma} \opi_{\rho} = \pi_{u_1} \pi_{u_2} \cdots \widehat{\pi}_{u_r} \cdots \pi_{u_l} \opi_{v_1} \opi_{v_2} \cdots \opi_{v_m} \pi_{j_0}$.
\end{proof}

Given $h = \sum_{\sigma \in \SG_n} c_\sigma \pi_\sigma \in H_n(0)$, let us say that \emph{$\pi_\sigma$ appears at $h$} if $c_\sigma \neq 0$.
The following lemma plays a key role in describing the $H_n(0)$-action on $H_n(0) \pi_\sigma \opi_\rho$.

\begin{lemma}\label{lem: zero characterization}
For any $\sigma, \rho \in \SG_n$, the following hold. 
\begin{enumerate}[label = {\rm (\arabic*)}]
\item $\pi_\sigma \opi_\rho$ is nonzero if and only if $\ell(\sigma\rho) = \ell(\sigma) + \ell(\rho)$.
\item If $\ell(\sigma\rho) = \ell(\sigma) + \ell(\rho)$, then
$\sigma\rho$ is a unique element of maximal length in
$\{\gamma \in \SG_n \mid \text{$\pi_\gamma$ appears at $\pi_\sigma\opi_\rho$}\}$.
\item $\ell(\sigma\rho) = \ell(\sigma) + \ell(\rho)$ if and only if $\sigma \preceq_L w_0\rho^{-1}$.
\end{enumerate}
\end{lemma}

\begin{proof}
Let $s_{u_1}s_{u_2} \cdots s_{u_l}$ be a reduced expression of $\sigma$ and $s_{v_1}s_{v_2} \cdots s_{v_m}$ a reduced expression of $\rho$.

Let us prove the ``if'' part of (1) and (2) simultaneously.
For each $\gamma \in \SG_n$, let $c_\gamma$ be the integer defined by
\begin{align*}
\pi_\sigma \opi_\rho 
& = \pi_\sigma (\pi_{v_1} - 1) (\pi_{v_2} - 1) \cdots (\pi_{v_m} - 1)
= \pi_{\sigma\rho} + \sum_{\gamma \in \SG_n} c_\gamma \pi_{\gamma}.
\end{align*}
It is clear that $c_\gamma = 0$ for all $\gamma \in \SG_n$ with $\ell(\gamma) \ge \ell(\sigma\rho)$. Thus, $\sigma\rho$ is a unique element of maximal length in
$\{\gamma \in \SG_n \mid \text{$\pi_\gamma$ appears at $\pi_\sigma\opi_\rho$}\}$ and $\pi_\sigma \opi_\rho \neq 0$.

For the ``only if'' part of (1), we prove that if $\ell(\sigma\rho) < \ell(\sigma) + \ell(\rho)$, then $\pi_\sigma \opi_\rho = 0$.
Since $\ell(\sigma\rho) < \ell(\sigma) + \ell(\rho)$, there exists $1 \le m' \le m$ such that 
$$
\ell(\sigma s_{v_1} s_{v_2} \cdots s_{v_{m'}}) < \ell(\sigma s_{v_1} s_{v_2} \cdots s_{v_{m'-1}}),
$$ 
that is, $v_{m'} \in \Des_R(\sigma s_{v_1} s_{v_2} \cdots s_{v_{m'-1}})$.
Let $\rho' = s_{v_1}s_{v_2} \cdots s_{v_{m'-1}}$.
Since $s_{v_1}s_{v_2} \cdots s_{v_m}$ is a reduced expression of $\rho$, we have that 
$v_{m'} \in \Des_R(\sigma\rho') \setminus \Des_R(\rho')$. 
By Lemma~\ref{lem: Des and last pi_i}, there exists a word $\rmw = w_1 w_2 \cdots w_{l+m'-1}$ such that 
$\pi_\sigma \opi_{\rho'} = \iota(\rmw)$ and $w_{l+m'-1} = \pi_{v_{m'}}$.
But, since $\pi_i \opi_i = 0$ for all $1 \le i \le n-1$, this implies that
\begin{align*}
\pi_\sigma \opi_{\rho} 
& = \pi_\sigma \opi_{\rho'} \opi_{v_{m'}} \opi_{v_{m' + 1}} \cdots \opi_{v_{m}}  = \rpi_{i_1} \rpi_{i_2} \cdots \rpi_{i_{l+m'-2}} \pi_{v_{m'}} \opi_{v_{m'}} \opi_{v_{m' + 1}} \cdots \opi_{v_{m}} 
= 0.    
\end{align*}

To prove (3), note that for any $\gamma \in \SG_n$, $\ell(w_0 \gamma) = \ell(w_0) - \ell(\gamma)$ and $\ell(\gamma^{-1}) = \ell(\gamma)$.
This implies that $\ell(\sigma \rho) = \ell(\sigma) + \ell(\rho)$ if and only if $\ell(w_0\rho^{-1}\sigma^{-1}) = \ell(w_0\rho^{-1}) - \ell(\sigma)$.
Now our assertion is obvious from $(w_0\rho^{-1}\sigma^{-1})\sigma = w_0\rho^{-1}$.
\end{proof}

Now we are ready to state the main theorem of this subsection.

\begin{theorem}\label{thm: embedding}
Let $\sigma, \rho \in \SG_n$. 
\begin{enumerate}[label = {\rm (\arabic*)}]
\item The set 
$\basisI(\sigma,\rho) := \{ \pi_\gamma \opi_\rho \mid \gamma \in [\sigma, w_0\rho^{-1}]_L \}$ forms a $\mathbb C$-basis for $H_n(0)\pi_\sigma \opi_\rho$.
\item For any $\pi_{\gamma} \opi_{\rho} \in \basisI(\sigma,\rho)$ and $1 \le i \le n-1$,
\begin{align*}
\pi_i \cdot (\pi_\gamma \opi_\rho) 
& = \begin{cases}
\pi_{\gamma} \opi_{\rho}  & \text{if $i \in \Des_L(\gamma)$,} \\
\pi_{s_i\gamma} \opi_{\rho} & \text{if $i \notin \Des_L(\gamma)$.}
\end{cases} 
\end{align*}
Moreover, if $i \notin \Des_L(\gamma)$, then $\pi_{s_i\gamma} \opi_{\rho} = 0$ if and only if $i \in \Des_L(\gamma\rho)$.
\item
The linear map $\sfem:\sfB(\sigma,\rho) \ra H_n(0) \pi_\sigma \opi_{\rho^{-1}w_0}$ defined by
\[
\gamma \mapsto \pi_{\gamma}\opi_{\rho^{-1}w_0} \quad \text{for $\gamma \in [\sigma, \rho]_L$}
\]
is an $H_n(0)$-module isomorphism.
\end{enumerate}
\end{theorem}

\begin{proof}
The assertion (2) follows from the definition of $\Des_L(\gamma)$ and Lemma~\ref{lem: zero characterization}~(1).

For the assertion (1), we first observe that $H_n(0)\pi_\sigma\opi_\rho$ is spanned by $\{ \pi_\gamma\opi_\rho \mid \gamma \in [\sigma, w_0]_L \}$.
But, Lemma~\ref{lem: zero characterization}~(1) and Lemma~\ref{lem: zero characterization}~(3) ensure that  
$\pi_\gamma\opi_\rho = 0$ unless $\gamma \preceq_L w_0\rho^{-1}$.
To prove that $\basisI(\sigma,\rho)$ is linearly independent, we suppose that $\sum_{\gamma \in [\sigma, w_0\rho^{-1}]_L} c_\gamma \pi_\gamma \opi_\rho = 0$, but not all coefficients are zero.
Let $A:=\{\gamma \in [\sigma, w_0\rho^{-1}]_L \mid c_\gamma \neq 0 \}$ and 
choose a permutation $\gamma_0 \in A$ which is of maximal length in $A$.
Again, Lemma~\ref{lem: zero characterization}~(1) and Lemma~\ref{lem: zero characterization}~(3) ensure that $\pi_\gamma\opi_\rho$ is nonzero for all $\gamma \preceq_L  w_0\rho^{-1}$. 
Combining this fact with Lemma~\ref{lem: zero characterization}~(2) shows that $\pi_{\gamma_0\rho}$ cannot appear at $\pi_{\gamma'} \opi_\rho$ for any $\gamma' \in A \setminus \{\gamma_0\}$. 
It means that $\pi_{\gamma_0\rho}$ appears at $\sum_{\gamma \in [\sigma, w_0\rho^{-1}]_L} c_\gamma \pi_\gamma \opi_\rho = 0$, which is absurd.

Finally, let us prove the assertion (3).
Suppose that $\gamma \in [\sigma, \rho]_L$ and $i \notin \Des_L(\gamma)$.
For our purpose, by virtue of (1) and (2), we have only to prove that 
$s_i \gamma \in [\sigma,\rho]_L$ is equivalent to $i \notin \Des_L(\gamma \rho^{-1} w_0)$.
Let us first show that $s_i \gamma \in [\sigma,\rho]_L$ implies that $i \notin \Des_L(\gamma \rho^{-1} w_0)$.
Since $s_i \gamma \in [\sigma,\rho]_L$, there exists $\xi \in \SG_n$ such that
\begin{align}\label{eq: rho complement}
\rho = \xi s_i\gamma \quad \text{and} \quad \ell(\rho) = \ell(\xi) + \ell(s_i \gamma).
\end{align}
From the second equality in~\eqref{eq: rho complement} together with the assumption $i \notin \Des_L(\gamma)$, we deduce that   
$\ell(\xi) + \ell(s_i \gamma) = \ell(\xi) + \ell(s_i) + \ell(\gamma)$.
Applying this to the first equality in~\eqref{eq: rho complement} says 
that $i \notin \Des_R(\xi)$, equivalently, $i \in \Des_L(\xi^{-1}w_0)$ or equally $i \notin \Des_L(s_i\xi^{-1}w_0)$. 
Now, $i \notin \Des_L(\gamma \rho^{-1} w_0)$ is obvious since 
$\gamma \rho^{-1} w_0 = s_i \xi^{-1} w_0$ by the first equality in~\eqref{eq: rho complement}.
Next, let us show that $i\notin \Des_L(\gamma \rho^{-1} w_0)$ implies that $s_i \gamma \in [\sigma, \rho]_L$.
The assumption $i \notin \Des_L(\gamma\rho^{-1} w_0)$ says that 
there exists a permutation $\xi \in \SG_n$ such that $\xi (s_i \gamma) \rho^{-1} w_0 = w_0$ and $\ell(w_0) = \ell(\xi) + \ell(s_i \gamma) + \ell(\rho w_0)$,
thus $s_i \gamma \preceq_L \rho$.
\end{proof}

In Subsection \ref{subsec: 0-Hecke alg and QSym}, we introduced projective indecomposable modules 
$\calP_\alpha$ and irreducible modules $\bfF_\alpha$. To be precise, $\calP_\alpha = H_n(0) \pi_{w_0(\alpha^\rmc)} \opi_{w_0(\alpha)}$ 
and $\bfF_\alpha$ is isomorphic to $\rmtop(\calP_\alpha)$ and one-dimensional.
By applying Theorem~\ref{thm: embedding}, we derive the following isomorphisms of $H_n(0)$-modules: 
\begin{align*}
\calP_\alpha \cong \sfB(w_0(\alpha^\rmc), w_0 w_0(\alpha))
\quad \text{and} \quad
\bfF_\alpha \cong \sfB(w_0(\alpha^\rmc), w_0(\alpha^\rmc)).
\end{align*}

\begin{example}
Since $\rho^{-1} w_0 = 52314 \in \SG_5$ with $\rho = 24315$, by Theorem~\ref{thm: embedding}, we have
an $H_5(0)$-module isomorphism $\sfB(14325,24315) \ra H_5(0) \pi_{14325} \opi_{52314}$.
This is illustrated in the following figure:
\[
\def \hp{0.15}
\def \vp{0.3}
\def \hhp{4.15}
\def \hhhp{4}
\begin{tikzpicture}[baseline = 0mm]
\node[below] at (0.1,-3.6) {$\sfB(14325,24315)$};

\node at (0,0+2*\vp) {
$14325$
};
\node at (0.7-\hp,0.25+1.25*\vp) {} edge [out=40,in=320, loop] ();
\node at (1.4-\hp,0.7+1.25*\vp) {\small $\pi_2,\pi_3$};

\node at (1.5,-1.25+3*\vp) {$0$};
\draw [->] (0.5,-0.8+3*\vp) -- (1.25,-1.15+3*\vp);
\node at (1,-0.75+3*\vp) {\small $\pi_4$};

\draw [->] (0,-0.8+3*\vp) -- (0,-1.75+3*\vp);
\node[right] at (0.,-1.25+3*\vp) {\small $\pi_1$};

\node at (0,-2.5+4*\vp) {
$24315$
};
\node at (0.7-\hp,-2.75+4.75*\vp) {} edge [out=40,in=320, loop] ();
\node at (1.3-0*\hp,-2.3+4.75*\vp) {\small $\pi_1,\pi_3$};

\draw [->] (0,-3.0 +4*\vp) -- (0,-3.95+4*\vp);
\node[right] at (0,-3.5+4*\vp) {\small $\pi_2, \pi_4$};
\node at (0,-4.6+5*\vp) {$0$};

\node at (2.3,-2.5+4*\vp) {$\cong$};

\node[below] at (0.1 + \hhp,-3.6) {$H_5(0) \pi_{14325} \opi_{52314}$};

\node at (0 + \hhp,0+2*\vp) {
$\pi_{14325} \opi_{52314}$
};
\node at (1.2-\hp + \hhp, 0.25+1.25*\vp) {} edge [out=40,in=320, loop] ();
\node at (1.9-\hp + \hhp, 0.7+1.25*\vp) {\small $\pi_2,\pi_3$};

\node at (1.5 + \hhp, -1.25+3*\vp) {$0$};
\draw [->] (0.5 + \hhp, -0.8+3*\vp) -- (1.25 + \hhp,-1.15+3*\vp);
\node at (1 + \hhp, -0.75+3*\vp) {\small $\pi_4$};

\draw [->] (0 + \hhp, -0.8+3*\vp) -- (0 + \hhp, -1.75+3*\vp);
\node[right] at (0 + \hhp, -1.25+3*\vp) {\small $\pi_1$};

\node at (0 + \hhp, -2.5+4*\vp) {
$\pi_{24315} \opi_{52314}$
};
\node at (1.2-\hp + \hhp, -2.75+4.75*\vp) {} edge [out=40,in=320, loop] ();
\node at (1.9-0*\hp + \hhp, -2.3+4.75*\vp) {\small $\pi_1,\pi_3$};

\draw [->] (0 + \hhp, -3.0 +4*\vp) -- (0 + \hhp,-3.95+4*\vp);
\node[right] at (0 + \hhp, -3.5+4*\vp) {\small $\pi_2, \pi_4$};
\node at (0 + \hhp,-4.6+5*\vp) {$0$};
\end{tikzpicture}
\]
\end{example}

\begin{remark}\label{rem: Specht}
It is remarked in~\cite[Definition 4.3]{02DHT} that all the specializations $q = 0$ of the Specht modules of $H_n(q)$ are of the form $H_n(0) \opi_{\xi} \pi_{w_0(\alpha^\rmc)}$ for some $\alpha \models n$ and $\xi \in [w_0(\alpha), w_0 w_0(\alpha^\rmc)]_L$.
On the other hand, following the way as in Theorem~\ref{thm: embedding}, we can deduce that the $\C$-linear map 
$$
\overline{\sfem}: \osfB(\sigma,\rho) \ra H_n(0) \opi_\sigma \pi_{\rho^{-1}w_0}, \quad \gamma \mapsto \opi_\gamma \pi_{\rho^{-1} w_0} \quad \text{for $\gamma \in [\sigma,\rho]_L$}
$$ 
is an $H_n(0)$-module isomorphism.
Combining these results, we see that all the specializations $q = 0$ of the Specht modules of $H_n(q)$ appear as $\osfB(\xi, w_0 w_0(\alpha^\rmc))$ for some $\alpha\models n$ and $\xi \in [w_0(\alpha), w_0 w_0(\alpha^\rmc)]_L$.
\end{remark}

\subsection{Restriction and Mackey formula}\label{subsec: restriction and Mackey formula}
Throughout this subsection, $m$ and $n$ denote positive integers. 

Hivert, Novelli, and Thibon~\cite{06HNT} presented a formula on induction product of semi-combinatorial modules associated with Yang-Baxter intervals.
Using this formula, one can derive the following lemma. 

\begin{lemma}\label{lem: tensor product} {\rm (cf. \cite[Theorem 3.8]{06HNT})}
For any $\sigma, \rho \in \SG_m$ and $\sigma', \rho' \in \SG_n$, 
we have 
$$
\sfB(\sigma,\rho) \boxtimes \sfB(\sigma',\rho') 
\cong \sfB(\sigma \conc \sigma', ~ \rho \ostar \rho').
\footnote{
Following the notation in~\cite{06HNT}, $\sigma \conc \sigma' = \sigma \cdot \sigma'[m]$ and $\rho \ostar \rho' = \rho[n] \cdot \rho'$.
And, there is a typo in~\cite[Theorem 3.8]{06HNT}, where $\beta = \beta''[k]\cdot \beta'$ should appear as $\beta = \beta'[n-k] \cdot \beta''$.
}
$$
Here,
\begin{align*}
(\sigma \conc \sigma')(i) 
& := \begin{cases}
\sigma(i) & \text{if $1 \le i \le m$,} \\
\sigma'(i-m) + m & \text{if $m+1 \le i \le m+n$,}
\end{cases}
\end{align*}
and
\begin{align*}
(\rho \ostar \rho')(i)
& := \begin{cases}
\rho(i) + n & \text{if $1 \le i \le m$,} \\
\rho'(i-m) & \text{if $m+1 \le i \le m+n$.} \\
\end{cases}
\end{align*}
\end{lemma}

\begin{remark}
For $\sigma \in \SG_m$ and $\sigma' \in \SG_n$, let $\sigma \shuffle \sigma'$ be the set of permutations $\gamma \in \SG_{m+n}$ satisfying that $\sigma(1)\sigma(2)\cdots\sigma(m)$ and $(\sigma'(1)+m) (\sigma'(2)+m) \cdots (\sigma'(n)+m)$ are subwords of $\gamma(1) \gamma(2) \cdots \gamma(m+n)$ and $\sigma \tshuffle \sigma' := \{\gamma^{-1} \mid  \gamma \in \sigma^{-1} \, \shuffle \, \sigma'^{-1}\}$.
For $X \subseteq \SG_m$ and $Y \subseteq \SG_n$, let
\[
X \tshuffle Y := \bigcup_{\sigma \in X, \, \sigma' \in Y} \sigma \tshuffle \sigma'.
\]
It is not difficult to show that
$[\sigma \conc \sigma', \rho \ostar \rho']_L = [\sigma,\rho]_L \tshuffle [\sigma',\rho']_L$ for $\sigma, \rho \in \SG_m$ and $\sigma', \rho' \in \SG_n$.
Therefore, Lemma~\ref{lem: tensor product} can be rewritten as
$\sfB(\sigma,\rho) \boxtimes \sfB(\sigma',\rho') 
\cong
\sfB([\sigma,\rho]_L \tshuffle [\sigma',\rho']_L)$.
In particular, 
$
\sfB(\sigma,\sigma) \boxtimes \sfB(\rho, \rho) 
\cong
\sfB(\sigma \tshuffle \rho)
$, which can be regarded as a lift of
\begin{align*}
F_\alpha F_\beta = \sum_{\gamma \in \sigma \shuffle \rho} F_{\comp(\Des_R(\gamma))}.
\end{align*}
Here, $\sigma \in \SG_{\ell(\alpha)}$ and $\rho \in \SG_{\ell(\beta)}$ are arbitrary permutations satisfying $\Des_R(\sigma) = \set(\alpha)$ and $\Des_R(\rho) = \set(\beta)$.
\end{remark}

In contrast to induction product, restriction of semi-combinatorial modules associated with Yang-Baxter intervals
has not yet been well studied except for the simple and projective indecomposable modules (see~\cite{16Huang}). 
\footnote{
In fact, in~\cite{16Huang}, the author considered the induction product and restriction of the simple and projective indecomposable modules over the 0-Hecke algebra of not only type $A$ but also type $B$ and $D$.}
The main purpose of this section is to provide an explicit restriction rule for weak Bruhat interval modules.
To begin with, let us collect necessary notations.
Let $\matr{[m+n]}{m}$ be the set of $m$-element subsets of $[m+n]$, on which 
$\SG_{m+n}$ acts in the natural way, that is, $\gamma' \cdot J:=\gamma'(J)$.
Given $\sigma, \rho \in \SG_{m+n}$ and $J \in \matr{[m+n]}{m}$, let
\begin{equation}
\begin{aligned}\label{J-permutation}
\frakI(\sigma, \rho; J) & := \{\gamma \in [\sigma,\rho]_L \mid \gamma^{-1}([1,m]) = J \} \ \ \text{and}\\[1ex]
\scrS_{\sigma, \rho}^{(m)} & := \left\{J \in \matr{[m+n]}{m} \; \middle| \; \frakI(\sigma, \rho; J) \ne \emptyset \right\}.
\end{aligned}
\end{equation}
For instance, $\frakI(2134, 4312; \tre{\{2,3\}})=\{3\tre{12}4, 4\tre{12}3, 3\tre{21}4, 4\tre{21}3\}$. 
A simple calculation shows that $\scrS_{2134,4312}^{(2)} = \{ \{1,2\}, \{2,3\}, \{3,4\} \}$.

When $J=\{j_1<j_2<\cdots<j_m\}$, write $[m+n] \setminus J$
as $\{ j^\rmc_1 < j^\rmc_2 < \cdots <j^\rmc_n\}$. Let $\rmperm_J$ and $\rmperm^J$ be the permutations in $\SG_{m+n}$ given by
\[
\left\{\begin{array}{l}
\rmperm_J(j_k) = k, \\ 
\rmperm_J(j^\rmc_s) = m+s,
\end{array}\right.
\ \  \text{and} \ \ 
\left\{\begin{array}{l}
\rmperm^J(j_k) = m-k+1, \\ 
\rmperm^J(j^\rmc_s) = m+n-s+1,
\end{array}\right.
\ \  (1\le k \le m,~ 1 \le s \le n),
\]
respectively.
For instance,
$\rmperm_{\{2,3\}}=3124$ and $\rmperm^{\{2,3\}} = 4213$.
Note that 
\begin{align}\label{eq: J-interval}
\frakI(\id, w_0; J) = [\rmperm_{J},\rmperm^{J}]_L \quad \text{for any $J \in \matr{[m+n]}{m}$,}
\end{align}
which will be used in the proof of Lemma~\ref{lem: frakI nonempty}.
For each $J \in \matr{[m+n]}{m}$, let 
\begin{align}\label{eq: sigJ rhoJ}
\bfsigJ := \rmperm_{\sigma \cdot J} \sigma
\quad \text{and} \quad
\bfrhoJ := \rmperm^{ (w_0\rho) \cdot J} w_0 \rho.
\end{align}
The following lemma shows that $[\sigma, \rho]_L=\bigcup\limits_{J \in \scrS_{\sigma, \rho}^{(m)}} [\bfsigJ, \bfrhoJ]_L$.

\begin{lemma}\label{lem: frakI nonempty}
Let $\sigma, \rho \in \SG_{m+n}$. For $J \in \scrS_{\sigma, \rho}^{(m)}$, we have
\begin{align*}
\frakI(\sigma, \rho; J) = [\bfsigJ, \bfrhoJ]_L.
\end{align*}
\end{lemma}

\begin{proof}
Let us fix $J \in \scrS_{\sigma, \rho}^{(m)}$.
It is well known that given $\gamma \in \SG_{m+n}$, if $\sigma \preceq_L \rho$ and $\sigma\gamma \preceq_L \rho\gamma$, then there exists an isomorphism $f_{\gamma}: [\sigma, \rho]_L \ra [\sigma \gamma, \rho \gamma]_L$ defined by $\xi \mapsto \xi \gamma$
(for instance, see~\cite[Proposition 3.1.6]{05BB}).
By the definition of $\frakI(\sigma, \rho; J)$ in~\eqref{J-permutation}, for all $\gamma \in \SG_{m+n}$, $\xi \in \frakI(\sigma, \rho; J)$ if and only if $\xi \gamma \in \frakI(\sigma\gamma, \rho\gamma; \gamma^{-1} \cdot J)$.
Indeed, the restriction
\begin{align}\label{eq: poset iso}
f_{\gamma}|_{\frakI(\sigma, \rho; J)} : \frakI(\sigma, \rho; J) \ra \frakI(\sigma\gamma, \rho\gamma; \gamma^{-1} \cdot J)
\end{align}
is an isomorphism of posets. 

Since $J \in \scrS_{\sigma,\rho}^{(m)}$, the set $f_{\sigma^{-1}}(\frakI(\sigma, \rho; J)) = \frakI(\id, \rho\sigma^{-1}; \sigma \cdot J)$ is nonempty.
Combining~\eqref{eq: J-interval} with the equality $\frakI(\id, \rho\sigma^{-1}; \sigma \cdot J)=[\id, \rho\sigma^{-1}]_L \cap \frakI(\id, w_0; \sigma \cdot J)$ yields that $\rmperm_{\sigma \cdot J} \preceq_L \xi$ for any $\xi \in \frakI(\id, \rho\sigma^{-1}; \sigma \cdot J)$.
Since $[\id, \xi]_L \subseteq [\id, \rho \sigma^{-1}]_L$,
$\rmperm_{\sigma \cdot J}$ is contained in $[\id, \rho \sigma^{-1}]_L$.
Therefore, $\rmperm_{\sigma \cdot J}  \preceq_L \xi$ for all $\xi \in \frakI(\id, \rho \sigma^{-1}; \sigma \cdot J)=[\id, \rho \sigma^{-1}]_L \cap \frakI(\id, w_0; \sigma \cdot J)$.
Via the isomorphism~\eqref{eq: poset iso}, we have $\bfsigJ \in \frakI(\sigma, \rho; J)$  and
$\bfsigJ 
\preceq_L \xi$ for all $\xi \in \frakI(\sigma, \rho; J)$.
In the same manner, 
one can prove that $\bfrhoJ \in \frakI(\sigma, \rho; J)$  and 
$\xi \preceq_L \bfrhoJ$ for all $\xi \in \frakI(\sigma, \rho; J)$.
 \end{proof}

By a careful reading of the proof of Lemma~\ref{lem: frakI nonempty}, one can derive that
\[
\scrS_{\sigma, \rho}^{(m)} = \left\{J \in \matr{[m+n]}{m} \; \middle| \; \rmperm_{\sigma \cdot J} \in [\id, \rho \sigma^{-1}]_L \right\}.
\]

\begin{example}
For $\sigma = 2134$ and $\rho = 4312$,  
\[
\scrS_{\sigma, \rho}^{(2)} = \{\{1,2\}, \{2,3\}, \{3,4\}\}.
\]
One can easily calculate that
\begin{align*}
\sigma_{\{1,2\}} = 2134, \quad
\sigma_{\{2,3\}} = 3124, \quad 
\sigma_{\{3,4\}} = 4312
\end{align*}
and
\begin{align*}
\rho^{\{1,2\}} = 2134, \quad
\rho^{\{2,3\}} = 4213, \quad
\rho^{\{3,4\}} = 4312.
\end{align*}
Thus, by Lemma~\ref{lem: frakI nonempty}~(2),
\begin{align*}
[2134, 4312]_L
& = [2134, 2134]_L \cup [3124, 4213]_L \cup [4312,4312]_L.
\end{align*}
Figure~\ref{fig: restriction} illustrates this partition. 
\end{example}

Given $\gamma \in \SG_{m+n}$, let $ \gamma^{-1}([1,m])=\{i_1 < i_2 < \cdots < i_m\}$ 
and $ \gamma^{-1}([m+1,m+n])=\{i'_1 < i'_2 < \cdots < i'_n\}$.
Let $\gamma_{\le m} \in \SG_m$ and $\gamma_{> m} \in \SG_n$ be the permutations given by
\begin{align}\label{eq: gamma >, <}
\gamma_{\le m}(j) = \gamma(i_j) \quad (1\le j \le m) 
\quad \text{and} \quad
\gamma_{> m}(j) = \gamma(i'_j) - m \quad  (1\le j \le n).
\end{align}
For instance, if $m=3$, $n=5$, and $\gamma = 58\textcolor{red}{32}6\textcolor{red}{1}47$, then $\gamma_{\le 3} = \textcolor{red}{321}$ and $\gamma_{> 3} = 25314$.

\begin{theorem}\label{thm: restriction}
For $\sigma, \rho \in \SG_{m+n}$, we have
\begin{align*}
\sfB(\sigma,\rho)\downarrow_{H_m(0) \otimes H_n(0)}^{H_{m+n}(0)} 
\hspace{1ex} \cong \hspace{-0.5ex} \bigoplus_{J\in \scrS_{\sigma, \rho}^{(m)} }
\sfB((\bfsigJ)_{\le m},  (\bfrhoJ)_{\le m}) \otimes \sfB((\bfsigJ)_{> m}, (\bfrhoJ)_{> m}).
\end{align*}
\end{theorem}

\begin{proof}
For each $J \in \scrS_{\sigma, \rho}^{(m)}$, we observe that 
$\frakI(\sigma, \rho; J) \cup \{0\}$ is closed under the $\pi_i$-action for $i \in [m + n - 1] \setminus \{m\}$. 
This means, by virtue of Lemma~\ref{lem: frakI nonempty}, that  
$\C \, [\bfsigJ, \bfrhoJ]_L$ is an $H_m(0) \otimes H_n(0)$-module.  
Since 
\[
[\sigma,\rho]_L = \bigcup_{J \in \scrS_{\sigma, \rho}^{(m)}} [\bfsigJ, \bfrhoJ]_L,
\]
it follows that 
\begin{align*}
\sfB(\sigma,\rho)\downarrow_{H_m(0) \otimes H_n(0)}^{H_{m+n}(0)} \ 
\hspace{0.5ex} \cong \hspace{-0.5ex} \bigoplus_{J\in \scrS_{\sigma, \rho}^{(m)}} \C \, [\bfsigJ, \bfrhoJ]_L \quad \text{(as $H_m(0) \otimes H_n(0)$-modules).}
\end{align*}
On the other hand, $[\bfsigJ, \bfrhoJ]_L$ is in bijection with $[(\bfsigJ)_{\le m},  (\bfrhoJ)_{\le m}]_L \times [(\bfsigJ)_{> m}, (\bfrhoJ)_{> m}]_L$ under $\gamma \mapsto (\gamma_{\le m}, \gamma_{>m})$, which again induces
an $H_m(0) \otimes H_n(0)$-module isomorphism 
$$\C \, [\bfsigJ, \bfrhoJ]_L \to \sfB((\bfsigJ)_{\le m},  (\bfrhoJ)_{\le m}) \otimes \sfB((\bfsigJ)_{> m}, (\bfrhoJ)_{> m}), \quad \gamma \mapsto (\gamma_{\le m}, \gamma_{>m}),$$
as required. 
\end{proof}

\begin{example}\label{ex: restriction}
Using Theorem~\ref{thm: restriction}, 
we derive that 
$\sfB(2134,4312)\downarrow^{H_4(0)}_{H_2(0)\otimes H_2(0)}$ is isomorphic to
\[
(\sfB(21,21) \otimes \sfB(12,12)) 
\oplus (\sfB(12,21) \otimes \sfB(12,21)) \oplus (\sfB(12,12) \otimes \sfB(21,21)).
\]
The difference between the $H_4(0)$-action and the $H_2(0) \otimes H_2(0)$-action on $\sfB(2134,4312)$
is well illustrated in Figure~\ref{fig: restriction}. 
\end{example}

\begin{figure}[t]
\centering
\def \hp{0.25}
\def \wp{0.2}
\def \wtab{1.95}
\def \htab{1.5}
\begin{tikzpicture}[baseline = 0mm, scale = 0.7]

\node[below] at (0,-0.7) {$\sfB(2134,4312)$};
% node
%level: 7
\node at (0,0) {\color{lightgray}$4321$};
%level: 6
\node at (-\wtab,1*\htab) {\color{lightgray}$4231$};
\node at (0,1*\htab) {\color{black}$4312$};
\node at (\wtab,1*\htab) {\color{lightgray}$3421$};
%level: 5
\node at (-\wtab*2,2*\htab) {\color{lightgray}$4132$};
\node at (-\wtab,2*\htab) {\color{black}$4213$};
\node at (0,2*\htab) {\color{lightgray}$3241$};
\node at (\wtab,2*\htab) {\color{lightgray}$3412$};
\node at (\wtab*2,2*\htab) {\color{lightgray}$2431$};
% level: 4
\node at (-\wtab/2*5,3*\htab) {\color{black}$4123$};
\node at (-\wtab/2*3,3*\htab) {\color{black}$3214$};
\node at (-\wtab/2,3*\htab) {\color{lightgray}$3142$};
\node at (\wtab/2,3*\htab) {\color{lightgray}$2413$};
\node at (\wtab/2*3,3*\htab) {\color{lightgray}$1432$};
\node at (\wtab/2*5,3*\htab) {\color{lightgray}$2341$};
% level: 3
\node at (-\wtab*2,4*\htab) {\color{black}$3124$};
\node at (-\wtab,4*\htab) {\color{lightgray}$2314$};
\node at (0,4*\htab) {\color{lightgray}$2143$};
\node at (\wtab,4*\htab) {\color{lightgray}$1423$};
\node at (\wtab*2,4*\htab) {\color{lightgray}$1342$};
%level: 2
\node at (-\wtab,5*\htab) {\color{black}$2134$};
\node at (0,5*\htab) {\color{lightgray}$1324$};
\node at (\wtab,5*\htab) {\color{lightgray}$1243$};
%level: 1
\node at (0,6*\htab) {\color{lightgray}$1234$};

%arrow
\draw [lightgray, ->] (-\wp,6*\htab-\hp) -- (-\wtab+\wp,5*\htab+\hp);
\draw [lightgray, ->] (0,6*\htab-\hp) -- (0,5*\htab+\hp);
\draw [lightgray, ->] (\wp, 6*\htab -\hp) -- (\wtab-\wp,5*\htab+\hp);

%1
\draw [black, ->] (-\wtab-\wp,5*\htab - \hp) -- (-\wtab*2+\wp,4*\htab+\hp);
\node at (-1.75*\wtab+\wp,4.8*\htab-\hp) {\scriptsize $\pi_2$};
\node at (-0.7*\wtab-\wp,5.25*\htab - \hp) {} edge [out=40,in=320, loop] ();
\node at (-0.15*\wtab-\wp,5.4*\htab - \hp) {\scriptsize $\pi_1$};
\draw [lightgray, ->] (\wp-\wtab,5*\htab - \hp) -- (0-\wp,4*\htab+\hp);
\node at (-0.44*\wtab, 4.45*\htab+\hp) {\color{lightgray}\scriptsize $\pi_3$};
%2
\draw [lightgray, ->] (-\wp,5*\htab - \hp) -- (-\wtab+\wp,4*\htab+\hp);
\draw [lightgray, ->] (\wp,5*\htab - \hp) -- (\wtab-\wp,4*\htab+\hp);
%3
\draw [lightgray, ->] (-\wp+\wtab,5*\htab - \hp) -- (0+\wp,4*\htab+\hp);
\draw [lightgray, ->] (\wp+\wtab,5*\htab - \hp) -- (2*\wtab-\wp,4*\htab+\hp);

%1
\draw [black, ->] (-\wtab*2-\wp,4*\htab-\hp) -- (-\wtab/2*5+\wp,3*\htab+\hp);
\draw [black, ->] (-\wtab*2+\wp,4*\htab-\hp) -- (-\wtab/2*3-\wp,3*\htab+\hp);
\node at (-\wtab*2.55+\wp,3.7*\htab-\hp) {\scriptsize $\pi_3$};
\node at (-\wtab*1.7+\wp,3.7*\htab-\hp) {\scriptsize $\pi_1$};
\node at (-1.7*\wtab-\wp,4.25*\htab - \hp) {} edge [out=40,in=320, loop] ();
\node at (-1.15*\wtab-\wp,4.4*\htab - \hp) {\scriptsize $\pi_2$};
%2
\draw [lightgray, ->] (-\wtab-\wp,4*\htab-\hp) -- (-\wtab/2*3+\wp,3*\htab+\hp);
\draw [lightgray, ->] (-\wtab+\wp,4*\htab-\hp) -- (\wtab/2-\wp,3*\htab+\hp);
%3
\draw [lightgray, ->] (-\wp,4*\htab-\hp) -- (-\wtab/2+\wp,3*\htab+\hp);
%4
\draw [lightgray, ->] (\wtab-\wp,4*\htab-\hp) -- (\wtab/2+\wp,3*\htab+\hp);
\draw [lightgray, ->] (\wtab+\wp,4*\htab-\hp) -- (\wtab/2*3-\wp,3*\htab+\hp);
%5
\draw [lightgray, ->] (2*\wtab-\wp,4*\htab-\hp) -- (\wtab/2*3+\wp,3*\htab+\hp);
\draw [lightgray, ->] (2*\wtab+\wp,4*\htab-\hp) -- (\wtab/2*5-\wp,3*\htab+\hp);

%1
\draw [lightgray, ->] (-\wtab/2*5+\wp,3*\htab-\hp) -- (-\wtab*2-\wp,2*\htab+\hp);
\node at (-\wtab*2.2+\wp,2.7*\htab-\hp) {\scriptsize \color{lightgray} $\pi_2$};
\draw [black, ->] (-\wtab*2.3+\wp,3*\htab-\hp) -- (-1.2*\wtab-\wp,2*\htab+\hp);
\node at (-\wtab*1.7+\wp,2.7*\htab-\hp) {\scriptsize $\pi_1$};
%2
%3
\draw [lightgray, ->] (-\wtab/2-\wp,3*\htab-\hp) -- (-\wtab*2+\wp,2*\htab+\hp);
\draw [lightgray, ->] (-\wtab/2+\wp,3*\htab-\hp) -- (-\wp,2*\htab+\hp);
%4
\draw [lightgray, ->] (\wtab/2+\wp,3*\htab-\hp) -- (\wtab-\wp,2*\htab+\hp);
%5
\draw [lightgray, ->] (\wtab/2*3+\wp,3*\htab-\hp) -- (\wtab*2-\wp,2*\htab+\hp);
%6
\draw [lightgray, ->] (\wtab/2*5-\wp,3*\htab-\hp) -- (\wp,2*\htab+\hp);
\draw [lightgray, ->] (\wtab/2*5,3*\htab-\hp) -- (\wtab*2+\wp,2*\htab+\hp);
%2
\draw [black, ->] (-\wtab/2*3+\wp,3*\htab-\hp) -- (-\wtab-\wp,2*\htab+\hp);
\node at (-\wtab*1.2+\wp,2.7*\htab-\hp) {\scriptsize $\pi_3$};
\node at (-1.2*\wtab-\wp,3.25*\htab - \hp) {} edge [out=40,in=320, loop] ();
\node at (-0.7*\wtab-\wp,3.55*\htab - \hp) {\scriptsize $\pi_1, \pi_2$};
\node at (-2.2*\wtab-\wp,3.25*\htab - \hp) {} edge [out=40,in=320, loop] ();
\node at (-1.8*\wtab-\wp,3.55*\htab - \hp) {\scriptsize $\pi_3$};

%1
\draw [lightgray, ->] (-\wtab*2+\wp,2*\htab-\hp) -- (-\wtab-\wp, \htab+\hp);
%2
%3
\draw [lightgray, ->] (-\wp,2*\htab-\hp) -- (-\wtab+\wp, \htab+\hp);
%4
\draw [lightgray, ->] (\wtab-\wp,2*\htab-\hp) -- (\wp, \htab+\hp);
\draw [lightgray, ->] (\wtab,2*\htab-\hp) -- (\wtab, \htab+\hp);
%5
\draw [lightgray, ->] (2*\wtab-\wp,2*\htab-\hp) -- (\wtab+\wp, \htab+\hp);
%2
\draw [black, ->] (-\wtab+\wp,2*\htab-\hp) -- (-\wp, \htab+\hp);
\node at (-0.55*\wtab+\wp,1.8*\htab-\hp) {\scriptsize $\pi_2$};
\node at (-0.7*\wtab-\wp,2.25*\htab - \hp) {} edge [out=40,in=320, loop] ();
\node at (-0.19*\wtab-\wp,2.55*\htab - \hp) {\scriptsize $\pi_1, \pi_3$};

%1
\draw [lightgray, ->] (-\wtab+\wp,\htab-\hp) -- (-\wp, \hp);
%2
\draw [lightgray, ->] (0,\htab-\hp) -- (0, \hp);
\node at (\wtab*0.05+\wp,0.7*\htab-\hp) {\scriptsize \color{lightgray} $\pi_1$};
%3
\draw [lightgray, ->] (\wtab-\wp,\htab-\hp) -- (\wp, \hp);

\node at (0.3*\wtab-\wp,1.25*\htab - \hp) {} edge [out=40,in=320, loop] ();
\node at (0.8*\wtab-\wp,1.55*\htab - \hp) {\scriptsize $\pi_2, \pi_3$};

\end{tikzpicture}
\begin{tikzpicture}[baseline = 0mm, scale = 0.7]
\node[below] at (0.7,-0.7) {$\sfB(2134,4312)\downarrow^{H_4(0)}_{H_2(0)\otimes H_2(0)}$};
% node
%level: 7
\node at (0,0) {\color{lightgray}$4321$};
%level: 6
\node at (-\wtab,1*\htab) {\color{lightgray}$4231$};
\node at (0,1*\htab) {\color{black}$4312$};
\node at (\wtab,1*\htab) {\color{lightgray}$3421$};
%level: 5
\node at (-\wtab*2,2*\htab) {\color{lightgray}$4132$};
\node at (-\wtab,2*\htab) {\color{black}$4213$};
\node at (0,2*\htab) {\color{lightgray}$3241$};
\node at (\wtab,2*\htab) {\color{lightgray}$3412$};
\node at (\wtab*2,2*\htab) {\color{lightgray}$2431$};
% level: 4
\node at (-\wtab/2*5,3*\htab) {\color{black}$4123$};
\node at (-\wtab/2*3,3*\htab) {\color{black}$3214$};
\node at (-\wtab/2,3*\htab) {\color{lightgray}$3142$};
\node at (\wtab/2,3*\htab) {\color{lightgray}$2413$};
\node at (\wtab/2*3,3*\htab) {\color{lightgray}$1432$};
\node at (\wtab/2*5,3*\htab) {\color{lightgray}$2341$};
% level: 3
\node at (-\wtab*2,4*\htab) {\color{black}$3124$};
\node at (-\wtab,4*\htab) {\color{lightgray}$2314$};
\node at (0,4*\htab) {\color{lightgray}$2143$};
\node at (\wtab,4*\htab) {\color{lightgray}$1423$};
\node at (\wtab*2,4*\htab) {\color{lightgray}$1342$};
%level: 2
\node at (-\wtab,5*\htab) {\color{black}$2134$};
\node at (0,5*\htab) {\color{lightgray}$1324$};
\node at (\wtab,5*\htab) {\color{lightgray}$1243$};
%level: 1
\node at (0,6*\htab) {\color{lightgray}$1234$};

%arrow
\draw [lightgray, ->] (-\wp,6*\htab-\hp) -- (-\wtab+\wp,5*\htab+\hp);
\draw [lightgray, ->] (0,6*\htab-\hp) -- (0,5*\htab+\hp);
\draw [lightgray, ->] (\wp, 6*\htab -\hp) -- (\wtab-\wp,5*\htab+\hp);
%1
\draw [lightgray, ->] (-\wtab-\wp,5*\htab - \hp) -- (-\wtab*2+\wp,4*\htab+\hp);
\node at (-0.7*\wtab-\wp,5.25*\htab - \hp) {} edge [out=40,in=320, loop] ();
\node at (-0.15*\wtab-\wp,5.4*\htab - \hp) {\scriptsize $\pi_1$};
\draw [lightgray, ->] (\wp-\wtab,5*\htab - \hp) -- (0-\wp,4*\htab+\hp);
\node at (-0.44*\wtab, 4.45*\htab+\hp) {\color{lightgray}\scriptsize $\pi_3$};
%2
\draw [lightgray, ->] (-\wp,5*\htab - \hp) -- (-\wtab+\wp,4*\htab+\hp);
\draw [lightgray, ->] (\wp,5*\htab - \hp) -- (\wtab-\wp,4*\htab+\hp);
%3
\draw [lightgray, ->] (-\wp+\wtab,5*\htab - \hp) -- (0+\wp,4*\htab+\hp);
\draw [lightgray, ->] (\wp+\wtab,5*\htab - \hp) -- (2*\wtab-\wp,4*\htab+\hp);
%1
\draw [black, ->] (-\wtab*2-\wp,4*\htab-\hp) -- (-\wtab/2*5+\wp,3*\htab+\hp);
\draw [black, ->] (-\wtab*2+\wp,4*\htab-\hp) -- (-\wtab/2*3-\wp,3*\htab+\hp);
\node at (-\wtab*2.55+\wp,3.7*\htab-\hp) {\scriptsize $\pi_3$};
\node at (-\wtab*1.7+\wp,3.7*\htab-\hp) {\scriptsize $\pi_1$};
%2
\draw [lightgray, ->] (-\wtab-\wp,4*\htab-\hp) -- (-\wtab/2*3+\wp,3*\htab+\hp);
\draw [lightgray, ->] (-\wtab+\wp,4*\htab-\hp) -- (\wtab/2-\wp,3*\htab+\hp);
%3
\draw [lightgray, ->] (-\wp,4*\htab-\hp) -- (-\wtab/2+\wp,3*\htab+\hp);
%4
\draw [lightgray, ->] (\wtab-\wp,4*\htab-\hp) -- (\wtab/2+\wp,3*\htab+\hp);
\draw [lightgray, ->] (\wtab+\wp,4*\htab-\hp) -- (\wtab/2*3-\wp,3*\htab+\hp);
%5
\draw [lightgray, ->] (2*\wtab-\wp,4*\htab-\hp) -- (\wtab/2*3+\wp,3*\htab+\hp);
\draw [lightgray, ->] (2*\wtab+\wp,4*\htab-\hp) -- (\wtab/2*5-\wp,3*\htab+\hp);
%1
\draw [lightgray, ->] (-\wtab/2*5+\wp,3*\htab-\hp) -- (-\wtab*2-\wp,2*\htab+\hp);
\draw [black, ->] (-\wtab*2.3+\wp,3*\htab-\hp) -- (-1.2*\wtab-\wp,2*\htab+\hp);
\node at (-\wtab*1.7+\wp,2.7*\htab-\hp) {\scriptsize $\pi_1$};
%2
%3
\draw [lightgray, ->] (-\wtab/2-\wp,3*\htab-\hp) -- (-\wtab*2+\wp,2*\htab+\hp);
\draw [lightgray, ->] (-\wtab/2+\wp,3*\htab-\hp) -- (-\wp,2*\htab+\hp);
%4
\draw [lightgray, ->] (\wtab/2+\wp,3*\htab-\hp) -- (\wtab-\wp,2*\htab+\hp);
%5
\draw [lightgray, ->] (\wtab/2*3+\wp,3*\htab-\hp) -- (\wtab*2-\wp,2*\htab+\hp);
%6
\draw [lightgray, ->] (\wtab/2*5-\wp,3*\htab-\hp) -- (\wp,2*\htab+\hp);
\draw [lightgray, ->] (\wtab/2*5,3*\htab-\hp) -- (\wtab*2+\wp,2*\htab+\hp);
%2
\draw [black, ->] (-\wtab/2*3+\wp,3*\htab-\hp) -- (-\wtab-\wp,2*\htab+\hp);
\node at (-\wtab*1.2+\wp,2.7*\htab-\hp) {\scriptsize $\pi_3$};
\node at (-1.2*\wtab-\wp,3.25*\htab - \hp) {} edge [out=40,in=320, loop] ();
\node at (-0.75*\wtab-\wp,3.55*\htab - \hp) {\scriptsize $\pi_1$};
\node at (-2.2*\wtab-\wp,3.25*\htab - \hp) {} edge [out=40,in=320, loop] ();
\node at (-1.8*\wtab-\wp,3.55*\htab - \hp) {\scriptsize $\pi_3$};
%1
\draw [lightgray, ->] (-\wtab*2+\wp,2*\htab-\hp) -- (-\wtab-\wp, \htab+\hp);
%2
%3
\draw [lightgray, ->] (-\wp,2*\htab-\hp) -- (-\wtab+\wp, \htab+\hp);
%4
\draw [lightgray, ->] (\wtab-\wp,2*\htab-\hp) -- (\wp, \htab+\hp);
\draw [lightgray, ->] (\wtab,2*\htab-\hp) -- (\wtab, \htab+\hp);
%5
\draw [lightgray, ->] (2*\wtab-\wp,2*\htab-\hp) -- (\wtab+\wp, \htab+\hp);
%2
\draw [lightgray, ->] (-\wtab+\wp,2*\htab-\hp) -- (-\wp, \htab+\hp);
\node at (-0.7*\wtab-\wp,2.25*\htab - \hp) {} edge [out=40,in=320, loop] ();
\node at (-0.19*\wtab-\wp,2.55*\htab - \hp) {\scriptsize $\pi_1, \pi_3$};
%1
\draw [lightgray, ->] (-\wtab+\wp,\htab-\hp) -- (-\wp, \hp);
%2
\draw [lightgray, ->] (0,\htab-\hp) -- (0, \hp);
\node at (\wtab*0.05+\wp,0.7*\htab-\hp) {\scriptsize \color{lightgray} $\pi_1$};
%3
\draw [lightgray, ->] (\wtab-\wp,\htab-\hp) -- (\wp, \hp);
\node at (0.3*\wtab-\wp,1.25*\htab - \hp) {} edge [out=40,in=320, loop] ();
\node at (0.8*\wtab-\wp,1.55*\htab - \hp) {\scriptsize $\pi_3$};
\end{tikzpicture}
\caption{$\sfB(2134,4312)$ and $\sfB(2134,4312)\downarrow^{H_4(0)}_{H_2(0)\otimes H_2(0)}$}
\label{fig: restriction}
\end{figure}
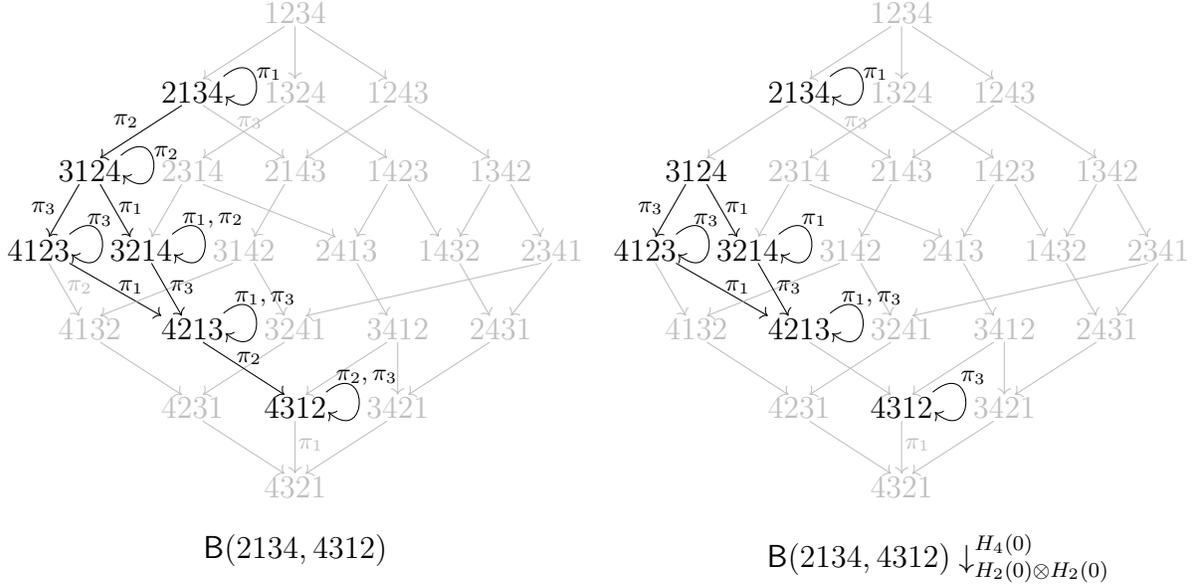

Next, let us deal with a Mackey formula for weak Bruhat interval modules.
Bergeron and Li~\cite[Subsection 3.1 (5)]{09BL}
provide a Mackey formula working on the Grothendieck ring
$\calG = \bigoplus_{n \ge 0} \calG_0(H_n(0))$ of $0$-Hecke algebras.
It says that for any $H_m(0)$-module $M$, $H_n(0)$-module $N$, and $k \in [1,m+n-1]$,
\begin{align*}
&[(M \boxtimes N)\downarrow^{H_{m+n}(0)}_{H_k(0) \otimes H_{m+n-k}}] \\
&=
\sum_{\substack{t + s = k \\ t \le m,~ s \le n }} 
\left[
\mathtt{T}_{t,s}\left(
M\downarrow^{H_m(0)}_{H_t(0) \otimes H_{m-t}(0)} \otimes \; N\downarrow^{H_n(0)}_{H_s(0) \otimes H_{n-s}(0)}
\right)
\uparrow^{H_k(0) \otimes H_{m+n-k}}_{H_t(0) \otimes H_s(0) \otimes H_{m-t}(0) \otimes H_{n-s}(0)}\right],
\end{align*}
where 
\[
\scalebox{0.94}{$
\mathtt{T}_{t,s}: \module (H_t(0) \otimes H_{m-t}(0) \otimes H_s(0) \otimes H_{n-s}(0))
\ra \module (H_t(0) \otimes H_s(0) \otimes H_{m-t}(0) \otimes H_{n-s}(0))$}
\]
is the functor sending $M_1 \otimes M_2 \otimes N_1 \otimes N_2 \mapsto M_1 \otimes N_1 \otimes M_2 \otimes N_2$.
On the other hand, by combining Lemma~\ref{lem: tensor product} with Theorem~\ref{thm: restriction}, we can derive a formula working weak Bruhat interval modules:
\begin{align}\label{eq: Mackey formula WBIM}
\begin{aligned}
& (\sfB(\sigma,\rho) \boxtimes \sfB(\sigma',\rho')) \downarrow_{H_k(0) \otimes H_{m+n - k}(0)}^{H_{m+n}(0)} \\
& \cong
\bigoplus_{J \in \scrS_{\sigma \conc \sigma', \rho \ostar \rho'}^{(k)}} 
\sfB\big(((\sigma \conc \sigma')_{\scalebox{0.55}{$J$}})_{\le k}, ((\rho \ostar \rho')^{\scalebox{0.55}{$J$}})_{\le k}\big) \otimes
\sfB\big(((\sigma \conc \sigma')_{\scalebox{0.55}{$J$}})_{> k}, ((\rho \ostar \rho')^{\scalebox{0.55}{$J$}})_{> k}\big).
\end{aligned}
\end{align}
Although it is very naive, one can expect 
that Bergeron and Li's Mackey formula lifts to our formula
at least for weak Bruhat interval modules.

To prove our result, we need the notion of standardization.
For $\sigma \in \SG_n$ and $1\le k' \le k \le n$, let $\tst(\sigma;[k',k])$ be a unique permutation in $\SG_{k-k'+1}$ 
satisfying the condition that $\tst(\sigma;[k',k])(i) < \tst(\sigma;[k',k])(j)$ if and only if $\sigma(k'+i-1) < \sigma(k'+j-1)$ for all $1\le i,j \le k-k'+1$.
For instance, if $\sigma = 25143 \in \SG_5$, then $\tst(\sigma;[2,4]) = 312 \in \SG_3$.
This standardization preserves the left weak Bruhat order on $\SG_n$, in other words,  
\begin{equation}\label{lem: preceq preserve}
\tst(\sigma;[k',k]) \preceq_{L} \tst(\rho; [k',k]) \text{ whenever }\sigma \preceq_L \rho.
\end{equation}
The following theorem shows that~\eqref{eq: Mackey formula WBIM} is a natural lift of Bergeron and Li's Mackey formula.

\begin{theorem}\label{thm: Mackey formula}
For $\sigma, \rho \in \SG_m$, $\sigma',\rho' \in \SG_n$, and $1 \le k < m+n$,
\begin{align*}
&(\sfB(\sigma,\rho) \boxtimes \sfB(\sigma',\rho')) \downarrow_{H_k(0) \otimes H_{m+n - k}(0)}^{H_{m+n}(0)}\\
&\cong \hspace{-1.1ex}
\bigoplus_{\substack{t + s = k \\ t \le m,~ s \le n }}
\hspace{-1.1ex}
\mathtt{T}_{t,s}\left(
\sfB(\sigma,\rho)\downarrow^{H_m(0)}_{H_t(0) \otimes H_{m-t}(0)} \otimes \;
\sfB(\sigma',\rho')\downarrow^{H_n(0)}_{H_s(0) \otimes H_{n-s}(0)}
\right)
\uparrow^{H_k(0) \otimes H_{m+n-k}}_{H_t(0) \otimes H_s(0) \otimes H_{m-t}(0) \otimes H_{n-s}(0)}.
\end{align*}
\end{theorem}

\begin{proof}
Using Lemma~\ref{lem: tensor product} and Theorem~\ref{thm: restriction}, we derive that 
\begin{align*}
&
\bigoplus_{\substack{t + s = k \\ t \le m,~ s \le n }}
\hspace{-1.1ex}
\mathtt{T}_{t,s}\left(
\sfB(\sigma,\rho)\downarrow^{H_m(0)}_{H_t(0) \otimes H_{m-t}(0)} \otimes \;
\sfB(\sigma',\rho')\downarrow^{H_n(0)}_{H_s(0) \otimes H_{n-s}(0)}
\right)
\uparrow^{H_k(0) \otimes H_{m+n-k}}_{H_t(0) \otimes H_s(0) \otimes H_{m-t}(0) \otimes H_{n-s}(0)}\\
&\cong
\bigoplus_{J_1,\, J_2}
\sfB\left(
(\sigma_{\sJo})_{\le t} \conc (\sigma'_{\sJt})_{\le s},
(\rho^{\sJo})_{\le t} \ostar ({\rho'}^{\sJt})_{\le s}
\right) 
\otimes
\sfB\left(
(\sigma_{\sJo})_{> t} \conc (\sigma'_{\sJt})_{> s},
(\rho^{\sJo})_{> t} \ostar ({\rho'}^{\sJt})_{> s}
\right),
\end{align*}
where the sum ranges over all pairs $(J_1,J_2)$ in
$$
\bigcup_{\substack{t + s = k \\ t \le m,~ s \le n }} \scrS_{\sigma,\rho}^{(t)} \times \scrS_{\sigma',\rho'}^{(s)}.
$$
Let $f: \scrS_{\sigma \conc \sigma', \rho \ostar \rho'}^{(k)} \ra \bigcup_{\substack{t + s = k \\ t \le m,~ s \le n }} \scrS_{\sigma,\rho}^{(t)} \times \scrS_{\sigma',\rho'}^{(s)}$ be a map defined by
\[
J \mapsto \left(J \cap [1,m],\; \{i \mid i+m\in J \cap [m+1,m+n]\}\right)
\]
for all $J \in \scrS_{\sigma \conc \sigma', \rho \ostar \rho'}^{(k)}$.
To begin with, let us verify that $f$ is a well-defined bijection.

First, we prove that $f$ is well-defined.
Given $J \in \scrS_{\sigma \conc \sigma', \rho \ostar \rho'}^{(k)}$, 
let 
\begin{align}\label{eq: J1, J2}
J_1 = J \cap [1,m],
\ 
J_2 = \{i \mid i + m \in J \cap [m+1,m+n]\},
\ 
t = |J_1|, 
\  \text{and} \ 
s = |J_2|.
\end{align}
Since $J \in \scrS_{\sigma \conc \sigma', \rho \ostar \rho'}^{(k)}$,
there exists a permutation $\delta \in [\sigma \conc \sigma', \rho \ostar \rho']_L$ such that $\delta(J) = [1,k]$.
Combining~\eqref{lem: preceq preserve} with the definition of $\sigma \conc \sigma'$ and $\rho \ostar \rho'$ yields that 
\[
\sigma \preceq_L \tst(\delta; [1,m]) \preceq_L \rho
\quad \text{and} \quad 
\sigma' \preceq_L \tst(\delta; [m + 1, m + n]) \preceq_L \rho'.
\]
Note that
\[
\tst(\delta;[1,m]) (J_1) = [1,t] 
\quad \text{and} \quad
\tst(\delta;[m+1,m+n]) (J_2) = [1,s].
\]
This tells us that $J_1 \in \scrS_{\sigma,\rho}^{(t)}$ and $ J_2 \in \scrS_{\sigma',\rho'}^{(s)}$, thus $f$ is well-defined.

Next, we prove that $f$ is bijective by constructing its inverse.
Let $t'$ and $s'$ be nonnegative integers satisfying $t' + s' = k$, $t' \le m$, and $s' \le n$.
Given $K_1 \in \scrS_{\sigma,\rho}^{(t')}$ and $K_2 \in \scrS_{\sigma',\rho'}^{(s')}$, 
consider the mapping
$(K_1, K_2) \mapsto K := K_1 \cup \{i \mid i-m \in K_2\}.$
Note that there exist permutations $\gamma_1 \in [\sigma,\rho]_L$ and $\gamma_2 \in [\sigma', \rho']_L$ such that
\[
\gamma_1(K_1) = [1,t'] 
\quad \text{and} \quad 
\gamma_2(K_2) = [1,s']. 
\]
Let $\delta \in \SG_{m+n}$ given by
\[
\delta(i) = \begin{cases}
\gamma_1(i) & \text{if $i \in K_1$},\\
\gamma_1(i) + t' + s' & \text{if $i \in [1,m] \setminus K_1$},\\
\gamma_2(i-m) + t' & \text{if $i \in K'_2$},\\
\gamma_2(i-m) + m + s' & \text{if $i \in [m+1, m+n] \setminus K'_2$},
\end{cases}
\]
where $K'_2 := \{j \mid j-m \in K_2\}$.
Then $\delta \in \gamma_1 \tshuffle \gamma_2$ and $\delta(K) = [1,k]$, 
and therefore $K \in \scrS_{\sigma \conc \sigma', \rho \ostar \rho'}^{(k)}$.
In addition, $f(K) = (K_1, K_2)$ by the definition of $K$.
On the other hand, given $J \in \scrS_{\sigma \conc \sigma', \rho \ostar \rho'}^{(k)}$, letting $(K_1,K_2) = f(J)$, one can easily see that $K_1 \cup \{i \mid i-m \in K_2\} = J$.
So the inverse of $f$ is well-defined and thus $f$ is bijective.

For our assertion, it suffices to show that for each $J \in \scrS_{\sigma \conc \sigma', \rho \ostar \rho'}^{(k)}$, 
\begin{align*}
&\sfB\big(((\sigma  \conc  \sigma')_{\scalebox{0.55}{$J$}})_{\le k}, ((\rho \ostar \rho')^{\scalebox{0.55}{$J$}})_{\le k}\big) \otimes
\sfB\big(((\sigma  \conc  \sigma')_{\scalebox{0.55}{$J$}})_{> k}, ((\rho \ostar \rho')^{\scalebox{0.55}{$J$}})_{> k}\big)\\
&\cong
\sfB\left(
(\sigma_{\sJo})_{\le t}  \conc  (\sigma'_{\sJt})_{\le s},
(\rho^{\sJo})_{\le t} \ostar ({\rho'}^{\sJt})_{\le s}
\right) 
\otimes
\sfB\left(
(\sigma_{\sJo})_{> t}  \conc  (\sigma'_{\sJt})_{> s},
(\rho^{\sJo})_{> t} \ostar ({\rho'}^{\sJt})_{> s}
\right),
\end{align*}
where $J_1, J_2, t$, and $s$ are defined as in~\eqref{eq: J1, J2}.
This isomorphism immediately follows from the four equalities:
\begin{align*}
&(1)~ ((\sigma  \conc  \sigma')_{\scalebox{0.55}{$J$}})_{\le k} = (\sigma_{\sJo})_{\le t}  \conc  (\sigma'_{\sJt})_{\le s}, \quad
(2)~((\rho \ostar \rho')^{\scalebox{0.55}{$J$}})_{\le k} = (\rho^{\sJo})_{\le t} \ostar ({\rho'}^{\sJt})_{\le s}, \\
&(3)~ ((\sigma  \conc  \sigma')_{\scalebox{0.55}{$J$}})_{> k} = (\sigma_{\sJo})_{> t}  \conc  (\sigma'_{\sJt})_{> s}, \quad
(4)~((\rho \ostar \rho')^{\scalebox{0.55}{$J$}})_{> k} = (\rho^{\sJo})_{> t} \ostar ({\rho'}^{\sJt})_{> s}.
\end{align*}
Let us prove the equality (1).
Let $J = \{j_1< j_2 < \cdots < j_{k}\} \in \scrS_{\sigma \conc \sigma', \rho \ostar \rho'}^{(k)}$.
Then,
\[
((\sigma  \conc  \sigma')_{\scalebox{0.55}{$J$}})_{\le k} = (\sigma  \conc  \sigma')_{\scalebox{0.55}{$J$}}(j_i) 
\quad \text{for $1 \le i \le k$}.
\]
Assume that $j_t \le m$ and $j_{t+1} > m$.
Set
\[
J_1 = \{j_1 < j_2 < \cdots < j_t\} \quad \text{and} \quad
J_2 = \{j_{t+1}-m < j_{t+2}-m < \cdots < j_{t+s}-m\}.
\]
Since
\begin{align*}
(\sigma_{\sJo})_{\le t}(i) &= \sigma_{\sJo}(j_i) \quad \text{for $1 \le i \le t$}, \\
(\sigma'_{\sJt})_{\le s}(i) &= \sigma'_{\sJt}(j_{t+i} - m) \quad \text{for $1 \le i \le s$},
\end{align*}
it follows that
\begin{align*}
(\sigma_{\sJo})_{\le t}  \conc  (\sigma'_{\sJt})_{\le s} (i) = \begin{cases}
\sigma_{\sJo}(j_i) & \text{if $1 \le i \le t$,} \\
\sigma'_{\sJt}(j_{i} - m) + t & \text{if $t+1 \le i \le t+s$}.
\end{cases}
\end{align*}
In case where $i \in [1,t]$, we have
\begin{align*}
(\sigma  \conc  \sigma')_{\scalebox{0.55}{$J$}}(j_i) 
& = \rmperm_{(\sigma  \conc  \sigma') \cdot J} (\sigma  \conc  \sigma') (j_i)
= \rmperm_{\sigma \cdot J_1} \sigma(j_i)
= \sigma_{\sJo}(j_i).
\end{align*}
In case where $i \in [t+1, t+s]$, we have
\begin{align*}
(\sigma  \conc  \sigma')_{\scalebox{0.55}{$J$}}(j_i) 
& = \rmperm_{(\sigma  \conc  \sigma') \cdot J} (\sigma  \conc  \sigma') (j_i)
= \rmperm_{\sigma' \cdot J_2} \sigma'(j_i - m) + t
= \sigma'_{\sJt}(j_i - m) + t.
\end{align*}
Thus, the equality (1) holds.

Next, let us prove the equality (2).
Under the same setting with the above paragraph, we have
\begin{align*}
((\rho \ostar \rho')^{\sJ})_{\le k}(i)
& = (\rho \ostar \rho')^{\sJ}(j_i) \quad \text{for $1 \le i \le k$}, \\
(\rho^{\sJo})_{\le t}(i) 
& = \rho^{\sJo} (j_i) \quad \text{for $1 \le i \le t$}, \\
({\rho'}^{\sJt})_{\le s}(i) 
& = \rho'^{\sJt} (j_{t+i} - m) \quad \text{for $1 \le i \le s$}. 
\end{align*}
From the second and third equalities, we have
\[
(\rho^{\sJo})_{\le t} \ostar ({\rho'}^{\sJt})_{\le s}
= \begin{cases}
\rho^{\sJo}(j_i) + s \quad \text{for $1\le i \le t$}, \\
\rho^{\sJt}(j_i - m)  \quad \text{for $t + 1\le i \le t + s$}.
\end{cases}
\]
In case where $i \in [1,t]$, we have
\begin{align*}
(\rho \ostar \rho')^{\sJ}(j_i) 
& = \rmperm^{w_0^{(m+n)}(\rho \ostar \rho') \cdot J} w_0^{(m+n)}(\rho \ostar \rho')(j_i) 
= \rmperm^{w_0^{(m)}(\rho) \cdot J_1} w_0^{(m)}\rho(j_i) + s \\
& = \rho^{\sJo}(j_i) + s.
\end{align*}
In case where $i \in [t+1,t+s]$, we have
\begin{align*}
(\rho \ostar \rho')^{\sJ}(j_i) 
& = \rmperm^{w_0^{(m+n)}(\rho \ostar \rho') \cdot J} w_0^{(m+n)}(\rho \ostar \rho')(j_i) 
= \rmperm^{w_0^{(n)}(\rho') \cdot J_2} w_0^{(n)}\rho'(j_i- m) \\
& = \rho'^{\sJt}(j_i - m).
\end{align*}
Thus, the equality (2) holds.

Equalities (3) and (4) can be proven in a similar way with (1) and (2) respectively, so we omit the proofs.
\end{proof}

\begin{remark}
For $\sigma \in \SG_m$ and $\rho \in \SG_n$, let $\sigma \shuffle \rho$ be the set of permutations $\gamma \in \SG_{m+n}$ satisfying that $\sigma(1)\sigma(2)\cdots\sigma(m)$ and $(\rho(1)+m) (\rho(2)+m) \cdots (\rho(n)+m)$ are subwords of $\gamma(1) \gamma(2) \cdots \gamma(m+n)$ and $\sigma \tshuffle \rho := \{\gamma^{-1} \mid  \gamma \in \sigma^{-1} \, \shuffle \, \rho^{-1}\}$.
For $X \subseteq \SG_m$ and $Y \subseteq \SG_n$, let
\[
X \tshuffle Y := \bigcup_{\sigma \in X, \, \rho \in Y} \sigma \tshuffle \rho.
\]
In the proof of Theorem~\ref{thm: Mackey formula}, we employ the fact that
$[\sigma \conc \sigma', \rho \ostar \rho']_L \subseteq [\sigma,\rho]_L \tshuffle [\sigma',\rho']_L$
for $\sigma, \rho \in \SG_m$ and $\sigma', \rho' \in \SG_n$.
In fact,
\[
[\sigma \conc \sigma', \rho \ostar \rho']_L = [\sigma,\rho]_L \tshuffle [\sigma',\rho']_L.
\]
Therefore, by Lemma~\ref{lem: tensor product}, we have that
\begin{align*}
\sfB(\sigma,\rho) \boxtimes \sfB(\sigma',\rho') 
\cong
\sfB([\sigma,\rho]_L \tshuffle [\sigma',\rho']_L).
\end{align*}

It is well known that the multiplicative rule for the fundamental quasisymmetric functions are described as follow:
\begin{align}\label{eq: product of F}
F_\alpha F_\beta = \sum_{\gamma \in \sigma \shuffle \rho} F_{\comp(\Des_R(\gamma))}.
\end{align}
Here, $\sigma \in \SG_{\ell(\alpha)}$ and $\rho \in \SG_{\ell(\beta)}$ satisfying $\Des_R(\sigma) = \set(\alpha)$ and $\Des_R(\rho) = \set(\beta)$.
Now, one can see that the multiplicative rule~\eqref{eq: product of F} lifts to the induction product
\begin{align*}
\sfB(\sigma,\sigma) \boxtimes \sfB(\rho, \rho) 
\cong
\sfB(\sigma \tshuffle \rho).
\end{align*}
\end{remark}

\subsection{(Anti-)automorphism twists of weak Bruhat interval modules}\label{subsec: automorphism twists}
Let $\mu: B \to A$ be an isomorphism of associative algebras over $\mathbb{C}$. Given an $A$-module $M$, we define $\mu[M]$ by the $B$-module with the same underlying space as $M$ and with the action $\cdot_{\mu}$ twisted by $\mu$ in such a way that
\begin{align*}
b \cdot_{\mu} v := \mu(b) \cdot v \quad \text{for $a \in A$ and $v \in M$}.
\end{align*}
Let $\module A$ be the category of finite dimensional left $A$-modules. 
Any isomorphism $\mu: B \to A$ induces  
a covariant functor 
\begin{align}\label{eq: isomorphism twist}
\bfT^+_{\mu}: \module A \ra \module B, \quad M \mapsto \mu[M],
\end{align}
where $\bfT^+_{\mu}(h):\mu[M] \to \mu[N], m \mapsto h(m)$ for every $A$-module homomorphism $h: M\to N$.
We call $\bfT^+_{\mu}$ the \emph{$\mu$-twist}.

Similarly, given an anti-isomorphism $\nu: B \to A$,
we define $\nu[M]$ to be the $B$-module with $M^*$, the dual space of $M$, as the underlying space and with the action $\cdot^{\nu}$ defined by
\begin{align}\label{eq: anti-automorphism twist}
(b \cdot^{\nu} \delta) (v) := \delta(\nu (b) \cdot v) \quad \text{for $b \in B$, $\delta \in M^*$, and $v \in M$}.
\end{align}
Any anti-isomorphism $\nu: B \to A$ induces a contravariant functor
\[
\bfT^-_{\nu}: \module A \ra \module B, \quad M \mapsto \nu[M], 
\]
where $\bfT^-_{\nu}(h):\nu[N] \to \nu[M], \delta \mapsto \delta \circ h$
for every $A$-module homomorphism $h: M\to N$.
We call $\bfT^-_{\nu}$ the \emph{$\nu$-twist}.
In~\cite{05Fayers}, Fayers introduced the involutions $\autophi, \autotheta$ and the anti-involution $\autochi$ of $H_n(0)$ defined in the following manner:
\begin{align*}
&\autophi: H_n(0) \ra H_n(0), \quad \pi_i \mapsto \pi_{n-i} \quad \text{for $1 \le i \le n-1$},\\
&\autotheta: H_n(0) \ra H_n(0), \quad \pi_i \mapsto - \opi_i \quad \text{for $1 \le i \le n-1$}, \\
&\autochi: H_n(0) \ra H_n(0), \quad \pi_i \mapsto \pi_i \quad \text{for $1 \le i \le n-1$}.
\end{align*}
These morphisms commute with each other.
We study the (anti-)involution twists for $\autophi$, $\autotheta$, $\autochi$, and their compositions $\autoomega:= \autophi \circ \autotheta$, $\hautophi:= \autophi \circ \autochi$, $\hautotheta := \autotheta \circ \autochi$, $\hautoomega:= \autoomega \circ \autochi$.

The viewpoint of looking at (anti-)involutions as functors is quite useful for many reasons.
The primary reason is that using the exactness of the corresponding functors, 
one can transport various structures of a given $H_n(0)$-module to their twists in a functorial way. 
An application in this direction can be found in Subsection~\ref{Involution twists; application}. 
Additional reasons include that some well known functors appear in the context of our (anti-)involution twists.
Given any anti-automorphism $\nu$ of $H_n(0)$, the \emph{standard duality $D: \module H_n(0) \to \module H_n(0)^{\mathrm{op}}$} appears as $F_{\nu^{-1}} \circ \bfT^-_{\nu}$, where $F_{\nu^{-1}}:\module H_n(0) \to \module H_n(0)^{\mathrm{op}}$ is the functor induced by
the inverse of $\mathring{\nu}: H_n(0) \to  H_n(0)^{\mathrm{op}}, x \mapsto \nu(x)$.
In particular, $D \cong F_{\chi} \circ \bfT_\chi^-$.
The \emph{Nakayama functor $\upnu$} is naturally isomorphic to $\bfT^+_{\autophi}$, which can be derived by combining~\cite[Proposition 4.2]{05Fayers} with~\cite[Proposition IV.3.13]{11SY}. 
To explain in more detail, the former reference implies $\autophi$ is a Nakayama automorphism and 
the latter reference shows the relationship between Nakayama automorphisms and $\upnu$.
And, the \emph{$H_n(0)$-dual functor}, ${\Hom}_{H_n(0)}(-, H_n(0))$, is naturally isomorphic to $D\circ \upnu$, 
and therefore is naturally isomorphic to $F_{\chi} \circ \bfT_{\hautophi}^-$. 

Now, let us focus on the main topic of this subsection, (anti-)involution twists of weak Bruhat interval modules.
For irreducible modules and projective indecomposable modules, it was shown in~\cite{05Fayers,16Huang} that
\begin{align*}
&\autophi[\bfF_\alpha] \cong \bfF_{\alpha^\rmr}, \quad
\autotheta[\bfF_\alpha] \cong \bfF_{\alpha^\rmc}, \quad
\autochi[\bfF_\alpha] \cong \bfF_\alpha, \\
&\autophi[\calP_\alpha] \cong \calP_{\alpha^\rmr}, \quad 
\autotheta[\calP_\alpha] \cong \calP_{\alpha^\rmc}.
\end{align*}
The following theorem shows how the (anti-)involution twists act on weak Bruhat interval modules.
\begin{theorem} \label{thm: auto twist part 1}
For $\sigma, \rho \in \SG_n$, we have the following isomorphisms of $H_n(0)$-modules.
\begin{enumerate}[label = {\rm (\arabic*)}]
\item $\autophi[\sfB(\sigma,\rho)] \cong \sfB(\sigma^{w_0},\rho^{w_0})$.
\item $\autotheta[\sfB(\sigma,\rho)] \cong \osfB(\sigma,\rho)$.
\item $\autochi[\sfB(\sigma,\rho)] \cong \osfB(\rho w_0, \sigma w_0)$.
In particular, $\autochi[\calP_{\alpha}] \cong \calP_{\alpha^\rmr}$.
\end{enumerate}
\end{theorem}

\begin{proof}
Consider the $\mathbb C$-linear isomorphisms defined by  
\begin{align*}
&f_1: \autophi[\sfB(\sigma, \rho)] \ra \sfB(\sigma^{w_0}, \rho^{w_0}), \quad \gamma \mapsto\gamma^{w_0},\\
&f_2: \autotheta[\sfB(\sigma, \rho)] \ra \osfB(\sigma, \rho),  \quad
\gamma \mapsto (-1)^{\ell(\gamma \sigma^{-1})} \gamma,\\
&f_3:\autochi[\sfB(\sigma, \rho)] \ra \osfB(\rho w_0, \sigma w_0), \quad \gamma^\ast \mapsto \gamma w_0,
\end{align*}
where $\gamma \in [\sigma, \rho]_L$ and $\gamma^\ast$ denotes the dual of $\gamma$ with respect to the basis $[\sigma, \rho]_L$ for $\sfB(\sigma, \rho)$.
Since it can be proven in a similar manner that these maps are $H_n(0)$-isomorphisms, 
we here only deal with (3). 

Note that, for $1 \leq i \le n-1$,
\[
\opi_i \cdot^\autochi \gamma^* 
= \begin{cases}
-\gamma^* & \text{if $i \notin \Des_L(\gamma)$,} \\
(s_i\gamma)^* & \text{if $i \in \Des_L(\gamma)$ and $s_i\gamma \in [\sigma, \rho]_L$,} \\
0 & \text{if $i \in \Des_L(\gamma)$ and $s_i\gamma \notin [\sigma, \rho]_L$},
\end{cases}
\]
which yields that 
\[
f_3(\opi_i \cdot^\autochi \gamma^*) 
= \begin{cases}
-\gamma w_0 & \text{if $i \notin \Des_L(\gamma)$,} \\
s_i \gamma w_0 & \text{if $i \in \Des_L(\gamma)$ and $s_i\gamma \in [\sigma, \rho]_L$,} \\
0 & \text{if $i \in \Des_L(\gamma)$ and $s_i\gamma \notin [\sigma, \rho]_L$.}
\end{cases}
\]
On the other hand,
\begin{align*}
\opi_i \star f_3(\gamma^*) 
&= \opi_i \star \gamma w_0 \\
&= \begin{cases}
-\gamma w_0 & \text{if $i \in \Des_L(\gamma w_0)$,} \\
 s_i \gamma w_0 & \text{if $i \notin \Des_L(\gamma w_0)$ and $s_i \gamma w_0 \in [\rho w_0, \sigma w_0]_L$,} \\
0 & \text{if $i \notin \Des_L(\gamma w_0)$ and $s_i \gamma w_0 \notin [\rho w_0, \sigma w_0]_L$.}
\end{cases}
\end{align*}
It immediately follows from~\eqref{eq: Des alternating def} that $i \notin \Des_L(\gamma)$ if and only if $i \in \Des_L(\gamma w_0)$.
Moreover, it is trivial that $s_i\gamma \in [\sigma, \rho]_L$ if and only if $s_i \gamma w_0 \in [\rho w_0, \sigma w_0]_L$.
Thus, we verified $\autochi[\sfB(\sigma,\rho)] \cong \osfB(\rho w_0, \sigma w_0)$.
And, combining the equality $w_0(\alpha)^{w_0} = w_0(\alpha^\rmr)$ with (3) yields $\autochi[\calP_{\alpha}] \cong \calP_{\alpha^\rmr}$.
\end{proof}

\begin{table}[t]
\fontsize{9}{9}
\renewcommand*\arraystretch{1.2}
\begin{tabular}{ c || c | c | c | c | c | c | c  } 
& $\autophi$-twist
& $\autotheta$-twist
& $\autoomega$-twist
& $\autochi$-twist
& $\hautophi$-twist
& $\hautotheta$-twist
& $\hautoomega$-twist
\\ \hline \hline
$\sfB(\sigma, \rho)$
& $\sfB(\sigma^{w_0}, \rho^{w_0})$
& $\osfB(\sigma, \rho)$
& $\osfB(\sigma^{w_0}, \rho^{w_0})$
& $\osfB(\rho w_0, \sigma w_0)$
& $\osfB(w_0 \rho, w_0 \sigma)$
& $\sfB(\rho w_0, \sigma w_0)$
& $\sfB(w_0 \rho, w_0 \sigma)$
\\
\hline
$\osfB(\sigma, \rho)$ 
& $\osfB(\sigma^{w_0}, \rho^{w_0})$
& $\sfB(\sigma, \rho)$
& $\sfB(\sigma^{w_0}, \rho^{w_0})$ 
& $\sfB(\rho w_0, \sigma w_0)$
& $\sfB(w_0 \rho, w_0 \sigma)$
& $\osfB(\rho w_0, \sigma w_0)$
& $\osfB(w_0 \rho, w_0 \sigma)$
\end{tabular}
\caption{(Anti-)involution twists of weak Bruhat interval modules}
\label{tab: auto twist of BIM}
\end{table}

\begin{table}[t]
\fontsize{9}{9}
\renewcommand*\arraystretch{1.2}
\begin{tabular}{ c || c | c | c | c | c | c | c  } 
& $\autophi$-twist
& $\autotheta$-twist
& $\autoomega$-twist
& $\autochi$-twist
& $\hautophi$-twist
& $\hautotheta$-twist
& $\hautoomega$-twist
\\ \hline \hline
$\bfF_\alpha$
& $\bfF_{\alpha^\rmr}$
& $\bfF_{\alpha^\rmc}$
& $\bfF_{\alpha^\rmt}$
& $\bfF_{\alpha}     $
& $\bfF_{\alpha^\rmr}$
& $\bfF_{\alpha^\rmc}$
& $\bfF_{\alpha^\rmt}$
\\ \hline
$\calP_\alpha$
& $\calP_{\alpha^\rmr}$
& $\calP_{\alpha^\rmc}$
& $\calP_{\alpha^\rmt}$
& $\calP_{\alpha^\rmr}$
& $\calP_{\alpha}     $
& $\calP_{\alpha^\rmt}$
& $\calP_{\alpha^\rmc}$
\end{tabular}
\caption{(Anti-)involution twists of $\bfF_\alpha$ and $\calP_\alpha$}
\label{tab: auto twist of F and P}
\end{table}

As seen in Table~\ref{tab: auto twist of BIM}, various (anti-)involution twists can be obtained from Theorem~\ref{thm: auto twist part 1} by composing $\autotheta$, $\autophi$, and $\autochi$.
For the reader's understanding, we deal with irreducible modules and projective indecomposable modules in a separate table (see Table~\ref{tab: auto twist of F and P}).

\begin{example}\label{eg: inv twist}
By Theorem~\ref{thm: auto twist part 1}, we have
\begin{align*}
\autophi[\sfB(2134, 4123)] &\cong \sfB(1243, 2341), \\
\autotheta[\sfB(2134, 4123)] &\cong \osfB(2134, 4123), \ \text{and} \\
\autochi[\sfB(2134, 4123)] &\cong \osfB(3214, 4312).
\end{align*}
We illustrate these $H_4(0)$-modules as follows:
\[
\def \hp{0.25}
\def \wp{0.2}
\def \wtab{1.7}
\def \htab{2.1}
\begin{tikzpicture}[baseline = 0mm, scale = 0.7]
%node
\node[below] at (0,-0.3) {$\sfB(2134, 4123)$};
\node at (0, 0.2*\htab) {$0$};
\node at (0, \htab) {$4123$};
\node at (0, 2*\htab) {$3124$};
\node at (0, 3*\htab) {$2134$};
\node at (0.8*\wtab, 2.55*\htab) {$0$};
\node at (0.8*\wtab, 1.55*\htab) {$0$};
%arrow and label
%1 
\draw [->] (2*\wp, 3*\htab - 1.5*\hp) -- (0.8*\wtab-\wp, 2.45*\htab+\hp);
\node[right] at (0.25*\wtab + 0.5*\wp ,2.78*\htab) {\scriptsize $\pi_3$};
%2
\draw [->] (2*\wp, 2*\htab - 1.5*\hp) -- (0.8*\wtab - \wp, 1.45*\htab + \hp);
\node[right] at (0.25*\wtab + 0.5*\wp ,1.78*\htab ) {\scriptsize $\pi_1$};
%loop
\node at (0.25*\wtab, 3.05*\htab)  {} edge [out=35,in=325, loop] ();
\node[right] at (0.65*\wtab, 3.15*\htab) {\scriptsize $\pi_1$};
\node at (0.25*\wtab, 2.05*\htab)  {} edge [out=35,in=325, loop] ();
\node[right] at (0.65*\wtab, 2.15*\htab) {\scriptsize $\pi_2$};
\node at (0.25*\wtab, 1.05*\htab)  {} edge [out=35,in=325, loop] ();
\node[right] at (0.65*\wtab, 1.15*\htab) {\scriptsize $\pi_3$};
%arrow and label: vertical
%1
\draw [->] (0, 3*\htab - 1.5*\hp) -- (0, 2*\htab+1.5*\hp);
\node[left] at (0.8*\wp, 2.5*\htab) {\scriptsize $\pi_2$};
%2
\draw [->] (0, 2*\htab - 1.5*\hp) -- (0, 1*\htab+1.5*\hp);
\node[left] at (0.8*\wp, 1.5*\htab) {\scriptsize $\pi_3$};
%3
\draw [->] (0, 1*\htab - 1.5*\hp) -- (0, 0.2*\htab+1.5*\hp);
\node[left] at (0.8*\wp, 0.6*\htab) {\scriptsize $\pi_1, \pi_2$};
\end{tikzpicture}
\quad
\begin{tikzpicture}[baseline = 0mm, scale = 0.7]
%node
\node[below] at (0,-0.3) {$\autophi[\sfB(2134, 4123)]$};
\node[below] at (0.1,-1.2) {$(\cong \sfB(1243, 2341))$};
\node at (0, 0.2*\htab) {$0$};
\node at (0, \htab) {$4123$};
\node at (0, 2*\htab) {$3124$};
\node at (0, 3*\htab) {$2134$};
\node at (0.8*\wtab, 2.55*\htab) {$0$};
\node at (0.8*\wtab, 1.55*\htab) {$0$};
%arrow and label
%1 
\draw [->] (2*\wp, 3*\htab - 1.5*\hp) -- (0.8*\wtab-\wp, 2.45*\htab+\hp);
\node[right] at (0.25*\wtab + 0.5*\wp ,2.78*\htab) {\scriptsize $\pi_1$};
%2
\draw [->] (2*\wp, 2*\htab - 1.5*\hp) -- (0.8*\wtab - \wp, 1.45*\htab + \hp);
\node[right] at (0.25*\wtab + 0.5*\wp ,1.78*\htab ) {\scriptsize $\pi_3$};
%loop
\node at (0.25*\wtab, 3.05*\htab)  {} edge [out=35,in=325, loop] ();
\node[right] at (0.65*\wtab, 3.15*\htab) {\scriptsize $\pi_3$};
\node at (0.25*\wtab, 2.05*\htab)  {} edge [out=35,in=325, loop] ();
\node[right] at (0.65*\wtab, 2.15*\htab) {\scriptsize $\pi_2$};
\node at (0.25*\wtab, 1.05*\htab)  {} edge [out=35,in=325, loop] ();
\node[right] at (0.65*\wtab, 1.15*\htab) {\scriptsize $\pi_1$};
%arrow and label: vertical
%1
\draw [->] (0, 3*\htab - 1.5*\hp) -- (0, 2*\htab+1.5*\hp);
\node[left] at (0.8*\wp, 2.5*\htab) {\scriptsize $\pi_2$};
%2
\draw [->] (0, 2*\htab - 1.5*\hp) -- (0, 1*\htab+1.5*\hp);
\node[left] at (0.8*\wp, 1.5*\htab) {\scriptsize $\pi_1$};
%3
\draw [->] (0, 1*\htab - 1.5*\hp) -- (0, 0.2*\htab+1.5*\hp);
\node[left] at (0.8*\wp, 0.6*\htab) {\scriptsize $\pi_2, \pi_3$};
\end{tikzpicture}
\quad
\begin{tikzpicture}[baseline = 0mm, scale = 0.7]
%node
\node[below] at (0,-0.3) {$\autotheta[\sfB(2134, 4123)]$};
\node[below] at (0.1,-1.2) {$(\cong \osfB(2134, 4123))$};
\node at (0, 0.2*\htab) {$0$};
\node at (0, \htab) {$4123$};
\node at (0, 2*\htab) {$3124$};
\node at (0, 3*\htab) {$2134$};
\node at (0.8*\wtab, 2.55*\htab) {$0$};
\node at (0.8*\wtab, 1.55*\htab) {$0$};
%arrow and label
%1 
\draw [->] (2*\wp, 3*\htab - 1.5*\hp) -- (0.8*\wtab-\wp, 2.45*\htab+\hp);
\node[right] at (0.25*\wtab + 0.5*\wp ,2.78*\htab) {\scriptsize $-\opi_3$};
%2
\draw [->] (2*\wp, 2*\htab - 1.5*\hp) -- (0.8*\wtab - \wp, 1.45*\htab + \hp);
\node[right] at (0.25*\wtab + 0.5*\wp ,1.78*\htab ) {\scriptsize $-\opi_1$};
%loop
\node at (0.25*\wtab, 3.05*\htab)  {} edge [out=35,in=325, loop] ();
\node[right] at (0.65*\wtab, 3.15*\htab) {\scriptsize $-\opi_1$};
\node at (0.25*\wtab, 2.05*\htab)  {} edge [out=35,in=325, loop] ();
\node[right] at (0.65*\wtab, 2.15*\htab) {\scriptsize $-\opi_2$};
\node at (0.25*\wtab, 1.05*\htab)  {} edge [out=35,in=325, loop] ();
\node[right] at (0.65*\wtab, 1.15*\htab) {\scriptsize $-\opi_3$};
%arrow and label: vertical
%1
\draw [->] (0, 3*\htab - 1.5*\hp) -- (0, 2*\htab+1.5*\hp);
\node[left] at (0.8*\wp, 2.5*\htab) {\scriptsize $-\opi_2$};
%2
\draw [->] (0, 2*\htab - 1.5*\hp) -- (0, 1*\htab+1.5*\hp);
\node[left] at (0.8*\wp, 1.5*\htab) {\scriptsize $-\opi_3$};
%3
\draw [->] (0, 1*\htab - 1.5*\hp) -- (0, 0.2*\htab+1.5*\hp);
\node[left] at (0.8*\wp, 0.6*\htab) {\scriptsize $-\opi_1, -\opi_2$};
\end{tikzpicture}
\quad
\begin{tikzpicture}[baseline = 0mm, scale = 0.7]
%node
\node[below] at (0,-0.3) {$\autochi[\sfB(2134, 4123)]$};
\node[below] at (0.1,-1.2) {$(\cong \osfB(3214, 4312))$};
\node at (0, 0.2*\htab) {$0$};
\node at (0, \htab) {$2134^*$};
\node at (0, 2*\htab) {$-3124^*$};
\node at (0, 3*\htab) {$4123^*$};
\node at (0.35*\wtab, 3.05*\htab)  {} edge [out=35,in=325, loop] ();
\node[right] at (0.75*\wtab, 3.15*\htab) {\scriptsize $-\opi_1,-\opi_2$};
\node at (0.35*\wtab, 2.05*\htab)  {} edge [out=35,in=325, loop] ();
\node[right] at (0.75*\wtab, 2.15*\htab) {\scriptsize $-\opi_1, -\opi_3$};
\node at (0.35*\wtab, 1.05*\htab)  {} edge [out=35,in=325, loop] ();
\node[right] at (0.75*\wtab, 1.15*\htab) {\scriptsize $-\opi_2, -\opi_3$};
%arrow and label: vertical
%1
\draw [->] (0, 3*\htab - 1.5*\hp) -- (0, 2*\htab+1.5*\hp);
\node[left] at (0.8*\wp, 2.5*\htab) {\scriptsize $-\opi_3$};
%2
\draw [->] (0, 2*\htab - 1.5*\hp) -- (0, 1*\htab+1.5*\hp);
\node[left] at (0.8*\wp, 1.5*\htab) {\scriptsize $-\opi_2$};
%3
\draw [->] (0, 1*\htab - 1.5*\hp) -- (0, 0.2*\htab+1.5*\hp);
\node[left] at (0.8*\wp, 0.6*\htab) {\scriptsize $-\opi_1$};
\end{tikzpicture}
\]
\end{example}

For an (anti-)automorphism $\zeta:H_{m+n}(0) \to H_{m+n}(0)$ and an $H_m(0) \otimes H_n(0)$-module $M$, 
we simply write $\zeta[M]$ for $\zeta|_{H_m(0) \otimes H_n(0)}[M]$.
The subsequent corollary shows that (anti-)involution twists behave nicely with respect to induction product and restriction.
\begin{corollary}\label{cor: automorhpism, induction, restriction}
Let $M$, $N$, and $L$ be weak Bruhat interval modules of $H_m(0)$, $H_n(0)$, and $H_{m+n}(0)$, respectively. 
Then we have following isomorphisms of modules.
\begin{enumerate}[label = {\rm (A\arabic*)}]
\item $\autophi[M \boxtimes N] \cong \autophi[N] \boxtimes \autophi[M]$
\item $\autotheta[M \boxtimes N] \cong \autotheta[M] \boxtimes \autotheta[N]$
\item $\autochi [M \boxtimes N] \cong \autochi [N] \boxtimes \autochi [M]$
\end{enumerate}
\begin{enumerate}[label = {\rm (B\arabic*)}]
\item $\autophi[L \downarrow^{H_{m+n}(0)}_{H_n(0) \otimes H_m(0)}] \cong \autophi[L] \downarrow^{H_{m+n}(0)}_{H_m(0) \otimes H_n(0)}$
\item $\autotheta[L \downarrow^{H_{m+n}(0)}_{H_m(0) \otimes H_n(0)}] 
\cong \autotheta[L] \downarrow^{H_{m+n}(0)}_{H_m(0) \otimes H_n(0)}$
\item $\autochi[L \downarrow^{H_{m+n}(0)}_{H_m(0) \otimes H_n(0)}] \cong \autochi[L] \downarrow^{H_{m+n}(0)}_{H_m(0) \otimes H_n(0)}$
\end{enumerate}
\end{corollary}
\begin{proof}
(A1), (A2), (B1), (B2), and (B3) are straightforward from the definitions of $\autophi$, $\autotheta$, and $\autochi$. 

For (A3), let $M=\sfB(\sigma_1, \rho_1)$ and $N=\sfB(\sigma_2, \rho_2)$. 
Then, by Lemma~\ref{lem: tensor product} and Theorem~\ref{thm: auto twist part 1}~(3),
$$
\autochi[M \boxtimes N] \cong \osfB(\rho_1 \ostar \rho_2 \; w_0^{(m+n)}, \sigma_1  \conc  \sigma_2 \; w_0^{(m+n)})
$$
and
\begin{align*}
\autochi[N] \boxtimes \autochi[M] 
&\cong \osfB(\rho_2 \; w_0^{(n)}, \sigma_2 \; w_0^{(n)}) \boxtimes \osfB(\rho_1 \; w_0^{(m)}, \sigma_1 \; w_0^{(m)}) \\
&\cong \osfB(\rho_2 \; w_0^{(n)}  \conc  \rho_1 \; w_0^{(m)}, \sigma_2 \; w_0^{(n)} \ostar \sigma_1 \; w_0^{(m)}).
\end{align*}
Here, the notation $w_0^{(k)}$ denotes the longest element in $\SG_k$ and the last isomorphism follows from Lemma~\ref{lem: tensor product} and (A2).
Therefore, the assertion follows from the fact that the one-line notation of $\sigma w_0$ is obtained by reversing $\sigma$ for any permutation $\sigma$.
\end{proof}

\section{Various $0$-Hecke modules constructed using tableaux}\label{sec: Examples of WBI modules}

Suppose we have a family of quasisymmetric functions that can be expanded in the basis of the fundamental quasisymmetric functions with positive coefficients.
The correspondence \eqref{quasi characteristic} tells us that 
each of them appears as the image of the isomorphism classes of 
certain $H_n(0)$-modules.  
Among them, it would be very nice
to find or construct one which is nontrivial, in other words,
not a direct sum of irreducible modules
and has a combinatorial model that can be handled well.     
Since the mid-2010s, some $H_n(0)$-modules have been constructed in line with this philosophy, more precisely,   
in~\cite{20BS, 15BBSSZ, 19Searles, 15TW}.
In this section, we show that all of them are equipped with the structure of weak Bruhat interval modules. 

To deal with these modules, we need the notion of source and sink.

\begin{definition} \label{def: of source and sink}
Let $B$ be a basis for an $H_n(0)$-module such that $B \cup \{0\}$ is closed under the action of $\{\pi_i \mid 1 \le i \le n-1\}$.
\begin{enumerate}[label = {\rm (\arabic*)}]
\item 
An element $x_0 \in B$ is called a \emph{source of $B$} if, 
for each $x \in B$, there exists $\sigma \in \SG_n$ such that $\pi_\sigma \cdot x_0 = x$.
\item 
An element $x'_0 \in B$ is called a \emph{sink of $B$} if, 
for each $x \in B$, there exists $\sigma \in \SG_n$ such that $\pi_\sigma \cdot x = x'_0$.
\end{enumerate}
\end{definition}

Following the way as in \cite{15TW}, one can see that  
there are at most one source and sink in $B$.
In case where $B$ is the basis $[\sigma, \rho]_L$ for $\sfB(\sigma, \rho)$, $\sigma$ is the source and $\rho$ is the sink.

Hereafter, $\alpha$ denotes a composition of $n$.
To introduce the tableaux in our concern, we need to define the \emph{composition diagram $\tcd(\alpha)$ of shape $\alpha$}.
It is a left-justified array of $n$ boxes where the $i$th row from the top has $\alpha_i$ boxes for $1 \le i \le k$.
For a filling $\tau$ of $\tcd(\alpha)$, we denote by $\tau_{i,j}$ the entry in the $i$th row from the top and $j$th column from the left.

\subsection{Standard immaculate tableaux, standard extended tableaux, and their $H_n(0)$-modules}

We begin with introducing the definition of standard immaculate tableaux and standard extended tableaux.

\begin{definition}\label{def: SIT}{\rm (\cite{15BBSSZ,19Searles})}
Let $\alpha$ be a composition of $n$. 
\begin{enumerate}
\item 
A \emph{standard immaculate tableau of shape $\alpha$} is a filling $\calT$ of the composition diagram $\tcd(\alpha)$ with $\{1,2,\ldots,n\}$ such that the entries are all distinct, the entries in each row increase from left to right, and the entries in the first column increase from top to bottom.
\item
A \emph{standard extended tableau of shape $\alpha$} is a filling $\sfT$ of the composition diagram $\tcd(\alpha)$ with $\{1,2,\ldots,n\}$ such that the entries are all distinct, the entries in each row increase from left to right, and the entries in each column increase from top to bottom.
\end{enumerate}
\end{definition}

We remark that our standard extended tableaux are slightly different from those of Searles~\cite{19Searles}.
In fact, the former can be obtained by flipping the latter horizontally.

Denote by $\SIT(\alpha)$ the set of all standard immaculate tableaux of shape $\alpha$ and 
by $\SET(\alpha)$ the set of all standard extended tableaux of shape $\alpha$.
Berg et al.~\cite{15BBSSZ} define a $0$-Hecke action on $\SIT(\alpha)$ and denote the resulting module by $\calV_\alpha$.
And, Searles~\cite{19Searles} define a $0$-Hecke action on $\SET(\alpha)$ and denote the resulting module by $X_\alpha$.
By the construction of $\calV_\alpha$ and $X_\alpha$, it is clear that $\SIT(\alpha)$ and $\SET(\alpha)$ are bases for $\calV_\alpha$ and $X_\alpha$, respectively.

It is not difficult to show that both $\SIT(\alpha)$ and $\SET(\alpha)$ have a unique source and a unique sink.
Denote the source of $\SIT(\alpha)$ by $\calT_\alpha$ and the source of $\SET(\alpha)$ by $\sfT_\alpha$. 
They are obtained by filling $\tcd(\alpha)$ with entries $1, 2, \ldots, n$ from left to right and from top to bottom.
Denote the sink of $\SIT(\alpha)$ by $\calT'_\alpha$ and the sink of $\SET(\alpha)$ by $\sfT'_\alpha$.
In contrast of $\calT_\alpha$ and $\sfT_\alpha$, $\calT'_\alpha$ and $\sfT'_\alpha$ have to be constructed separately.
The former $\calT'_\alpha$ is obtained from $\tcd(\alpha)$ in the following steps:
\begin{enumerate}[label = {\rm (\arabic*)}]
\item Fill the first column with entries $1,2,\ldots, \ell(\alpha)$ from top to bottom.
\item Fill the remaining boxes with entries $\ell(\alpha) + 1, \ell(\alpha) + 2, \ldots, n$ from left to right from bottom to top.
\end{enumerate}
On the other hand, the latter $\sfT'_\alpha$ is obtained by filling $\tcd(\alpha)$ with the entries $1,2,\ldots, n$ from top to bottom and from left to right.

\begin{definition}\label{def: reading word}
For a filling $T$ of a composition diagram, $\rmread(T)$ is defined to be the word obtained from $T$ by reading the entries from right to left starting with the top row.
\end{definition}

With this definition, we can state the following theorem.
 
\begin{theorem}\label{thm: Va Xa pi opi form}
For any $\alpha \models n$, we have the $H_n(0)$-module isomorphisms
\begin{align}\label{eq: isomorphisms for Va and Xa}
\calV_\alpha \cong \sfB(\rmread(\calT_\alpha),\rmread(\calT'_\alpha)) 
\quad \text{and} \quad
X_\alpha \cong \sfB(\rmread(\sfT_\alpha),\rmread(\sfT'_\alpha)).
\end{align}
Here, the words in the parentheses are being viewed as permutations in one-line notation.
\end{theorem}

\begin{proof}
To prove the first isomorphism in~\eqref{eq: isomorphisms for Va and Xa}, we need the $H_n(0)$-module homomorphisms
\begin{equation}\label{eq: morphism sequence}
\begin{tikzcd}[column sep=large]
\autotheta[\bfP_{\alpha^\rmc}] \arrow[r, "\mathsf{w}"] 
&
\calP_\alpha 
\end{tikzcd}
\quad \text{and} \quad
\begin{tikzcd}[column sep=large]
\autotheta[\bfP_{\alpha^\rmc}] \arrow[r,two heads, "\mPhi"] 
& \calV_\alpha,
\end{tikzcd}
\end{equation}
where
\begin{enumerate}[label = -, leftmargin = *]
\item the notation $\bfP_{\alpha^\rmc}$ denotes the projective indecomposable module spanned by the standard ribbon tableaux of shape $\alpha^\rmc$ in~\cite[Subsection 3.2]{16Huang},
\item the first homomorphism $\mathsf{w}$ is an isomorphism given in~\cite[Theorem 3.3 and Proposition 5.1]{16Huang}, which is given by reading
standard ribbon tableaux from left to right starting with the bottom row, and
\item the second homomorphism $\mPhi$ is an essential epimorphism given in~\cite[Theorem 3.2]{20CKNO2}.
\end{enumerate}
Composing $\mathsf{w}^{-1}$ with $\mPhi$ yields a surjective $H_n(0)$-module homomorphism $\tPhi: \calP_\alpha \ra \calV_\alpha$.
In view of the definition of $\mathsf{w}$ and $\mPhi$, one can see that 
\begin{align*}
\tPhi(\pi_\gamma \opi_{w_0(\alpha)}) = 
\begin{cases}
\calT & \text{if $\rmread(\calT) = \gamma$ for some $\calT \in \SIT(\alpha)$,}\\
0 & \text{else}.
\end{cases}
\end{align*}
Composing $\tPhi$ with the isomorphism 
\[
\sfem: \sfB(w_0(\alpha^\rmc), w_0 w_0(\alpha)) \ra \calP_{\alpha} \quad \text{(in Theorem~\ref{thm: embedding}~(3))},
\]
we finally have the surjective $H_n(0)$-module homomorphism
\[
\tPhi \circ \sfem: \sfB(w_0(\alpha^\rmc), w_0 w_0(\alpha)) \ra \calV_\alpha.
\]
Next, let us consider the projection 
\begin{align*}
\pr: \sfB(w_0(\alpha^\rmc), w_0 w_0(\alpha)) 
& \ra \sfB(\rmread(\calT_\alpha), \rmread(\calT'_\alpha)), \\ 
\gamma \quad 
& \mapsto 
\begin{cases}
\gamma & \text{if $\gamma \in [\rmread(\calT_\alpha), \rmread(\calT'_\alpha)]_L$},\\
0 & \text{else.}
\end{cases}
\end{align*}
By the definition of $\tPhi$, one sees that $\rmread(\calT_\alpha) = w_0(\alpha^\rmc)$ and $\rmread(\calT'_\alpha) \preceq_L w_0 w_0(\alpha)$. This implies that $\pr$ is a surjective $H_n(0)$-module homomorphism.

For our purpose, we have only to show $\ker(\tPhi \circ \sfem) = \ker(\pr)$.
From the definition of the $H_n(0)$-action on $\calV_\alpha$ it follows that $\rmread(\calT) \in [\rmread(\calT_\alpha), \rmread(\calT'_\alpha)]_L$ for any $\calT \in \SIT(\alpha)$ and therefore $\ker(\tPhi \circ \sfem) \supseteq \ker(\pr)$.
On the other hand, using the fact that the $0$-Hecke action on $\SIT(\alpha)$ satisfies the braid relations, one can show that every $\gamma \in [\rmread(\calT_\alpha), \rmread(\calT'_\alpha)]_L$ appears as $\rmread(\calT)$ for some $\calT \in \SIT(\alpha)$.
This says that $\ker(\tPhi \circ \sfem)^\rmc \supseteq \ker(\pr)^\rmc$, so we are done.

The second isomorphism in~\eqref{eq: isomorphisms for Va and Xa} can be obtained by replacing $\mPhi$ with $\Gamma \circ \mPhi$ in~\eqref{eq: morphism sequence}, where $\Gamma: \calV_\alpha \ra X_\alpha$ is the surjection in~\cite[Subsection 3.2]{20CKNO2}.
\end{proof}

\subsection{Standard permuted composition tableaux and their $H_n(0)$-modules}

We begin with introducing the definition of standard permuted composition tableaux.
In this subsection, $\upsig$ denotes a permutation in $\SG_{\ell(\alpha)}$.

\begin{definition}\label{def: PCT}{\rm (\cite{19TW})}
Given $\alpha \models n$ and $\upsig \in \SG_{\ell(\alpha)}$, a \emph{standard permuted composition tableau $(\SPCT)$ of shape $\alpha$ and type $\upsig$} is a filling $\tau$ of $\tcd(\alpha)$ with entries in $\{1,2,\ldots,n\}$ such that the following conditions hold:
\begin{enumerate}[label = {\rm (\arabic*)}]
\item The entries are all distinct.
\item The standardization of the word obtained by reading the first column from top to bottom is $\upsig$.
\item The entries in each rows decrease from left to right.
\item If $i<j$ and $\tau_{i,k} > \tau_{j,k+1}$, then $(i,k+1) \in \tcd(\alpha)$ and $\tau_{i,k+1} > \tau_{j,k+1}$.
\end{enumerate}
\end{definition}
Denote by $\SPCTsa$ the set of all standard permuted composition tableaux of shape $\alpha$ and type $\upsig$.
Tewari and van Willigenburg define a $0$-Hecke action on $\SPCTsa$ and denote the resulting module by $\bfSsa$.
Contrary to $\calV_\alpha$ and $X_\alpha$, this module is not indecomposable in general.
In the following, we briefly explain how to decompose $\bfSsa$ into indecomposables.

For $\tau, \tau' \in \SPCTsa$, define $\tau \sim \tau'$ if for each positive integer $k$, the relative order of the entries in the $k$th column of $\tau$ is equal to that of $\tau'$.
This relation is an equivalence relation on $\SPCTsa$.
Let $\calEsa$ be the set of all equivalence classes under $\sim$.
Every class in $\calEsa$ is closed under the $0$-Hecke action, which gives rise to the decomposition
\begin{align*}
\bfSsa = \bigoplus_{E \in \calEsa} \bfSsaE,
\end{align*}
where  $\bfSsaE$ is the $H_n(0)$-module spanned by $E$.
All the results in the above can be found in~\cite{19TW}.
The decomposition was improved in~\cite{20CKNO} by showing that every direct summand is indecomposable.

\subsubsection{Weak Bruhat interval module structure of $\bfSsaE$}

We show that $\bfSsaE$ is isomorphic to a weak Bruhat interval module.
To do this, we need a special reading of standard permuted composition tableaux.
Let us introduce the necessary notations and the results.
Given $\tau \in \SPCTsa$, let 
\[
\calD(\tau) := \{1 \le i \le n-1 \mid \text{$i+1$ lies weakly right of $i$ in $\tau$}  \}.
\]
It should be noticed that the set $\calD(\tau)$ plays the same role as the complement of $\Des_L(\gamma)$ since
\[
\pi_i \cdot \gamma \neq \gamma 
\quad \text{if and only if} \quad 
i \notin \Des_L(\gamma)
\quad \text{for $\gamma \in [\sigma, \rho]_L$,}
\]
but
\[
\pi_i \cdot \tau \neq \tau
\quad \text{if and only if} \quad
i \in \calD(\tau)
\quad \text{for $\tau \in \SPCTsa$}. 
\]

Every equivalence class $E$ has a unique source and a unique sink.
Denote the source by $\tauE$ and the sink by $\tau_E'$.
Let $m_E^{~}$ be the number of elements in $\calD(\tauE)$ and set
\[
\calD(\tauE) = \left\{d_1 < d_2 < \ldots < d_{m_E^{~}}\right\}, \quad d_0 := 0, \quad \text{and} \quad d_{m_E^{~}+1} := n.
\]
For $1 \leq j \leq m_E^{~}+1$, let $\tH_j$ be the horizontal strip occupied by the boxes with entries from $d_{j-1}+1$ to $d_j$ in $\tauE$.
For each $\tau \in E$, let $\tau(\tH_j)$ be the subfilling of $\tau$ occupied by $\tH_j$ in $\tcd(\alpha)$.

\begin{definition}\label{def: reading for SPCT}
For $\tau \in E$ and $1 \le j \le m_E^{~} + 1$, let $\rmw^{(j)}(\tau)$ be the word obtained by reading $\tau(\tH_j)$ from left to right.
The reading word, $\tread(\tau)$, of $\tau$ is defined to be the word $\rmw^{(1)}(\tau) \ \rmw^{(2)}(\tau) \  \cdots \  \rmw^{(m_E^{~} + 1)}(\tau)$.
\end{definition}

\begin{example}\label{eg: equiv class E_0}
Let 
$E_0 = \left\{
\begin{array}{c}
\scalebox{0.8}{
\begin{ytableau}
4&3&2 \\
5&1
\end{ytableau}
}
\end{array},
\begin{array}{c}
\scalebox{0.8}{
\begin{ytableau}
4&3&1 \\
5&2
\end{ytableau}
}
\end{array}
\right\} \in \calE^{\id}((3,2))$.
We have
\[
\tau^{~}_{E_0} = \begin{array}{c}
\scalebox{0.8}{
\begin{ytableau}
4&3&2 \\
5&1
\end{ytableau}
}
\end{array},
\quad
\calD(\tau_{E_0}^{~}) = \{1,4\},
\quad \text{and} \quad
\begin{array}{l}
\scalebox{0.65}{
\begin{tikzpicture}
\def \ud {0.65}
\draw[] (0,0) -- (\ud,0) -- (\ud,\ud) -- (0,\ud) -- (0,0);
\draw[] (\ud,0) -- (2*\ud,0) -- (2*\ud,\ud) -- (\ud,\ud) -- (\ud,0);
\draw[] (0,\ud) -- (3*\ud,\ud) -- (3*\ud,2*\ud) -- (0,2*\ud) -- (0,\ud);
\draw[dotted] (\ud,\ud) -- (\ud,2*\ud);
\draw[dotted] (2*\ud,\ud) -- (2*\ud,2*\ud);

\node at (1.5*\ud,0.5*\ud) {\Large $\tH_1$};
\node at (1.5*\ud,1.5*\ud) {\Large $\tH_2$};
\node at (0.5*\ud,0.5*\ud) {\Large $\tH_3$};
\node at (0,2.2*\ud) {};
\end{tikzpicture}
}
\end{array}.
\]
Therefore,
\[
\tread\left(
\begin{array}{c}
\scalebox{0.8}{
\begin{ytableau}
4&3&2 \\
5&1
\end{ytableau}
}
\end{array}
\right) = 14325
\quad \text{and} \quad
\tread\left(
\begin{array}{c}
\scalebox{0.8}{
\begin{ytableau}
4&3&1 \\
5&2
\end{ytableau}
}
\end{array}
\right) = 24315.
\]
\end{example}
Next, we introduce the notion of generalized compositions.
A generalized composition $\bal$ of $n$ is defined to be a formal sum $\beta^{(1)} \oplus \beta^{(2)} \oplus \cdots \oplus  \beta^{(k)}$,
where $\beta^{(i)} \models n_i$ for positive integers $n_i$'s with $n_1 + n_2 + \cdots + n_k = n$.
Given $\bal = \beta^{(1)} \oplus \beta^{(2)} \oplus \cdots \oplus  \beta^{(k)}$, let $\bal^\rmc := (\beta^{(1)})^\rmc \oplus (\beta^{(2)})^\rmc \oplus \cdots \oplus  (\beta^{(k)})^\rmc$ and $w_0(\bal) := w_0(\beta^{(1)} \cdot \beta^{(2)} \cdot \  \cdots \  \cdot \beta^{(k)})$, 
where $\beta^{(i)} \cdot \beta^{(i+1)}$ is the concatenation of $\beta^{(i)}$ and $\beta^{(i+1)}$ for $1 \le i \le k-1$.
Let $\calP_\bal := H_n(0)\pi_{w_0(\bal^\rmc)} \opi_{w_0 (\bal)}$.

Choi, Kim, Nam, and Oh~\cite{20CKNO2} found the projective cover of $\bfSsaE$ by constructing an essential epimorphism $\eta: \autotheta[\bfP_{\bal(\alpha,\sigma;E)^\rmc}] \to \bfSsaE$, where $\bal(\alpha,\sigma;E)$ is a generalized composition defined by using the source of $E$ in a suitable manner and  $\bfP_{\bal(\alpha,\sigma;E)^\rmc}$ is the projective module spanned by standard ribbon tableaux of shape $\bal(\alpha,\sigma;E)^\rmc$.
From now on, we simply write $\bal_E$ for $\bal(\alpha,\sigma;E)$ since $E$ contains information on $\alpha$ and $\sigma$.
For the details, see~\cite[Subsection 2.3 and Section 5]{20CKNO2}.

Now, we are ready to prove the following theorem.

\begin{theorem}\label{thm: bfSsaE pi opi form}
Let $\alpha \models n$ and $\upsig \in \SG_{\ell(\alpha)}$.
For each $E \in \calEsa$, 
\[
\bfSsaE \cong \sfB(\tread(\tauE), \tread(\tau'_E)).
\]
\end{theorem}

\begin{proof}
To prove our assertion, we need the $H_n(0)$-module homomorphisms
\begin{equation*}
\begin{tikzcd}[column sep=large]
\autotheta[\bfP_{\bal_E}] \arrow[r, "\mathsf{w}"] 
&
\calP_{\bal_E^\rmc} 
\end{tikzcd}
\quad \text{and} \quad
\begin{tikzcd}[column sep=large]
\autotheta[\bfP_{\bal_E}] \arrow[r,two heads, "\eta"] 
& \bfSsaE,
\end{tikzcd}
\end{equation*}
where
\begin{enumerate}[label = -, leftmargin = *]
\item the notation $\bfP_{\bal_E}$ denotes the projective module spanned by the generalized ribbon tableaux of shape $\bal_E$,
\item the first homomorphism $\mathsf{w}$ is an isomorphism given in~\cite[Theorem 3.3 and Proposition 5.1]{16Huang}, which is given by reading
standard ribbon tableaux from left to right starting with the bottom row, and
\item the second homomorphism $\eta$ is an essential epimorphism given in~\cite[Theorem 5.3]{20CKNO2}.
\end{enumerate}
Composing $\mathsf{w}^{-1}$ with $\eta$ yields a surjective $H_n(0)$-module homomorphism $\teta: \calP_{\bal_E^\rmc} \ra \bfS^\upsig_{\alpha,E}$.
In view of the definition of $\mathsf{w}$ and $\eta$, one can see that 
\begin{align*}
\teta(\pi_\gamma \opi_{w_0 w_0(\bal_E^\rmc)}) = 
\begin{cases}
\tau & \text{if $\tread(\tau) = \gamma$ for some $\tau \in E$,}\\
0 & \text{otherwise}.
\end{cases}
\end{align*}
Composing $\teta$ with the isomorphism 
\[
\sfem: \sfB(w_0(\bal_E), w_0 w_0(\bal_E^\rmc)) \ra \calP_{\bal_E^\rmc} \quad \text{(in Theorem~\ref{thm: embedding}~(3))}
\]
we finally have the surjective $H_n(0)$-module homomorphism
\[
\teta \circ \sfem: \sfB(w_0(\bal_E), w_0 w_0(\bal_E^\rmc)) \ra \bfS^\upsig_{\alpha,E}.
\]
Next, let us consider the projection 
\begin{align*}
\pr: \sfB(w_0(\bal_E), w_0 w_0(\bal_E^\rmc)) 
& \ra \sfB(\tread(\tauE), \tread(\tau_E')), \\
\gamma \quad 
& \mapsto 
\begin{cases}
\gamma & \text{if $\gamma \in [\tread(\tauE), \tread(\tau_E')]_L$},\\
0 & \text{else.}
\end{cases}
\end{align*}
By the definition of $\teta$, $\tread(\tauE) = w_0(\bal_E)$ and $\tread(\tau_E') \preceq_L w_0 w_0(\bal_E^\rmc)$ and therefore $\pr$ is a surjective $H_n(0)$-module homomorphism.

For our purpose, we have only to show that $\ker(\teta \circ \sfem) = \ker(\pr)$.
From the definition of the $H_n(0)$-action on $\bfS^\upsig_{\alpha,E}$, it follows that $\tread(\tau) \in [\tread(\tauE), \tread(\tau_E')]_L$ for any $\tau \in E$ and therefore $\ker(\teta \circ \sfem) \supseteq \ker(\pr)$.
On the other hand, using the fact that the $0$-Hecke action on $E$ satisfies the braid relations, one can show that every $\gamma \in [\tread(\tauE), \tread(\tau_E')]_L$ appears as $\tread(\tau)$ for some $\tau \in E$.
This says that $\ker(\teta \circ \sfem)^\rmc \supseteq \ker(\pr)^\rmc$, so we are done.
\end{proof}

\begin{remark}\label{rem: SPCT conj}
In case where $\upsig = \id$, a different reading from ours in Definition~\ref{def: reading for SPCT}
has already been introduced in~\cite[Definition 4.1]{15TW}.
More precisely, for each $\tau \in \SPCT^\id(\alpha)$, they define a reading word
$\col_\tau$, called the \emph{column word of $\tau$}.
They also introduce a partial order $\preceq_\alpha$ on $\SPCT^\id(\alpha)$ 
and prove that $(E,\preceq_\alpha)$ is a graded poset isomorphic to $([\col_{\tauE},\col_{\tau_E'}]_L, \preceq_L)$ (see \cite[Theorem 6.18]{15TW}).
In view of Theorem~\ref{thm: Va Xa pi opi form}, one may expect that $\bfS^\id_{\alpha,E}$ is isomorphic to $\sfB(\col_{\tauE}, \col_{\tau_E'})$.
This, however, turns out to be false.
For instance, let $E_0$ be the equivalence class given in Example~\ref{eg: equiv class E_0}.
Then $\col_{\tau^{~}_{E_0}} = 45312$, $ \col_{\tau_{E_0}'} = 45321$,
and $\sfB(45312,45321)$ is not indecomposable as seen in Figure~\ref{fig:equi int not guarantee module iso}.
Therefore, $\bfS^\id_{\alpha,E_0}$ is not isomorphic to $\sfB(45312,45321)$.
\end{remark}

\subsubsection{Involution twists of $\bfS^\upsig_\alpha$} \label{Involution twists; application}
Standard Young row-strict tableaux were first introduced in~\cite{15MN} as a combinatorial model for the Young row-strict quasisymmetric Schur functions $\calR_\alpha$.
Recently, Bardwell and Searles~\cite{20BS} succeeded in constructing an $H_n(0)$-module $\bfR_\alpha$ whose quasisymmetric characteristic image equals $\calR_\alpha$.
It is constructed by defining a $0$-Hecke action on the set of standard Young row-strict tableaux of shape $\alpha$.
We here introduce permuted standard Young row-strict tableaux which turn out to be very useful in describing the $H_n(0)$-module $\hautoomega[\bfSsa]$,
where $\hautoomega=\autophi \circ \autotheta \circ \autochi$.

\begin{definition}\label{def: SPYRT}
Given $\alpha \models n$ and $\upsig \in \SG_{\ell(\alpha)}$, a \emph{standard permuted Young row-strict composition tableau $T$ $($SPYRT$)$ of shape $\alpha$ and type $\upsig$} is a filling of $\tcd(\alpha^\rmr)$ with entries $\{1,2,\ldots, n\}$ such that the following conditions hold:
\begin{enumerate}[label = {\rm (\arabic*)}]
\item The entries are all distinct.
\item The standardization of the word obtained by reading the first column from bottom to top is $\upsig$.
\item The entries in each row are increasing from left to right.
\item If $i<j$ and $T_{i,k} < T_{j,k+1}$, then $(i,k+1) \in \tcd(\alpha^\rmr)$ and $T_{i,k+1} < T_{j,k+1}$.
\end{enumerate}
\end{definition}

Denote by $\SPYRT^\upsig(\alpha)$ the set of all SPYRTs of shape $\alpha$ and type $\upsig$.
Let $\bfR^\upsig_\alpha$ be the $\C$-span of $\SPYRT^\upsig(\alpha)$.
Define 
\begin{align}\label{eq: action on SPYRT}
\pi_i \sq T := \begin{cases}
T & \text{if $i+1$ is weakly left of $i$ in $T$},\\
0 & \text{if $i+1$ is right-adjacent to $i$ in $T$},\\
s_i \cdot T & \text{otherwise}
\end{cases}
\end{align}
for $1\le i \le n-1$ and $T \in \SPYRT^\upsig(\alpha)$.
Here, $s_i \cdot T$ is obtained from $T$ by swapping $i$ and $i+1$.

We claim that \eqref{eq: action on SPYRT} defines an $H_n(0)$-action on $\bfR^\upsig_\alpha$.
For $T \in \SPYRT^\upsig(\alpha)$, let $\tau_T$ be the filling of $\tcd(\alpha^\rmr)$ defined by $(\tau_T)_{i,j} = n+1-T_{i,j}$.
Define a $\C$-linear isomorphism $\calW: \bfR^\upsig_\alpha \ra \hautoomega[\bfS^{\upsig^{w_0}}_{\alpha^\rmr}]$ by letting
\[
\calW(T) = (-1)^{\rank(\tau_T)} \tau_T^{*} \quad \text{for $T \in \SPYRT^\upsig(\alpha)$},
\]
then extending it by linearity.
Here, 
$\tau_T^*$ is the dual of $\tau_T$ with respect to the basis $\SPCT^{\upsig^{w_0}}(\alpha^\rmr)$ for $\bfS^{\upsig^{w_0}}_{\alpha^\rmr}$ and
$\rank(\tau_T) := \min\{\ell(\gamma) \mid \pi_\gamma \cdot \tauE = \tau_T\}$, where $E$ is the equivalence class containing $\tau_T$.
One can verify that 
\[
\calW(\pi_i \sq T) = \pi_i \cdot \calW(T) \quad \text{for all $1\le i \le n-1$},
\]
which proves our claim.
In particular, when $\upsig = \id$, our $\bfR^\upsig_\alpha$ is exactly same to $\bfR_\alpha$ due to Bardwell and Searles.
To summarize, we state the following proposition.

\begin{proposition}\label{prop: bfR and omega twist}
For each $\alpha \models n$ and $\upsig \in \SG_{\ell(\alpha)}$, \eqref{eq: action on SPYRT} defines an $H_n(0)$-action on $\bfR^\upsig_\alpha$.
Moreover, $\calW: \bfR^\upsig_\alpha \ra \hautoomega[\bfS^{\upsig^{w_0}}_{\alpha^\rmr}]$ is an $H_n(0)$-module isomorphism.
\end{proposition}

\begin{remark}
In the combinatorial aspect, our SPYRTs are precisely the standard permuted Young composition tableaux (SPYCT) in~\cite[Definition 4.4]{20CKNO2}.
But, they should be distinguished in the sense that they have different $0$-Hecke actions.
For the $0$-Hecke action on SPYCTs, see~\cite[Subsection 4.2]{20CKNO2}.
The set of SPYCTs is a combinatorial model for an $H_n(0)$-module $\hbfSsa$ which is isomorphic to $\autophi[\bfS^{\upsig^{w_0}}_{\alpha^\rmr}]$.
\end{remark}

By virtue of Proposition~\ref{prop: bfR and omega twist}, one can transport lots of properties of $\bfS^{\sigma^{w_0}}_{\alpha^\rmr}$ to $\bfR^\upsig_\alpha$ via the functor $\bfT^-_{\hautoomega}: \module H_n(0) \ra \module H_n(0)$.
For each $E \in \calE^{\upsig^{w_0}}(\alpha^\rmr)$, let $\bfR^\upsig_{\alpha,E} := \calW^{-1}(\hautoomega[\bfS^{\upsig^{w_0}}_{\alpha^{\rmr},E}])$.
Combining Proposition~\ref{prop: bfR and omega twist} with the results in~\cite{20CKNO,20CKNO2,19TW}, we have the following corollary.

\begin{corollary} 
For each $E \in \calE^{\upsig^{w_0}}(\alpha^\rmr)$, the following hold:
\begin{enumerate}[label = {\rm (\arabic*)}]
\item $\bfR^\upsig_{\alpha,E}$ is indecomposable. In particular, $\bfR^\upsig_\alpha = \bigoplus_{E \in \calEsa} \bfR^\upsig_{\alpha,E}$ is a decomposition of $\bfR^\upsig_\alpha$ into indecomposables.
\item The injective hull of $\bfR^\upsig_{\alpha,E}$ is
$\calP_{\bal_E} (= \calP_{\bal(\alpha^{\rmr}, \upsig^{w_0}; E)})$.
\item $\bfR^\upsig_{\alpha,E} \cong \sfB(w_0\tread(\tau_E'), w_0\tread(\tauE))$.
\end{enumerate}
\end{corollary}

\begin{proof}
(1) Since $\bfT^-_{\hautoomega}: \module H_n(0) \ra \module H_n(0)$ is an isomorphism, it preserves direct sum.
Therefore, the assertion can be obtained by combining Proposition~\ref{prop: bfR and omega twist} with~\cite[Theorem 3.1]{20CKNO}.

(2) As seen in the proof of Theorem~\ref{thm: bfSsaE pi opi form}, 
$\teta: \calP_{\bal_E^\rmc} \to \bfS^{\upsig^{w_0}}_{\alpha^\rmr,E}$ is an essential epimorphism,
thus $\calP_{\bal_E^\rmc}$ is the projective cover of $\bfS^{\upsig^{w_0}}_{\alpha^\rmr,E}$.
It is clear that $\bfT^-_{\hautoomega}$ is a contravariant exact functor.
So, taking $\bfT^-_{\hautoomega}$ on $\teta$ yields that  
$\hautoomega[\calP_{\bal_E^\rmc}]$ is the injective hull of $\hautoomega[\bfS^{\upsig^{w_0}}_{\alpha^\rmr,E}] \cong \bfR^{\upsig}_{\alpha,E}$.
Now the assertion follows from the isomorphism $\hautoomega[\calP_{\bal_E^\rmc}] \cong \calP_{\bal_E}$, which is due to Table~\ref{tab: auto twist of F and P}.

(3) The assertion follows from Theorem~\ref{thm: auto twist part 1} with Theorem~\ref{thm: bfSsaE pi opi form}.
\end{proof}

The three commutative diagrams in Figure~\ref{fig: commutative diagrams for twist of S} show 
various (anti-)involution twists of $\bfSsa$ as well as their images under the quasisymmetric characteristic when $\upsig = \id$.
\begin{figure}[t]
\centering
\begin{tikzpicture}
\def \sc {1.5}
\def \depth{1.2}
\def \hp{-5.5}
\def \hhp{-12}

\node[] at (0*\sc + 0.5*\depth,0*\sc + 0.5*\depth) {$\calS_\alpha$};
\node[] at (0*\sc + 0.5*\depth,1.5*\sc + 0.5*\depth) {$\mathcal{RS}_\alpha$};
\node[] at (1.7*\sc + 0.5*\depth,0 + 0.5*\depth) {$\hcalS_{\alpha^\rmr}$};
\node[] at (1.7*\sc + 0.5*\depth,1.5*\sc + 0.5*\depth) {$\calR_{\alpha^\rmr}$};

\node[] at (0.5*\depth - 0.15*\sc,0.75*\sc + 0.5*\depth) {\scriptsize $\uppsi$};
\draw [<->] (0*\sc + 0.5*\depth,0.3*\sc + 0.5*\depth) -- (0 + 0.5*\depth,1.2*\sc + 0.5*\depth);

\draw [<->] (1.7*\sc + 0.5*\depth,0.3*\sc + 0.5*\depth) -- (1.7*\sc + 0.5*\depth,1.2*\sc + 0.5*\depth);
\draw [<->] (0.4*\sc + 0.5*\depth,1.5*\sc + 0.5*\depth) -- (1.4*\sc + 0.5*\depth,1.5*\sc + 0.5*\depth);
\draw [<->] (0.3*\sc + 0.5*\depth,0*\sc + 0.5*\depth) -- (1.4*\sc + 0.5*\depth,0*\sc + 0.5*\depth);
\node[] at (0.85*\sc + 0.5*\depth,0.15*\sc + 0.5*\depth) {\scriptsize $\uprho$};
\node[] at (3*\depth - 0.15*\sc,0.75*\sc + 0.5*\depth) {\scriptsize $\uppsi$};
\node[] at (0.85*\sc + 0.5*\depth,1.65*\sc + 0.5*\depth) {\scriptsize $\uprho$};

\draw [->,thick] (-1, 0.75*\sc + 0.5*\depth) -- (0, 0.75*\sc + 0.5*\depth);
\node[] at (-0.5,0.95*\sc + 0.5*\depth) {\scriptsize $\ch$};

\node[] at (0*\sc+\hp,0*\sc) {$\bfS_\alpha$};
\node[] at (0*\sc+\hp,1.5*\sc) {$\autotheta[\bfS_\alpha]$};
\node[] at (1.7*\sc+\hp,0*\sc) {$\hbfS_{\alpha^\rmr}$};
\node[] at (1.7*\sc+\hp,1.5*\sc) {$\autoomega[\bfS_\alpha]$};
\draw [<->] (0*\sc+\hp,0.3*\sc) -- (0*\sc+\hp,1.2*\sc);
\draw [<->] (1.7*\sc+\hp,0.3*\sc) -- (1.7*\sc+\hp,1.2*\sc);
\draw [<->] (0.4*\sc+\hp,1.5*\sc) -- (1.3*\sc+\hp,1.5*\sc);
\draw [<->] (0.3*\sc+\hp, 0*\sc) -- (1.4*\sc+\hp, 0*\sc);

\node[] at (0*\sc+\hp+\depth,0*\sc+\depth) {$\autochi[\bfS_\alpha]$};
\node[] at (0*\sc+\hp+\depth,1.5*\sc+\depth) {$\hautotheta[\bfS_\alpha]$};
\node[] at (1.7*\sc+\hp+\depth,0*\sc+\depth) {$\hautophi[\bfS_\alpha]$};
\node[] at (1.7*\sc+\hp+\depth,1.5*\sc+\depth) {$\bfR_{\alpha^\rmr}$};
\draw [<->] (0*\sc+\hp+\depth,0.3*\sc+\depth) -- (0*\sc+\hp+\depth, 1.2*\sc+\depth);
\draw [<->] (1.7*\sc+\hp+\depth,0.3*\sc+\depth) -- (1.7*\sc+\hp+\depth, 1.2*\sc+\depth);
\draw [<->] (0.4*\sc+\hp+\depth,1.5*\sc+\depth) -- (1.4*\sc+\hp+\depth, 1.5*\sc+\depth);
\draw [<->] (0.4*\sc+\hp+\depth,0*\sc+\depth) -- (1.2*\sc+\hp+\depth, 0*\sc+\depth);

\draw [<->, red] (0*\sc+\hp + 0.25*\depth, 0*\sc  + 0.25*\depth) -- (0*\sc+\hp  + 0.75*\depth,0*\sc  + 0.75*\depth);
\draw [<->, red] (0*\sc+\hp + 0.25*\depth, 1.5*\sc  + 0.25*\depth) -- (0*\sc+\hp  + 0.75*\depth,1.5*\sc  + 0.75*\depth);
\draw [<->, red] (1.7*\sc+\hp + 0.25*\depth, 0*\sc  + 0.25*\depth) -- (1.7*\sc+\hp  + 0.75*\depth,0*\sc  + 0.75*\depth);
\draw [<->, red] (1.7*\sc+\hp + 0.25*\depth, 1.5*\sc  + 0.25*\depth) -- (1.7*\sc+\hp  + 0.75*\depth,1.5*\sc  + 0.75*\depth);

\draw [->,thick] (-1.8 + \hp, 0.75*\sc + 0.5*\depth) -- (-0.5 + \hp, 0.75*\sc + 0.5*\depth);
\node[] at (-1.2 + \hp,0.95*\sc + 0.5*\depth) {\scriptsize $\upsig = \id$};

\node[] at (0*\sc+\hhp,0*\sc) {$\bfS^{\upsig}_\alpha$};
\node[] at (0*\sc+\hhp,1.5*\sc) {$\autotheta[\bfS^{\upsig}_{\alpha}]$};
\node[] at (1.7*\sc+\hhp,0*\sc) {
$\hbfS^{\upsig^{w_0}}_{\alpha^{\rmr}}$};
\node[] at (1.7*\sc+\hhp,1.5*\sc) {$\autoomega[\bfS^{\upsig}_{\alpha}]$};
%middle
\node[] at (\hhp - 0.2*\sc - \hhp +\hp,0.75*\sc) {\tiny $\bfT^+_\autotheta$};

\node[] at (\hhp - 0.2*\sc,0.75*\sc) {\tiny $\bfT^+_\autotheta$};
\draw [<->] (0*\sc+\hhp,0.3*\sc) -- (0*\sc+\hhp,1.2*\sc);

\draw [<->] (1.7*\sc+\hhp,0.3*\sc) -- (1.7*\sc+\hhp,1.2*\sc);
\draw [<->] (0.4*\sc+\hhp,1.5*\sc) -- (1.3*\sc+\hhp,1.5*\sc);
%middle
\node[] at (0.85*\sc+\hhp - \hhp+\hp,0.15*\sc) {\tiny $\bfT^+_\autophi$};

\node[] at (0.85*\sc+\hhp,1.65*\sc-1.5*\sc) {\tiny $\bfT^+_\autophi$};
\draw [<->] (0.3*\sc+\hhp,0*\sc) -- (1.4*\sc+\hhp,0*\sc);

\node[] at (0*\sc+\hhp+\depth,0*\sc+\depth) {$\autochi[\bfS^{\upsig}_\alpha]$};
\node[] at (0*\sc+\hhp+\depth,1.5*\sc+\depth) {
$\hautotheta[\bfSsa]$};
\node[] at (1.7*\sc+\hhp+\depth,0*\sc+\depth) {
$\hautophi[\bfS^{\upsig}_\alpha]$};
\node[] at (1.7*\sc+\hhp+\depth,1.5*\sc+\depth) {
$\bfR^{\upsig^{w_0}}_{\alpha^\rmr}$};
\draw [<->] (0*\sc+\hhp+\depth,0.3*\sc+\depth) -- (0*\sc+\hhp+\depth, 1.2*\sc+\depth);

\draw [<->] (1.7*\sc+\hhp+\depth,0.3*\sc+\depth) -- (1.7*\sc+\hhp+\depth, 1.2*\sc+\depth);
\draw [<->] (0.4*\sc+\hhp+\depth,1.5*\sc+\depth) -- (1.3*\sc+\hhp+\depth, 1.5*\sc+\depth);

\draw [<->] (0.5*\sc+\hhp+\depth,0*\sc+\depth) -- (1.2*\sc+\hhp+\depth, 0*\sc+\depth);
\draw [<->,red] (0*\sc+\hhp + 0.25*\depth, 0*\sc  + 0.25*\depth) -- (0*\sc+\hhp  + 0.75*\depth,0*\sc  + 0.75*\depth);
%middle
\node[] at (-0.1*\sc+\hhp+0.5*\depth - \hhp +\hp, 0.2*\sc+0.5*\depth) {\tiny $\bfT^-_\autochi$};

\node[] at (-0.1*\sc+\hhp+0.5*\depth, 0.15*\sc+0.5*\depth) {\tiny $\bfT^-_\autochi$};
\draw [<->, red] (0*\sc+\hhp + 0.25*\depth, 1.5*\sc  + 0.25*\depth) -- (0*\sc+\hhp  + 0.75*\depth,1.5*\sc  + 0.75*\depth);

\draw [<->, red] (1.7*\sc+\hhp + 0.25*\depth, 0*\sc  + 0.25*\depth) -- (1.7*\sc+\hhp  + 0.75*\depth,0*\sc  + 0.75*\depth);
\draw [<->, red] (1.7*\sc+\hhp + 0.25*\depth, 1.5*\sc  + 0.25*\depth) -- (1.7*\sc+\hhp  + 0.75*\depth,1.5*\sc  + 0.75*\depth);
\end{tikzpicture}
\caption{(Anti-)involution twists of $\bfSsa$ and their images under the quasisymmetric characteristic when $\upsig=\id$}
\label{fig: commutative diagrams for twist of S}
\end{figure}
In the first and second diagram, 
the functors assigned to parallel arrows are all same
and the arrows in red are being used to indicate that the domain and codomain have the same image under the quasisymmetric characteristic.
In the last diagram, 
$\hcalS_{\alpha^\rmr}$ is the \emph{Young quasisymmetric Schur function} in~\cite[Definition 5.2.1]{13LMvW},
$\mathcal{RS}_\alpha$ is the \emph{row-strict quasisymmetric Schur function} in~\cite[Definition 3.2]{14MR}, 
and $\uppsi, \uprho$ are automorphisms of $\Qsym$ defined by $\uppsi(F_\alpha) = F_{\alpha^\rmc}$ and $\uprho(F_\alpha) = F_{\alpha^\rmr}$. 

\begin{remark}\label{rem: omega twist for S}
(1) Let $k$ be a positive integer and $\omega := \uppsi \circ \uprho$. 
It was stated in~\cite[Theorem 5.1]{14MR} that 
\[
\omega(\calS_\alpha(x_1,x_2, \ldots, x_k)) = \calRS_{\alpha}(x_k,x_{k-1},\ldots, x_1).
\]
On the other hand, the third diagram in Figure~\ref{fig: commutative diagrams for twist of S} 
shows that 
$\omega(\calS_\alpha) = \calR_{\alpha^\rmr}$, thus
\[
\omega(\calS_\alpha(x_1,x_2, \ldots, x_k)) = \calR_{\alpha^\rmr}(x_1, x_2, \ldots, x_k).
\]
As a consequence, we derive that   
\[
\calRS_{\alpha}(x_k,x_{k-1},\ldots, x_1)=\calR_{\alpha^\rmr}(x_1, x_2, \ldots, x_k).
\]
We add a remark that in some literature such as~\cite[Subsection 5.2]{13LMvW} and \cite[Remark 4.4]{18TW},
the identity $\omega(\calS_\alpha) = \calR_{\alpha^\rmr}$ is incorrectly stated as  
$\omega(\calS_\alpha) = \calRS_{\alpha}$.
\medskip

\noindent
(2) One can also observe $\omega(\hcalS_\alpha) = \calR_{\alpha}$ in~\cite[Theorem 12]{15MN}. 
This identity, however, should appear as $\omega(\hcalS_\alpha) = \calRS_{\alpha^\rmr}$ by the third diagram in 
Figure~\ref{fig: commutative diagrams for twist of S}. 
Within the best understanding of the authors, this error seems to have occurred for the reason that 
the descent sets of Young composition tableaux and that of standard Young row-strict composition tableaux are defined in a different manner.
The proof of~\cite[Theorem 12]{15MN}, with a small modification, 
can be used to verify $\omega(\hcalS_\alpha) = \calRS_{\alpha^\rmr}$.
\end{remark}

\section{Further avenues}\label{sec: Further remark}
\noindent
(1) We have studied the structure of weak Bruhat interval modules so far.
However, there are still many unsolved fundamental problems including the following:
\begin{enumerate}[label = -]
\item 
Classify all weak Bruhat interval modules up to isomorphism.
\item 
Given an interval $[\sigma,\rho]_L$, decompose $\sfB(\sigma,\rho)$ into indecomposables.
\item
Given an interval $[\sigma,\rho]_L$, find the projective cover and the injective hull of $\sfB(\sigma,\rho)$.
\end{enumerate}
\medskip

\noindent
(2)
In~\cite[Section 9 and 10]{15TW}, Tewari and van Willigenburg provide
a restriction rule for $\bfS_{\alpha}\downarrow^{H_n(0)}_{H_{n-1}(0)}$
and ask if there is a reciprocal induction rule for  $\bfS_{\alpha}\uparrow^{H_n(0)}_{H_{n-1}(0)}$ with respect to the restriction rule. 
By combining Lemma~\ref{lem: tensor product} with Theorem~\ref{thm: bfSsaE pi opi form}, we successfully decompose $\bfS_{\alpha}\uparrow^{H_n(0)}_{H_{n-1}(0)}$ into 
weak Bruhat interval modules.
But, at the moment, we do not know if it can be expressed as a direct sum of $\bfS_{\beta}$'s.
We expect that a better understanding of the weak Bruhat interval modules appearing in the decomposition would be of great help in solving this problem.
In line with this philosophy, it is interesting to  
find or characterize all intervals $[\sigma,\rho]_L$ such that $\sfB(\sigma, \rho)$ is isomorphic to $\calV_\alpha$, $X_\alpha$, or $\bfS^{\upsig}_{\alpha,E}$.
\medskip

\noindent
(3) 
Let $n$ and $N$ be arbitrary positive integers and $V:= \mathbb{C}^N$.
Using the fact that the left $\mathcal{U}_0(gl_N)$-action on $V^{\otimes n}$ commutes with the right $H_n(0)$-action on $V^{\otimes n}$, Krob and Thibon~\cite{99KT} construct $\mathcal{U}_0(gl_N)$-modules 
\begin{align*}
&\mathbf{D}_\alpha := V^{\otimes n} \cdot \opi_{w_0(\alpha^\rmc) \, w_0} \pi_{w_0(\alpha^\rmt)}
\quad \text{and}\\ 
&\mathbf{N}_\alpha := V^{\otimes n} \cdot \opi_{w_0(\alpha)} \pi_{w_0(\alpha^\rmc)}
\end{align*}
for every composition $\alpha$ of $n$.
Then they prove that, as $\alpha$ ranges over the set of nonempty compositions,
$\mathbf{D}_\alpha$'s form a complete family of irreducible polynomial $\mathcal{U}_0(gl_N)$-modules
and $\mathbf{N}_\alpha$'s a complete family of indecomposable polynomial $\mathcal{U}_0(gl_N)$-modules 
which arise as a direct summand of $V^{\otimes n}$ for some $n>0$.
They also realize $\bfF_\alpha$ and $\bfP_\alpha$ as the left ideals of $H_n(0)$
\begin{align}\label{eq: bfF and bfP in KT}
\begin{aligned}
\bfF_\alpha 
& \cong H_n(0) \cdot \opi_{w_0(\alpha^\rmc) \, w_0} \pi_{w_0(\alpha^\rmt)}
\quad \text{and}\\ 
\bfP_\alpha 
& \cong H_n(0) \cdot \opi_{w_0(\alpha)} \pi_{w_0(\alpha^\rmc)}.
\end{aligned}
\end{align}
Hence, by replacing $H_n(0)$ by $V^{\otimes n}$ in~\eqref{eq: bfF and bfP in KT}, one obtains 
the $\mathcal{U}_0(gl_N)$-modules $\mathbf{D}_\alpha$ and $\mathbf{N}_\alpha$ from the $H_n(0)$-modules $\bfF_\alpha$ and $\bfP_\alpha$, respectively.
This relationship seems to work well at the character level as well.
In this regard, Hivert~\cite{00Hivert} shows that the Weyl character of $\mathbf{D}_\alpha$ is equal to the quasisymmetric polynomial 
$F_\alpha(x_1, x_2, \ldots, x_N, 0,0,\ldots)$, where $F_\alpha(x_1, x_2,\ldots) = \ch([\bfF_\alpha])$.

In the present paper, we study intensively weak Bruhat interval modules, which are 
of the form $H_n(0)\pi_{\sigma} \opi_{\rho}$ or $H_n(0)\opi_{\sigma} \pi_{\rho}$ up to isomorphism (Theorem~\ref{thm: embedding}).
Hence, it would be very meaningful to investigate how our results about weak Bruhat interval modules are reflected on  
the corresponding $\mathcal{U}_0(gl_N)$-modules, in other words, 
the $\mathcal{U}_0(gl_N)$-modules of the form 
$V^{\otimes n} \cdot \pi_\sigma \opi_\rho$ and $V^{\otimes n} \cdot \opi_\sigma \pi_\rho$ for $\sigma, \rho \in \SG_n$.
\medskip

\noindent {\bf Acknowledgments.}
The authors would like to thank Sarah Mason and Elizabeth Niese for helpful discussions on Remark~\ref{rem: omega twist for S}.
The authors also would like to thank Dominic Searles for helpful discussions on the $0$-Hecke action on $\bfR^\sigma_\alpha$.
The authors are grateful to the anonymous referee for careful readings of the manuscript and valuable advice.

\bibliographystyle{abbrv}
\bibliography{references}

\end{document}